\documentclass[12pt]{article}

\title{\textbf{Numerical integration for fractal measures}}
\author{By JENS MALMQUIST\\
    \textit{Mathematics Department, Evans Hall, University of California,}\\
    \textit{Berkeley, CA 94720}\\
    \texttt{jmalmquist@berkeley.edu}
    \and
    and ROBERT S. STRICHARTZ\\
    \textit{Mathematics Department, Malott Hall, Cornell University,}\\
    \textit{Ithaca, NY 14853}\\
    \texttt{str@math.cornell.edu}}
\date{}

\usepackage{amsmath,amssymb,amsthm, xfrac, graphicx, tabu, tikz, framed, caption, subcaption, verbatim, url}
\usepackage[normalem]{ulem}
\usepackage{longtable}
\usepackage[outline]{contour}
\contourlength{1.2pt}

\numberwithin{equation}{section}
\numberwithin{figure}{section}

\newtheorem{theorem}{Theorem}[section]
\newtheorem{lemma}[theorem]{Lemma}

\newtheorem{corollary}[theorem]{Corollary}
\newtheorem{definition}[theorem]{Definition}
\newtheorem{example}{Example}[section]

\newcommand{\energy}{\mathcal{E}}
\newcommand{\harmonic}{\mathcal{H}}

\newcommand{\laplace}{\bigtriangleup}
\newcommand{\lap}{\bigtriangleup}
\newcommand{\dom}{\mathop{\mathrm{dom}}\nolimits}

\newcommand{\disc}{\mathop{\mathrm{disc}}}
\newcommand{\Var}{\mathop{\mathrm{Var}}}
\usepackage[english]{babel}
\usepackage[utf8]{inputenc}
\usepackage[nottoc]{tocbibind}

\begin{document}

\maketitle
\begin{abstract}
    We find estimates for the error in replacing an integral $\int f d\mu$ with respect to a fractal measure $\mu$ with a discrete sum $\sum_{x \in E} w(x) f(x)$ over a given sample set $E$ with weights $w$. Our model is the classical Koksma-Hlawka theorem for integrals over rectangles, where the error is estimated by a product of a \textit{discrepancy} that depends only on the geometry of the sample set and weights, and \textit{variance} that depends only on the smoothness of $f$. We deal with p.c.f self-similar fractals, on which Kigami has constructed notions of \textit{energy} and \textit{Laplacian}. We develop generic results where we take the variance to be either the energy of $f$ or the $L^1$ norm of $\laplace f$, and we show how to find the corresponding discrepancies for each variance. We work out the details for a number of interesting examples of sample sets for the Sierpi\'{n}ski gasket, both for the standard self-similar measure and energy measures, and for other fractals.
\end{abstract}\begin{centering}\rule{6.5in}{0.4pt}\end{centering}

Research of the first author was supported by the National Science Foundation through the Research Experiences for Undergraduates Program at Cornell, grant DMS-1156350.

Research of the second author was supported in part by the National Science Foundation, grant DMS-1162045.

2010 Mathematics Subject Classification 28A80, 65D30

Keywords: Numerical integration, fractal measures, p.c.f. self-similar fractals, energy measures, Laplacians, Koksma-Hlawka theorem, Sierpi\'{n}ski gasket %variance, discrepancy,

\pagebreak

\tableofcontents

\section{Introduction}

Numerical integration on domains in Euclidean space is a highly developed subject that is of interest from both a theoretical and practical point of view, with many open problems still being actively pursued (\cite{Ko}, \cite{Nie}). The goal of this paper is to develop a similar theory on fractals, following up on earlier work in \cite{energy}.

The gist of the matter, in any context, may be succinctly stated as follows. Given a measure $\mu$ on some space and a finite set of points $E$, the \textit{sample set}, we would like to approximate the integral $\int f d\mu$ by the sum $\sum_{x \in E} w(x)f(x) $ for a set of weights $\left\{w(x)\right\}$. The main problem is to estimate the error of the approximation. A desirable form of the error estimate is in terms of a product of two factors, a \textit{discrepancy} of the weights (or just of the set $E$ if the weights are chosen uniformly $w(x)=\frac{1}{\#E}$) depending on the \textquotedblleft geometry" of $E$, and a \textit{variance} of $f$ that measures the \textquotedblleft smoothness" of $f$ in a suitable norm. A well known version of such an estimate in the case of rectangles is the Koksma-Hlawka theorem (\cite{Nie}), and some of our results are modeled on this theorem. Other interesting questions concern how to choose the sample set $E$ to minimize the discrepancy, and how to choose \textquotedblleft natural" weights on $E$.

We will restrict attention to Kigami's class of \textit{p.c.f. self-similar fractals} with a \textit{regular harmonic structure} \cite{kigami}. A basic example is the Sierpi\'{n}ski gasket $SG$ (see \cite{S} for a detailed description of this example) and we will give the most detailed results for this example. We hope that our results will serve as a foundation for future work on products of fractals, motivated by the observation that rectangles are products of intervals, and intervals are in fact the most elementary examples of p.c.f. self-similar sets.

There are two types of measures that are natural to consider in this context. The first are the \textit{self-similar measures} that are naturally associated with the self-similar structure of the fractal, and include the normalized Hausdorff measure in the appropriate Hausdorff dimension. The second are the energy measures associated with the harmonic structure. Very briefly, the harmonic structure provides an energy $\energy(f, g)$, a bilinear Dirichlet form analogous to the energy $\int_{\Omega} (\nabla f \cdot \nabla g) dx$ on a domain $\Omega$ in Euclidean space. \textit{Harmonic functions} are energy $\energy(h, h)$ minimizers, analogous to linear functions on an interval.

The energy measure $\nu_{h, H}$ for harmonic functions $h$ and $H$ assigns to a set $C$ the \textquotedblleft restriction" of $\energy_{h, H}$ to $C$. An interesting and surprising result of Kusuoka \cite{kusuoka} shows that energy measures and self-similar measures are mutually singular, in start contrast to what happens in classical analysis. Associated to each measure $\mu$ is a \textit{Laplacian} $\laplace_{\mu}$. The study of Laplacians for self-similar measures was originally the focus of the theory of analysis on these fractals, but recently energy measure Laplacians have come to the fore (\cite{ki2}, \cite{ka1}, \cite{StTs}, \cite{energy}). For this reason, it is worth investigating numerical integration for both types of measures.

We will consider two types of smoothness conditions. The first is a very minimal smoothness that $\energy(f, f)$ is finite. This implies that $f$ is continuous in our context (but not in Euclidean space of dimension above one). The second is the finiteness of $|| \laplace_{\mu}f ||_1$. There will be a different discrepancy associated to each of these variances of $f$, with the second one typically a lot smaller because we are assuming more smoothness for the function. We will have two \textquotedblleft generic results" corresponding to these choices. We note that our results are not exactly analogs of Koksma-Hlawka; they are only similar in spirit.

For each sample set $E$ we will typically investigate a \textquotedblleft natural" set of weights $\left\{ p(x) \right\}$. These weights will allow the exact evaluation $\int f(x) d\mu(x) = \sum_{x \in E} p(x) f(x) $ for a finite dimensional space of functions called \textit{piecewise harmonic splines}. These are basically the continuous functions that are harmonic on the complement of $E$, and are the exact analog of piecewise linear functions on an interval. So then it is natural to estimate the error for a general set of weights $\left\{ w(x) \right\}$ in terms of the differences between the two sets of weights, using the approximation properties of the piecewise harmonic splines in terms of the smoothness norms of $f$.

We develop our generic results in section 2. Then in section 3 we study the example of $SG$ and the standard self-similar measure $\mu$, and work out in detail the natural weights and discrepancies for a variety of sample sets. In section 4 we briefly examine some other p.c.f. fractals. In section 5 we return to $SG$ but consider energy measures. See \cite{LRSU} for related work concerning values of smooth functions on discrete sets of points. The programs used to generate the data in sections 4 and 5 may be found at the website \cite{website}

\section{Generic results} \label{generic}

Let $K$ be a p.c.f. self-similar fractal generated by a finite iterated system $\left\{F_{j}\right\}$ of contractive similarity on some ambient Euclidean space. So

\begin{equation*}
K=\bigcup_{i}^{} F_{i}K
\end{equation*}
and there exists a finite set $V_{0}$ of \textit{boundary points} such that

\begin{equation*}
F_{i}K \cap F_{j}K \subseteq F_{i}V_{0}  \cap F_{j}V_{0}.
\end{equation*}
We assume there is a self-similar energy form $\energy(u)$ on $K$ such that 

\begin{equation*}
\energy(u)=\sum_{i} \frac{1}{r_{i}} \energy(u \circ F_{i} )
\end{equation*}
for some energy renormalization constants $0 < r_{i} < 1$. See \cite{kigami} for detailed definitions.

Let $\mu$ be a probability measure on $K$ that is non-atomic and assigns positive values to nonempty open sets. Let $E$ be a finite subset of $K$, and suppose we are given a set of positive weights $w(x)$ on $E$ with

\begin{equation*}
\sum_{x \in E} w(x) = 1.
\end{equation*}
Our goal is to understand how well the discrete sum $\sum_{x \in E} w(x)f(x)$ approximates the integral $\int_{K} f d\mu$ under various \textquotedblleft smoothness" assumptions on $f$. We want estimates of the form

\begin{equation*}
\bigg| \int f d\mu - \sum_{x \in E} w(x)f(x) \bigg| \leq \disc(E, w)\Var(f)
\end{equation*}
where the discrepancy $\disc(E, w)$ is some \textquotedblleft geometric" measurement of the distance between the original measure $\mu$ and the approximate measure $\sum_{x \in E}^{} w(x) \delta_{x}$, and $\Var(f)$ is some norm measuring the smoothness of $f$. The classical Koksma-Hlawka theorem is a model example of such an estimate.

Our approach to obtaining such estimates is to consider two separate subproblems. The first is to obtain estimates of $\int f d\mu$ under the assumption that $f|_{E}=0$. The second is to consider a family of splines defined in terms of $E$ and to find a family of weights $\left\{p(x)\right\}$ such that $\int g d\mu=\sum_{x \in E} p(x)g(x)$ for every spline $g$. Given a suitably smooth $f$, we write $f=(f-g)+g$ where $g$ is a spline satisfying $g|_{E}=f|_{E}$. We use the first subproblem to handle $f-g$ and the second subproblem to handle $g$, and then add.

Associated to the energy $\energy$ and the measure $\mu$ we have a Laplacian $\laplace_{\mu}$ defined by the weak formulation

\begin{equation} \label{laplacian-formulation}
\int_{K} v \laplace_{\mu} u d\mu = -\energy(u, v)
\end{equation}
for all test functions $v \in \dom \energy$ ($\dom\energy$ is the set of functions with $\energy(v) < \infty$, and $\energy(u, v)$ is the associated bilinear form).

[Note that this definition actually gives the Neumann Laplacian with vanishing normal derivatives at boundary points. In the case that $V_{0} \subseteq E$ we could just as well restrict \eqref{laplacian-formulation} to hold for just test functions $v$ vanishing on $V_{0}$.]

We define $\dom \laplace_{\mu}$ to be the space of functions $u$ where $\laplace_{\mu}u$ is continuous, and $\dom_{L^{1}}\laplace_{\mu}$ to be the larger space where $\laplace_{\mu}u \in L^{1}(d\mu)$ with seminorm

\begin{equation*}
||u||_{\dom_{L^{1}}\laplace_{\mu}}=\int |\laplace_{\mu}u|d\mu.
\end{equation*}

Associated with the set $E$ we have the Green's function $G_{E}(x, y)$ that gives the inverse of $-\laplace_{\mu}$ subject to Dirichlet boundary conditions on $E$. That means

\begin{equation} \label{eq:2.8}
F(x)=\int_{K} G_{E}(x, y) f(y) d\mu(y)
\end{equation}
gives the unique solution to

\begin{equation*}
-\laplace_{\mu}F=f, \quad F|_{E}=0.
\end{equation*}
Note that, in particular, the function

\begin{equation*}
g_{E}(x)=\int_{K} G_{E}(x, y) d\mu(y)
\end{equation*}
is the solution of

\begin{equation*}
-\laplace_{\mu} g_{E}=1, \quad g_{E} |_{E}=0.
\end{equation*}
Also, $G_{E}$ is symmetric under interchange of $x$ and $y$. Another useful expression for $G_{E}$ is

\begin{equation} \label{GEvarphi}
G_{E}(x, y)=\sum_{j} \frac{1}{\lambda_{j}} \varphi_{j}(x) \varphi_{j}(y)
\end{equation}
where $\left\{ \varphi_{j} \right\}$ is an orthonormal basis of Dirichlet eigenfunctions

\begin{equation*}
-\laplace_{\mu} \varphi_{j} = \lambda_{j} \varphi_{j}, \quad \varphi_{j}|_{E}=0.
\end{equation*}
Note that while the individual terms in \eqref{GEvarphi} depend on $\mu$, in fact $G_{E}$ depends only on $E$ and not $\mu$.

\begin{definition} \label{def:2.1}
Let 
\begin{equation*}
\delta_{0}(E)  = \left( \int_{K} \int_{K} G_{E}(x, y) d\mu(y) d\mu(x) \right)^{\sfrac{1}{2}}  = \left( \int_K g_{E}(x) d\mu(x) \right)^{\sfrac{1}{2}}
\end{equation*}
and

\begin{equation*}
\delta_{1}(E)=\sup_{x} g_{E}(x).
\end{equation*}
\end{definition}

\begin{theorem} \label{thm:2.2}
Suppose $u \in \dom\energy$ and $u|_E=0$. Then

\begin{equation} \label{thm2.2eq}
\left| \int u d\mu \right| \leq \delta_{0}(E) \energy(u) ^{\sfrac{1}{2}}
\end{equation}
\end{theorem}

\begin{proof}
Write $u=\sum c_{j} \varphi_{j}$ for $c_j=\int u\varphi_j d\mu$.
Since $\energy\left(\varphi_j, \varphi_k\right)=\int \left( - \laplace_{\mu}\varphi_j \right) \varphi_k d\mu = \lambda_j \int \varphi_j \varphi_k d\mu$, we have $\energy(\varphi_j, \varphi_k)=0$ if $j \neq k$ and $\energy(\varphi_j, \varphi_j)=\lambda_j$. So

\begin{equation} \label{thm2.2proof}
\energy(u) = \sum_j \lambda_j |c_j|^2 \mbox{,}
\end{equation}
and by Cauchy-Schwarz

\begin{align*}
\left| \int u d\mu \right| & = \left| \sum_{j} c_{j} \int \varphi_{j} d\mu \right|  \\
& \leq \left( \sum_{j} \lambda_j \left| c_j \right|^2 \right)^{\sfrac{1}{2}} \left( \sum_j \frac{1}{\lambda_{j}} \left| \int \varphi_j d\mu \right|^{2} \right)^{\sfrac{1}{2}}.
\end{align*}
But by \eqref{GEvarphi}, $ \sum_{j} \frac{1}{\lambda_{j}} \left| \int \varphi_{j} d\mu \right|^{2} = \int\int G_{E}(x, y) d\mu(y) d\mu(x) $, and combined with \eqref{thm2.2proof} this yields \eqref{thm2.2eq}.

\end{proof}

\begin{theorem} \label{thm:2.3}
Suppose $u \in \dom _{L^1}\laplace_{\mu}$ and $u|_E=0$. Then

\begin{equation} \label{thm2.3eq}
\left| \int u d\mu \right| \leq \delta_{1} (E) \int \left| \laplace_{\mu} u \right| d\mu.
\end{equation}

\end{theorem}

\begin{proof}
 Let $f=\-\laplace_{\mu} u$, so $f \in L^1(d\mu)$ and \eqref{eq:2.8} holds (with $F=u$). Then

\begin{align*}
\left|\int u d\mu \right| & = \left| \int \int G_{E}(x, y) f(y) d\mu(x) d\mu(y) \right| \\
& = \left| \int g_{E}(y) f(y) d\mu(y) \right| \\
& \leq \delta_{1} (E) \int |f| d\mu \end{align*}
which is \eqref{thm2.3eq}

\end{proof}

Generally speaking, we expect $\delta_1(E)$ to be a lot smaller than $\delta_0(E)$, because of the square root in the definition of $\delta_{0}(E)$. We gain this better estimate because we are requiring more smoothness in $u$ in Theorem \ref{thm:2.3}.

\begin{definition} \label{def:2.4}
Let $\harmonic_E$ denote the space of piecewise harmonic splines with nodes in $E$. In other words, the continuous functions $v$ such that $\laplace_{\mu} v = 0$ in the complement of $E$ (the condition $\laplace_{\mu} v = 0$ is independent of $\mu$). $\harmonic_E$ is a space of dimension $\#E$ and each $v \in \harmonic_E$ is uniquely determined by its values on $E$. (Note that if $E$ does not contain $V_{0}$, then the harmonic condition at points in $V_{0} \setminus E$ is just the vanishing of the normal derivative.)
\end{definition}

\begin{theorem} \label{thm:2.5}
There exists a set of weights $\left\{p(x)\right\}$ on $E$ such that 

\begin{equation} \label{thm2.5eq}
\int v d\mu = \sum_{x \in E} p(x)v(x) \quad \mbox{for all $v \in \harmonic_E$}.
\end{equation}

\end{theorem}

\begin{proof}
Let $v_{j} \in \harmonic_E$ be determined by the condition $v_j(x_k)=\delta_{jk}$ for all $x_k \in E$. Set

\begin{equation*}
p(x_j)=\int v_j d\mu.
\end{equation*}
Then \eqref{thm2.5eq} follows from $v=\sum_{j} v(x_j)v_j$.

\end{proof}

\begin{theorem} \label{thm:2.6}

(a) Suppose that $u \in \dom\energy$. Then

\begin{equation} \label{eq:2.21}
\left| \int u d\mu - \sum_{x \in E} p(x)u(x) \right| \leq \delta_{0} (E) \energy(u)^{\sfrac{1}{2}}.
\end{equation}
(b) Suppose $u \in \dom_{L^{1}}\laplace_{\mu}$. Then

\begin{equation} \label{eq:2.22}
\left| \int u d\mu - \sum_{x \in E} p(x)u(x) \right| \leq \delta_{1}(E) \int \left| \laplace_{\mu} u \right| d\mu.
\end{equation}

\end{theorem}

\begin{proof}

Write $u=(u-v)+v$ where $v \in \harmonic_E$ and $v|_E=u|_E$. Then

\[ \int u d\mu - \sum_{x \in E} p(x)u(x) = \int (u-v) d\mu + \int v d\mu - \sum_{x \in E} p(x)v(x) = \int (u-v) d\mu \]
by Theorem \ref{thm:2.5}. For part (a), we apply Theorem \ref{thm:2.2} to $u-v$ to obtain

\[ \left| \int (u-v) d\mu \right| \leq \delta_0(E) \energy(u-v)^{\sfrac{1}{2}}. \]
Since $u-v$ vanishes on $E$, we have

\[ \energy(u-v, v)=-\int (u-v) \laplace_{\mu}v d\mu=0\]
 since $\laplace_{\mu}v=0$ away from $E$. Thus, $\energy(u, v)=\energy(v, v)$ and hence
 
\[ \energy(u-v, u-v)=\energy(u, u)-\energy(v, v) \leq \energy(u, u), \] so we obtain \eqref{eq:2.21}. For part (b), we apply Theorem \ref{thm:2.3} to $u-v$ to obtain
 
\[ \left| \int (u-v) d\mu \right| \leq \delta_1(E) \int \left| \laplace_{\mu} u - \laplace_{\mu} v \right| d\mu \]
However, $\laplace_{\mu} v = 0$ away from $E$ and $\mu$ is assumed to be non-atomic, so we obtain \eqref{eq:2.22}.
 
\end{proof}
 
It may not always be feasible to compute the weights $\left\{p(x)\right\}$ precisely, or we may have a preference for a different set of weights, for example the uniform weights $w(x)=\frac{1}{\#E}$ for all $x \in E$. So we want a more flexible theorem that gives error estimates for general weights.
 
\begin{definition} \label{def:2.7}
Let $R$ denote the radius in the effective resistance metric, namely the minimum value for which there exists $x_{0} \in K$ (the \textquotedblleft center") such that the estimate
 
\begin{equation} \label{eq:2.23}
\left| u(x)-u(x_{0}) \right|^{2} \leq R\energy(u)
\end{equation}
holds for all $x \in K$ and all $u \in \dom\energy$. For any set of finite weights $\left\{w(x)\right\}$, define
 
\begin{equation} \label{eq:2.24}
\delta(E, w) = R^{\sfrac{1}{2}} \sum_{x \in E} \left| p(x)-w(x) \right|.
\end{equation}
 
\end{definition}
 
\begin{theorem} \label{thm:2.8}
(a) If $u \in \dom\energy$ then
 
\begin{equation*}
\left| \int u d\mu - \sum_{x \in E} w(x)u(x) \right| \leq \left( \delta_{0}(E)+\delta(E, w) \right)\energy(u)^{\sfrac{1}{2}}.
\end{equation*}
(b) If $u \in \dom_{L^1} \laplace_{\mu}$, then
 
\begin{equation*}
\left| \int u d\mu - \sum_{x \in E} w(x)u(x) \right| \leq \delta_1(E)\int |\laplace_{\mu} u| d\mu + \delta(E, w)\energy(u)^{\sfrac{1}{2}}.
\end{equation*}

\end{theorem}
 
\begin{proof}
 
In view of Theorem \ref{thm:2.6} it suffices to show
 
\begin{equation*}
\left| \sum_{x \in E} \left( p(x)-w(x) \right)u(x) \right| \leq \delta(E, w)\energy(u)^{\sfrac{1}{2}}.
\end{equation*}
for $u \in \dom\energy$.
Note that $\sum_{x \in E} (p(x)-w(x))=0$ since both $\left\{p(x)\right\}$ and $\left\{w(x)\right\}$ sum to $1$. Let $\bar{u}(x)=u(x)-c$, for $c=u(x_0)$. Then $\energy(\bar{u})=\energy(u)$ and $\sum_{x \in E} \left( p(x)-w(x) \right) u(x) = \sum_{x \in E} \left( p(x)-w(x) \right) \bar{u}(x)$. So
 
\begin{align*}
\left| \sum_{x \in E} (p(x)-w(x))u(x) \right| & \leq ||\bar{u}||_{\infty} \sum_{x \in E} \left| p(x)-w(x) \right| \\
& \leq \delta(E, w) \energy(u)^{\sfrac{1}{2}}
\end{align*}
by \eqref{eq:2.23} and \eqref{eq:2.24}.

\end{proof}
 
Note that we can not control $\left| \sum_{x \in E} (p(x)-w(x))u(x) \right|$ in terms of $\int |\laplace u| d\mu$ alone because $u$ could be harmonic and we can not make it zero by subtracting a constant.
 
In some examples the constant $\delta_{1}$ is larger than desirable because $g_{E}(x)$ has a large spike near the point where it assumes its maximum but is otherwise considerably smaller. In that case we may obtain a smaller constant by applying H\"{o}lder's inequality in the proof of Theorem \ref{thm:2.3}, at the cost of assuming that $\laplace_{\mu} u$ is in some $L^p$ space for $p>1$.
 
\begin{theorem} \label{thm:2.9}
Assume $u \in \dom_{L^p} \laplace_{\mu}$ for some $p \geq 1$, and let $q$ be the dual index, $\frac{1}{p}+\frac{1}{q}=1$. Then
 
\begin{equation}
\left| \int u d\mu - \sum_{x \in E} w(x)u(x) \right| \leq ||g_{E}||_{q}    ||\laplace_{\mu} u ||_{p} + \delta(E, w)\energy(u)^{\sfrac{1}{2}}.
\end{equation}
 
\end{theorem}
 
\begin{proof}
The same as the proof of Theorem \ref{thm:2.8}.b, except for the use of H\"{o}lder's inequality in the proof of Theorem \ref{thm:2.3}.
\end{proof}

Note that if we take $p=\infty$, then
 
\begin{equation*}
||g_{E}||_{1}=\delta_{0}^{2}.
\end{equation*}

\section{Basic examples on SG}

In this section we consider some examples of the set $E$ for the case of $SG$ with $\mu$ the standard symmetric self-similar measure. We will use ``$\lap$" to refer to the Laplacian with respect to this measure. For each example we compute our estimate of $\delta_0(E)$ and $\delta_1(E)$, the weights $\left\{p(x)\right\}$, and $\delta(E,w)$ when $w$ is the uniform weight $w(x)=\frac{1}{\# E}$. The results of section 2 give us a recipe to make these computations. We find the function $g_E(x)=\int G_E(x,y) d\mu(y)$. This function is non-negative and vanishes on $E$. Its integral over $SG$ is $(\delta_0(E))^2$ and its maximum value is $\delta_1(E)$. For all $x \in E$, we compute the weight $p(x)=\int v_x d\mu$ (where $v_x$ is specified by \ref{thm:2.5}) by computing the harmonic spline $v_x$.

\begin{example} \label{ex:3.1}
$E=V_0$.
\end{example}

\cite{splines} provides an algorithm to compute the values of multiharmonic functions on $V_*$ for an expansive family of fractals. For $SG$, section 5.1 of \cite{splines} gives the specific values resulting from this algorithm. By Table 5.1 of \cite{splines}, if $f_{1k}$ is the biharmonic function such that $f_{1k}|_{V_0} = 0$ and $\laplace f_{1k} = h_k$, then

\begin{equation*}
    \left.\begin{matrix}
        5\cdot f_{1k} (F_i q_k) = p_1 = -.12 \\
        f_{1k} (F_i q_k) = -\sfrac{9}{375}
    \end{matrix} \right\} \mbox{ (for $i \neq k$)}
\end{equation*}
and

\begin{equation*}
    \left.\begin{matrix}
        5\cdot f_{1k} (F_i q_j) = q_1 = -.09333\dots \\
        f_{1k} (F_i q_j) = -\sfrac{7}{375}
    \end{matrix} \right\} \mbox{ (for $i,j,k$ all distinct)}.
\end{equation*}
Thus, if $v(x) = \int G_{V_0} h_0 (y) d\mu(y) = f_{1]}(x)$ (or equivalently, $\laplace v = h_0$ and $v|_{V_0}=0$), then the values of $v$ on $V_1$ are shown in Figure \ref{one}. In general, if $v(x) = \int G_{V_0} (x, y) h_i (y) d\mu(y)$, then the values of $v$ on $V_1 \setminus V_0$ are $-\sfrac{9}{375}$, $-\sfrac{9}{375}$, and $-\sfrac{7}{375}$, with the $-\sfrac{7}{375}$ occurring at the midpoint of the side opposite from $q_i$.

Also $g_{V_0} (x) = \int G_{V_0} (x,y) \big( -h_0 (y) -h_1 (y) -h_2 (y) \big) d\mu(y) = -f_{10}(x)-f_{11}(x)-f_{12}(x)$, so $g_{V_0}$ takes the values shown in Figure 3.2.

\begin{figure}[h] %1,2
\begin{framed}
\centering
\begin{minipage}{.45\textwidth}
    \centering
    \begin{tikzpicture}

    \draw[black] (0,0) -- (3,0) -- (1.5, 2.598) -- cycle;
    \draw[black] (0,0) -- (-3,0) -- (-1.5, 2.598) -- cycle;
    \draw[black] (-1.5, 2.598) -- (1.5, 2.598) -- (0, 5.196) -- cycle;
    \filldraw[black] (0, 5.196) circle (1pt) node[anchor=south] {0};
    \filldraw[black] (-3, 0) circle (1pt) node[anchor=north] {0};
    \filldraw[black] (3, 0) circle (1pt) node[anchor=north] {0};
    \filldraw[black] (0, 0) circle (1pt) node[anchor=north] {-7/375};
    \filldraw[black] (1.5, 2.598) circle (1pt) node[anchor=west] {-9/375};
    \filldraw[black] (-1.5, 2.598) circle (1pt) node[anchor=east] {-9/375};

    \end{tikzpicture}
    \captionof{figure}{The values of $v(x)=\int G_{V_0} (x,y) h_0 (y) d\mu(y)$ on $V_1$.}
    \label{one}
\end{minipage}\hfill\begin{minipage}{.45\textwidth}
    \centering
    \begin{tikzpicture}

    \draw[black] (0,0) -- (3,0) -- (1.5, 2.598) -- cycle;
    \draw[black] (0,0) -- (-3,0) -- (-1.5, 2.598) -- cycle;
    \draw[black] (-1.5, 2.598) -- (1.5, 2.598) -- (0, 5.196) -- cycle;
    \filldraw[black] (0, 5.196) circle (1pt) node[anchor=south] {0};
    \filldraw[black] (-3, 0) circle (1pt) node[anchor=north] {0};
    \filldraw[black] (3, 0) circle (1pt) node[anchor=north] {0};
    \filldraw[black] (0, 0) circle (1pt) node[anchor=north] {1/15};
    \filldraw[black] (1.5, 2.598) circle (1pt) node[anchor=west] {1/15};
    \filldraw[black] (-1.5, 2.598) circle (1pt) node[anchor=east] {1/15};

    \end{tikzpicture}
    \captionof{figure}{The values of $g_{V_0}$ on $V_1$.}
    \label{two}
\end{minipage}
\end{framed}
\end{figure}

Moreover, $\lap^2 g_E = \lap(-1)=0$, so $g_E$ is the biharmonic function whose Laplacian is equal to $-1$ everywhere. From the values $g_{V_0}$ takes on $V_m$, we deduce what values it takes for $x \in V_{m+1}$. Once we have $(g_{V_0} \circ F_w)|_{V_0}$ for a word $w$ of length $m$, because we also know that $(\lap g_{V_0})\circ F_w = -1$, we can use the Green's function to calculate $(g_{V_0} \circ F_w)|_{V_1}$.

\begin{lemma} \label{lem:3.1}
If $|w|=m$ and $(g_{V_0} \circ F_w)|_{V_0}$ takes values as shown in Figure 3.3, then $(g_{V_0} \circ F_w)|_{V_1}$ takes values as shown in Figure 3.4.
\end{lemma}

\begin{figure}[h] %3,4
\begin{framed}
\begin{minipage}{.4\textwidth}
\centering
\begin{tikzpicture}

\draw[black] (-3,0) -- (3,0) -- (0, 5.196) -- cycle;
\filldraw[black] (0, 5.196) circle (1pt) node[anchor=south] {a};
\filldraw[black] (-3, 0) circle (1pt) node[anchor=north] {b};
\filldraw[black] (3, 0) circle (1pt) node[anchor=north] {c};

\end{tikzpicture}
\captionof{figure}{The values of $(g_{V_0} \circ F_w)$ on $V_0$, in the context of Lemma \ref{lem:3.1}.}
\label{three}
\end{minipage}\hspace{.05\textwidth}\begin{minipage}{.4\textwidth}
\centering
\begin{tikzpicture}

\draw[black] (0,0) -- (3,0) -- (1.5, 2.598) -- cycle;
\draw[black] (0,0) -- (-3,0) -- (-1.5, 2.598) -- cycle;
\draw[black] (-1.5, 2.598) -- (1.5, 2.598) -- (0, 5.196) -- cycle;
\filldraw[black] (0, 5.196) circle (1pt) node[anchor=south] {a};
\filldraw[black] (-3, 0) circle (1pt) node[anchor=north] {b};
\filldraw[black] (3, 0) circle (1pt) node[anchor=north] {c};
\filldraw[black] (0, 0) circle (1pt) node[anchor=north] {$\frac{a+2b+2c}{5} + \frac{1}{15\cdot 5^m}$};
\filldraw[black] (1.5, 2.598) circle (1pt) node[anchor=west] {$\frac{2a+b+2c}{5} + \frac{1}{15\cdot 5^m}$};
\filldraw[black] (-1.5, 2.598) circle (1pt) node[anchor=east] {$\frac{2a+2b+c}{5} + \frac{1}{15\cdot 5^m}$};

\end{tikzpicture}
\captionof{figure}{The values of $(g_{V_0} \circ F_w)$ on $V_1$, in the context of Lemma \ref{lem:3.1}.}
\label{four}
\end{minipage}
\end{framed}
\end{figure}

\begin{proof}
Let $u=g_{V_0}\circ F_w$. Let $\tilde{u}$ be the harmonic function that shares the values of $\tilde{u}$ on $V_0$. A simple consequence of the pointwise formulation of the Laplacian is:

\begin{equation*}
\lap (f \circ F_w) = r_w \mu_w (\lap f) \circ F_w
\end{equation*}
or

\begin{equation*}
\lap (f \circ F_w) =\left(\frac{1}{5}\right)^m (\lap f) \circ F_w.
\end{equation*}
Therefore, $\lap u = \lap ( g_{V_0} \circ F_w) = \left(\frac{1}{5}\right)^m (\lap g_{V_0}) \circ F_w = \left(\frac{1}{5}\right)^m$.

$u-\tilde{u}$ has the same Laplacian as $u$, so

\begin{equation*}
\int G(x,y) \left(\frac{1}{5}\right)^m d\mu(y) = u(x)-\tilde{u}(x),
\end{equation*}
and this yields

\begin{equation} \label{eq:3.7}
\tilde{u}(x)+\left(\frac{1}{5}\right)^m g_{V_0} (x) = u(x).
\end{equation}
By applying \eqref{eq:3.7} to all $x \in V_1$ we get the values shown in Figure 3.4.
\end{proof}

We compute $\delta_0 (V_0)$ by considering $g_{V_0}$ as a series of piecewise harmonic functions.

\begin{theorem} \label{thm:3.2}
\begin{equation} \label{eq:3.8}
\delta_0 (V_0) = \frac{1}{3\sqrt{2}}.
\end{equation}
\end{theorem}

\begin{proof}
For all $m$, let $f_m$ be the piecewise harmonic $m$-spline whose values on $V_m$ are the same as those of $g_{V_0}$. For all $m>0$, $x \in V_m$, Lemma \ref{lem:3.1} gives:

\begin{equation*}
f_m(x)-f_{m-1}(x) = \left\{
\begin{matrix}
0 & : & x \in V_{m-1} \\
\frac{1}{15\cdot 5^m} & : & x \notin V_{m-1}
\end{matrix} \right.
\end{equation*}\begin{figure}[h]
\begin{framed}
\centering
\begin{minipage}{.45\textwidth}
    \centering
    \begin{tikzpicture}
    
    \draw[black] (0,0) -- (3,0) -- (1.5, 2.598) -- cycle;
    \draw[black] (0,0) -- (-3,0) -- (-1.5, 2.598) -- cycle;
    \draw[black] (-1.5, 2.598) -- (1.5, 2.598) -- (0, 5.196) -- cycle;
    \filldraw[black] (0, 5.196) circle (1pt) node[anchor=south] {0};
    \filldraw[black] (-3, 0) circle (1pt) node[anchor=north] {0};
    \filldraw[black] (3, 0) circle (1pt) node[anchor=north] {0};
    \filldraw[black] (0, 0) circle (1pt) node[anchor=north] {1/15};
    \filldraw[black] (1.5, 2.598) circle (1pt) node[anchor=west] {1/15};
    \filldraw[black] (-1.5, 2.598) circle (1pt) node[anchor=east] {1/15};
    
    \end{tikzpicture}
    \captionof{figure}{The values of $f_1 - f_0$ on $V_1$, in the context of Theorem \ref{thm:3.2}.}
    \label{five}
\end{minipage}\hfill\begin{minipage}{.45\textwidth}
    \centering
    \begin{tikzpicture}
    
    \draw[black] (0,0) -- (3,0) -- (1.5, 2.598) -- cycle;
\draw[black] (0,0) -- (-3,0) -- (-1.5, 2.598) -- cycle;
\draw[black] (-1.5, 2.598) -- (1.5, 2.598) -- (0, 5.196) -- cycle;

\draw[black] (0, 2.598) -- (.75, 3.897) -- (-.75, 3.897) -- cycle;
\draw[black] (-1.5, 0) -- (-.75, 1.299) -- (-2.25, 1.299) -- cycle;
\draw[black] (1.5, 0) -- (2.25, 1.299) -- (.75, 1.299) -- cycle;

\filldraw[black] (0, 5.196) circle (1pt) node[anchor=south] {0};
\filldraw[black] (-3, 0) circle (1pt) node[anchor=north] {0};
\filldraw[black] (3, 0) circle (1pt) node[anchor=north] {0};
\filldraw[black] (0, 0) circle (1pt) node[anchor=north] {0};
\filldraw[black] (1.5, 2.598) circle (1pt) node[anchor=west] {0};
\filldraw[black] (-1.5, 2.598) circle (1pt) node[anchor=east] {0};

\filldraw[black] (0, 2.598) circle (1pt) node[anchor=north] {$\frac{1}{75}$};
\filldraw[black] (.75, 3.897) circle (1pt) node[anchor=west] {$\frac{1}{75}$};
\filldraw[black] (-.75, 3.897) circle (1pt) node[anchor=east] {$\frac{1}{75}$};

\filldraw[black] (-1.5, 0) circle (1pt) node[anchor=north] {$\frac{1}{75}$};
\filldraw[black] (-.75, 1.299) circle (1pt) node[anchor=west] {$\frac{1}{75}$};
\filldraw[black] (-2.25, 1.299) circle (1pt) node[anchor=east] {$\frac{1}{75}$};

\filldraw[black] (1.5, 0) circle (1pt) node[anchor=north] {$\frac{1}{75}$};
\filldraw[black] (2.25, 1.299) circle (1pt) node[anchor=west] {$\frac{1}{75}$};
\filldraw[black] (.75, 1.299) circle (1pt) node[anchor=east] {$\frac{1}{75}$};
    
    \end{tikzpicture}
    \captionof{figure}{The values of $f_2 - f_1$ on $V_2$, in the context of Theorem \ref{thm:3.2}.}
    \label{six}
\end{minipage}
\end{framed}
\end{figure}Because $f_m - f_{m-1}$ is a harmonic $m$-spline, we can compute its integral from its values on $V_m$. The values of $f_m - f_{m-1}$ on the boundary of some $m$-cell are (in some order) $0$, $\frac{1}{15\cdot5^{m-1}}$, and $\frac{1}{15\cdot5^{m-1}}$, because the boundary of each $m$-cell contains one point from $V_{m-1}$ and two from $V_m \setminus V_{m-1}$. Thus,

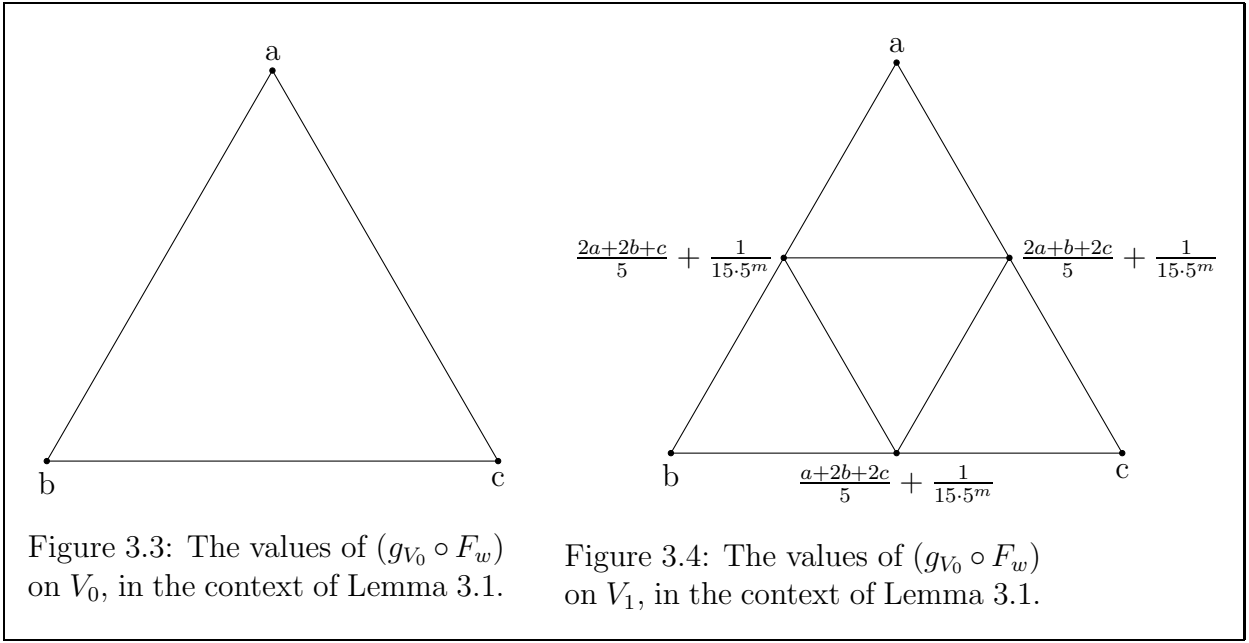
\begin{equation*}
\int f_m - f_{m - 1} d\mu = \frac{2}{3} \cdot \frac{1}{15\cdot5^{m-1}}.
\end{equation*}
Since $\int g_{V_0} d\mu$ is clearly the limit of $\int f_m d\mu$, we have

\begin{align*}
\int g_{V_0} d\mu &= \lim_{m\rightarrow\infty} f_m d\mu \\
&= \int f_0 d\mu + \sum_{m=1}^{\infty} \int f_m - f_{m-1} d\mu \\
&= 0 + \sum_{m=1}^{\infty} \frac{2}{3} \frac{1}{15\cdot5^{m-1}} \\
&= \frac{1}{18}.
\end{align*}
$\delta_0 (V_0)$ is equal to $\left(\int g_{V_0} d\mu\right)^{\sfrac{1}{2}}$, which gives us \eqref{eq:3.8}.

\end{proof}

To compute $\delta_1 (V_0)$, we must determine the maximum value of $g_{V_0}(x)$. To facilitate this computation, let us take advantage of the symmetry of $g_{V_0}$ and instead consider the function $u = 15g_{V_o} \circ F_0$. It is clear that

\begin{equation*}
\sup_{x\in SG} g_{V_0}(x) = \frac{1}{15} \sup_{x\in SG} u(x).
\end{equation*}

\begin{theorem} \label{thm:3.3}
If $F_w SG$ is an $m$-cell along the bottom line of $SG$ (that is, an $m$-cell whose bottom line is the bottom line of the gasket), then

\begin{equation} \label{eq:3.12}
u(F_w q_0) = 1-\left(\frac{1}{5}\right)^m,\quad u(F_w q_1) = 1,\quad \mbox{and } u(F_w q_2)=1.
\end{equation}

\end{theorem}

\begin{proof}

We will prove this claim by using induction on $m$.
 
 \textit{Base case:} If m=0, this is easy to verify.
 
 \textit{Inductive step:} Assume that $m \geq 0$ and that \eqref{eq:3.12} holds for all $m$-cells along the bottom line. For any $(m+1)$-cell $F_{w}K$ along the bottom line: $w$ is a word without any $0$s. We can assume without loss of generality that the last character of $w$ is a $1$ (rather than a $2$), so $w=w'1$ for some word $w'$, and $F_{w'} SG$ is an $m$-cell along the bottom line of $SG$. By the inductive hypothesis:
 
 \begin{equation*}
 u(F_w q_0) = 1-\left(\frac{1}{5}\right)^m,\quad u(F_w q_1) = 1,\quad \mbox{and } u(F_w q_2)=1.
 \end{equation*}
 
 Because $u=15g_{V_0} \circ F_0$, to describe the way $u|_{V_{m+1}\setminus V_m}$ depends on $u|_{V_m}$, we must use Lemma \ref{lem:3.1} but replace $\frac{1}{15\cdot5^m}$ with $\frac{1}{5^{m+1}}$. We obtain
 
 \begin{equation*}
 u(F_w q_0) = u(F_{w'} F_1 q_0) = \left( \frac{ 2 \left( 1 - \left(\frac{1}{5}\right)^{m+1}\right)+2+1}{5}\right)   +   \frac{1}{5^{m+1}}   = 1 - \left(\frac{1}{5}\right)^{m+1},
  \end{equation*}
 \begin{equation*}
 u(F_w q_1)=u(F_{w'} F_1 q_1)=u(F_{w'} q_1)=1,
 \end{equation*}
and
 
 \begin{equation*}
 u(F_w q_2) = u(F_{w'} F_1 q_2) = \left(\frac{1-\left(\frac{1}{5}\right)^m + 2 + 2}{5}\right) + \left(\frac{1}{5}\right)^{m+1}=1.
 \end{equation*}
 Thus, the inductive hypothesis holds for $(m+1)$-cells.

\end{proof}

\begin{theorem} \label{thm:3.4}
If $x$ is not on the bottom line of $SG$, then $u(x)<1$.
\end{theorem}

\begin{proof}
Because $x$ is not on the bottom line, there exists some $m$ such that an $m$-cell containing $x$ is along the bottom line, but an $(m+1)$-cell within this $m$-cell contains $x$ and is not on the bottom line. Thus, for some word $w$ consisting of $m$ characters, all of them are $1$ or $2$, and $x \in F_{w0} SG$.
By Theorem \ref{thm:3.3}, 

\begin{equation*}
u(F_{w0}q_0)=1 - \left(\frac{1}{5}\right)^m ,\quad u(F_{w0}q_1)=1,\quad \mbox{and } u(F_{w0}q_2)=1.
\end{equation*}
For all non-negative integers $k$, define 

\begin{equation*}
\varphi(k)=\sup_{x \in V_{m+k+1}, x \in F_{w0} SG} u(x).
\end{equation*}
We need only consider one $(m+1)$-cell, so

\begin{equation*}
\varphi(0)= 1 - \left(\frac{1}{5}\right)^{m+1}.
\end{equation*}
For all $k$, $\varphi(k+1)\leq \varphi(k) + \left(\frac{1}{5}\right)^{m+k+2}$ (by applying our adjustment of Lemma \ref{lem:3.1} on an $(m+1)$-cell whose values on the boundary are all less than or equal to $\varphi(k)$). For all $k$, 

\begin{equation*}
\varphi(k) \leq 1 - \left(\frac{1}{5}\right)^{m+1} + \sum_{i=0}^\infty \left(\frac{1}{5}\right)^{m+i+2} = 1 - \left(\frac{3}{4}\right) \left(\frac{1}{5}\right)^{m+1}.
\end{equation*}
Because $V_*$ is dense in $SG$, $u(x)\leq 1 - \left(\frac{3}{4}\right)^{m+1}$, so $u(x)<1$.

\end{proof}

\begin{corollary} \label{cor:3.5}
$\delta_1 (V_0)$, or the maximum value of $g_{V_0}$, is $\frac{1}{15}$.
\end{corollary}

\begin{proof}
Let $x$ be any point in the Sierpi\'{n}ski gasket. If $x \in F_0 SG$, then either $x$ lies on the bottom line of $F_0 SG$ (in which case $g_{V_0}(x)=\frac{1}{15}$), or $x$ is above this line (in which case $g_{V_0}<\frac{1}{15}$). In either case, $g_{V_0}\leq\frac{1}{15}$. If $x \notin F_0 SG$, then by symmetry of $g_{V_0}$, there is some point $y \in F_0 SG$ such that $g(x)=g(y)\leq\frac{1}{15}$.
\end{proof}

\begin{figure}[h] %7,8,9
\begin{framed}
\centering
\begin{minipage}{.27\textwidth}
\centering
    \begin{tikzpicture}[scale=.6]
    
    \draw[black] (0,0) -- (3,0) -- (1.5, 2.598) -- cycle;
\draw[black] (0,0) -- (-3,0) -- (-1.5, 2.598) -- cycle;
\draw[black] (-1.5, 2.598) -- (1.5, 2.598) -- (0, 5.196) -- cycle;

\draw[black] (-2.25, 1.299) -- (-.75, 1.299) -- (-1.5, 0) -- cycle;
\draw[black] (2.25, 1.299) -- (.75, 1.299) -- (1.5, 0) -- cycle;

\draw[black] (-2.625, .6445) -- (-1.875, .6445) -- (-2.25, 0) -- cycle;
\draw[black] (2.625, .6445) -- (1.875, .6445) -- (2.25, 0) -- cycle;
\draw[black] (-1.125, .6445) -- (-.365, .6445) -- (-.75, 0) -- cycle;
\draw[black] (1.125, .6445) -- (.365, .6445) -- (.75, 0) -- cycle;

\draw[black] (-2.8125, .32225) -- (-2.4375, .32225) -- (-2.625, 0) -- cycle;
\draw[black] (-2.0625, .32225) -- (-1.6875, .32225) -- (-1.875, 0) -- cycle;
\draw[black] (-1.3125, .32225) -- (-.9375, .32225) -- (-1.125, 0) -- cycle;
\draw[black] (-.5525, .32225) -- (-.1775, .32225) -- (-.365, 0) -- cycle;
\draw[black] (.5525, .32225) -- (.1775, .32225) -- (.365, 0) -- cycle;
\draw[black] (1.3125, .32225) -- (.9375, .32225) -- (1.125, 0) -- cycle;
\draw[black] (2.0625, .32225) -- (1.6875, .32225) -- (1.875, 0) -- cycle;
\draw[black] (2.8125, .32225) -- (2.4375, .32225) -- (2.625, 0) -- cycle;
    
    \end{tikzpicture}
    \captionof{figure}{The cells along the bottom line of $SG$.}
    \label{seven}
\end{minipage}\hfill\begin{minipage}{.27\textwidth}
    \centering
    \begin{tikzpicture}[scale=.6]
    
    \draw[black] (0,0) -- (3,0) -- (1.5, 2.598) -- cycle;
\draw[black] (0,0) -- (-3,0) -- (-1.5, 2.598) -- cycle;
\draw[black] (-1.5, 2.598) -- (1.5, 2.598) -- (0, 5.196) -- cycle;
\filldraw[black] (0, 5.196) circle (1pt) node[anchor=south] {$1 - \left(\frac{1}{5}\right)^m$};
\filldraw[black] (-3, 0) circle (1pt) node[anchor=north] {1};
\filldraw[black] (3, 0) circle (1pt) node[anchor=north] {1};
\filldraw[black] (0, 0) circle (1pt) node[anchor=north] {1};
\filldraw[black] (1.5, 2.598) circle (1pt) node[anchor=west] {$1 - \left(\frac{1}{5}\right)^{m+1}$};
\filldraw[black] (-1.5, 2.598) circle (1pt) node[anchor=east] {$1 - \left(\frac{1}{5}\right)^{m+1}$};

    \end{tikzpicture}
    \captionof{figure}{The values of $u$ along an $m$-cell along the bottom line of $SG$.}
    \label{eight}
\end{minipage}\hfill\hfill\hfill\begin{minipage}{.27\textwidth}
    \centering
    \begin{tikzpicture}[scale=.6]
    
    \draw[black] (-3, 0) -- (3, 0) -- (0, 5.196) -- cycle;
\draw[line width=2pt] (0, 0) -- (1.5, 2.598) -- (-1.5, 2.598) -- cycle;
\filldraw[black] (0, 5.196) circle (1pt) node[anchor=south] {0};
\filldraw[black] (-3, 0) circle (1pt) node[anchor=north] {0};
\filldraw[black] (3, 0) circle (1pt) node[anchor=north] {0};
\filldraw[black] (0, 0) circle (1pt) node[anchor=north] {1/15};
\filldraw[black] (1.5, 2.598) circle (1pt) node[anchor=west] {1/15};
\filldraw[black] (-1.5, 2.598) circle (1pt) node[anchor=east] {1/15};
    
    \end{tikzpicture}
    \captionof{figure}{The values of $g_{V_0}$ on $V_1$. $g_{V_0}$ attains its maximum value, $\frac{1}{15}$ all along the thickened lines.}
    \label{nine}
\end{minipage}
\end{framed}
\end{figure}

The weights $\left\{p(x)\right\}$ for all $x \in V_0$ are $\frac{1}{3}$ (the integral of any of $h_0$, $h_1$, and $h_2$). Because $\left\{p(x)\right\}$ for all $x \in V_0$ and the uniform weights $\{w(x)\}$ are one and the same, $\delta(E,w)=0$.

The calculations for all of the examples $E$ that follow may be greatly simplified by observing that the difference between $g_{V_0}$ and $g_E$ is a piecewise harmonic function, and combining calculations involving that piecewise harmonic function with the calculations that we already made in Example \ref{ex:3.1}.

 \begin{lemma} \label{lem:3.6} If $E$ is a finite superset of $V_{0}$, then $g_{V_{0}}-g_{E}$ is harmonic away from E. \end{lemma}
 \begin{proof} For $x \notin E$, $\laplace (g_{V_{0}}-g_{E})(x) = \laplace g_{V_{0}}(x) - \laplace g_{E}(x) = -1-(-1) = 0$, so $g_{V_{0}}-g_{E}$ is harmonic away from $E$. \end{proof}

 \begin{example} \label{ex:3.2}
 $E=V_{0} \cup \left\{x_0\right\}$, where $x_{0}$ is a member of $V_1 \setminus V_0$.
 \end{example}
 
 The next example we consider is when $E$ contains the three points from $V_{0}$ and one additional point from $V_{1}$. Without loss of generality, we take $x_{0}=F_{0}q_{1}$. The analysis would be exactly the same for either of the other choices of $x_{0}$ (if the characters $0$, $1$, and $2$ were permuted accordingly). For this set $E$, Lemma \ref{lem:3.6} guarantees that $g_{V_{0}}-g_{E}$ is harmonic away from $E$. Because $E \subset V_{1}$, $g_{V_{0}}-g_{E}$ must be a harmonic $1$-spline. Furthermore, $(g_{V_{0}}-g_{E})|_{V_{0}}=0$ (because both $G_E$ and $G_{V_{0}}$ vanish on $V_{0}$) and $g_{V_{0}}-g_{E}$ is harmonic at $F_{0}q_{2}$ and $F_{1}q_{2}$.
 
 Because $F_{0}q_{1} \in E$, $g_{E}(F_{0}q_{1})=0$, so $(g_{V_{0}}-g_{E})(F_{0}q_{1})=g_{V_{0}}(F_{0}q_{1})=\frac{1}{15}$.
 Let $x=(g_{V_{0}}-g_{E})(F_{0}q_{2})$. By symmetry, $(g_{V_{0}}-g_{E})(F_{1}q_{2})=x$.
 By the harmonicity of $g_{V_0}-g_{E}$ at $F_{0}q_{1}$, we have 
 
 \begin{equation*}
 x  = \frac{0+\frac{1}{15}+x+0}{4},
 \end{equation*} hence \begin{equation*}
 x  = \frac{1}{45}.
 \end{equation*}
 
 Thus, $g_{V_{0}}-g_{E}$ is the harmonic $1$-spline with values as shown in Figure \ref{ten}.

\begin{figure}[h]
\begin{framed}
\begin{center}
\begin{tikzpicture}

\draw[black] (-3, 0) -- (3, 0) -- (0, 5.196) -- cycle;
\draw[black] (0, 0) -- (1.5, 2.598) -- (-1.5, 2.598) -- cycle;
\filldraw[black] (0, 5.196) circle (1pt) node[anchor=south] {0};
\filldraw[black] (-3, 0) circle (1pt) node[anchor=north] {0};
\filldraw[black] (3, 0) circle (1pt) node[anchor=north] {0};
\filldraw[black] (0, 0) circle (1pt) node[anchor=north] {1/45};
\filldraw[black] (1.5, 2.598) circle (1pt) node[anchor=west] {1/45};
\filldraw[black] (-1.5, 2.598) circle (1pt) node[anchor=east] {1/15};

\end{tikzpicture}
\end{center}

\caption{The values of $g_{V_0} - g_E$ on $V_1$, in the context of Example \ref{ex:3.2}.}
\label{ten}

\end{framed}
\end{figure}

\begin{theorem}
If $E=V_0 \cup \left\{x_0\right\}$ for some $x_0 \in V_1 \setminus V_0$, then

\begin{equation*}
\delta_0 (E) = \left(\frac{5}{162}\right)^{\sfrac{1}{2}}.
\end{equation*}

\end{theorem}

\begin{proof}
Because $g_{V_0} - g_E$ is the harmonic spline with values shown in Figure \ref{ten}, 

\begin{equation*}
\int g_{V_0} - g_E d\mu = \frac{1}{3} \Bigg( \frac{1}{3} \left( \frac{1}{15}+\frac{1}{45} \right) + \frac{1}{3} \left( \frac{1}{15}+\frac{1}{45} \right) + \frac{1}{3} \left( \frac{1}{45}+\frac{1}{45} \right) \Bigg) = \frac{2}{81}
\end{equation*}

and so

\begin{equation*}
\int g_E d\mu = \int g_{V_0} d\mu - \int g_{V_0} - g_E d\mu = \frac{1}{18} - \frac{2}{81} = \frac{5}{162}.
\end{equation*}

\begin{equation*}
\delta_0 (E) = \left(\int g_E d\mu\right)^{\sfrac{1}{2}} = \left(\frac{5}{162}\right)^{\sfrac{1}{2}}.
\end{equation*}

\end{proof}

We know how to compute $g_{V_0}$ for all points in $V_*$, and $g_{V_0}-g_E$ is a harmonic spline. Therefore, we can equally well compute $g_E$ for all points in $V_*$. After computing the values of $g_E$ for the finite graph approximations up to $V_{10}$, the maximum value is $\frac{11}{225}$, a value that first occurs in $V_2$. We conjecture that this is the absolute maximum value the function takes, and that

\begin{equation*}
\delta_1 (E) = \frac{11}{225}.
\end{equation*}

\begin{figure}[!htb]
\begin{framed}
\centering
\begin{minipage}{.45\textwidth}
    \centering
    \begin{tikzpicture}
    
    \draw[black] (0,0) -- (3,0) -- (1.5, 2.598) -- cycle;
    \draw[black] (0,0) -- (-3,0) -- (-1.5, 2.598) -- cycle;
    \draw[black] (-1.5, 2.598) -- (1.5, 2.598) -- (0, 5.196) -- cycle;
    \filldraw[black] (0, 5.196) circle (1pt) node[anchor=south] {1};
    \filldraw[black] (-3, 0) circle (1pt) node[anchor=north] {0};
    \filldraw[black] (3, 0) circle (1pt) node[anchor=north] {0};
    \filldraw[black] (0, 0) circle (1pt) node[anchor=north] {1/15};
    \filldraw[black] (1.5, 2.598) circle (1pt) node[anchor=west] {4/15};
    \filldraw[black] (-1.5, 2.598) circle (1pt) node[anchor=east] {0};
    
    \end{tikzpicture}
    \captionof{figure}{The values of $v_{q_0}$ on $V_1$, in the context of Example 3.2.}
    \label{eleven}
\end{minipage}\hfill\begin{minipage}{.45\textwidth}
    \centering
    \begin{tikzpicture}
    
    \draw[black] (0,0) -- (3,0) -- (1.5, 2.598) -- cycle;
    \draw[black] (0,0) -- (-3,0) -- (-1.5, 2.598) -- cycle;
    \draw[black] (-1.5, 2.598) -- (1.5, 2.598) -- (0, 5.196) -- cycle;
    \filldraw[black] (0, 5.196) circle (1pt) node[anchor=south] {0};
    \filldraw[black] (-3, 0) circle (1pt) node[anchor=north] {1};
    \filldraw[black] (3, 0) circle (1pt) node[anchor=north] {0};
    \filldraw[black] (0, 0) circle (1pt) node[anchor=north] {4/15};
    \filldraw[black] (1.5, 2.598) circle (1pt) node[anchor=west] {1/15};
    \filldraw[black] (-1.5, 2.598) circle (1pt) node[anchor=east] {0};
    
    \end{tikzpicture}
    \captionof{figure}{The values of $v_{q_1}$ on $V_1$. in the context of Example 3.2.}
    \label{twelve}
\end{minipage}
\begin{minipage}{.45\textwidth}
    \centering
    \begin{tikzpicture}
    
    \draw[black] (0,0) -- (3,0) -- (1.5, 2.598) -- cycle;
    \draw[black] (0,0) -- (-3,0) -- (-1.5, 2.598) -- cycle;
    \draw[black] (-1.5, 2.598) -- (1.5, 2.598) -- (0, 5.196) -- cycle;
    \filldraw[black] (0, 5.196) circle (1pt) node[anchor=south] {0};
    \filldraw[black] (-3, 0) circle (1pt) node[anchor=north] {0};
    \filldraw[black] (3, 0) circle (1pt) node[anchor=north] {1};
    \filldraw[black] (0, 0) circle (1pt) node[anchor=north] {1/3};
    \filldraw[black] (1.5, 2.598) circle (1pt) node[anchor=west] {1/3};
    \filldraw[black] (-1.5, 2.598) circle (1pt) node[anchor=east] {0};
    
    \end{tikzpicture}
    \captionof{figure}{The values of $v_{q_2}$ on $V_1$ in the context of Example 3.2.}
    \label{thirteen}
\end{minipage}\hfill\begin{minipage}{.45\textwidth}
    \centering
    \begin{tikzpicture}
    
    \draw[black] (0,0) -- (3,0) -- (1.5, 2.598) -- cycle;
    \draw[black] (0,0) -- (-3,0) -- (-1.5, 2.598) -- cycle;
    \draw[black] (-1.5, 2.598) -- (1.5, 2.598) -- (0, 5.196) -- cycle;
    \filldraw[black] (0, 5.196) circle (1pt) node[anchor=south] {0};
    \filldraw[black] (-3, 0) circle (1pt) node[anchor=north] {0};
    \filldraw[black] (3, 0) circle (1pt) node[anchor=north] {0};
    \filldraw[black] (0, 0) circle (1pt) node[anchor=north] {1/3};
    \filldraw[black] (1.5, 2.598) circle (1pt) node[anchor=west] {1/3};
    \filldraw[black] (-1.5, 2.598) circle (1pt) node[anchor=east] {1};
    
    \end{tikzpicture}
    \captionof{figure}{The values of $v_{F_0 q_1}$ on $V_1$ in the context of Example 3.2.}
    \label{fourteen}
\end{minipage}
\end{framed}
\end{figure}

Each of the functions $v_x$ for $x \in E$ will be a function that is harmonic away from $E$ and is determined the same way we determined $g_{V_0}-g_E$: by assigning the appropriate values to the points in $E$ (in this case, $1$ for $x$, $0$ for all other points in $E$) and choosing the values for $V_1 \setminus V_0$ such that $v_x$ is harmonic at these points. Then, $p(x)$ will be $\int v_x d\mu$. The values of that the functions $v_x$ ($x \in E$) take on $V_1$ are shown in Figures \ref{eleven}-\ref{fourteen}. The weights, calculated from these functions, are:
\begin{align*}
    p(q_0)&=\sfrac{5}{27},\\
    p(q_1)&=\sfrac{5}{27},\\
    p(q_2)&=\sfrac{7}{27},\\
    p(F_0 q_1)&=\sfrac{10}{27}.
\end{align*}

If $\{w(x)\}$ are the uniform weights,
\begin{equation*}
    \delta(E,w) = R^{\sfrac{1}{2}} \cdot \frac{7}{27}.
\end{equation*}
Recall that we took $x_0$ to be $F_0 q_1$, but one could also take $F_0 q_2$ or $F_1 q_2$, and can figure out the resulting weights from the above analysis by symmetry.

\begin{example} \label{ex:3.3} $E=V_0 \cup \{x_0, x_1\}$, where $x_0$ and $x_1$ are distinct elements of $V_1 \setminus V_0$.
\end{example}

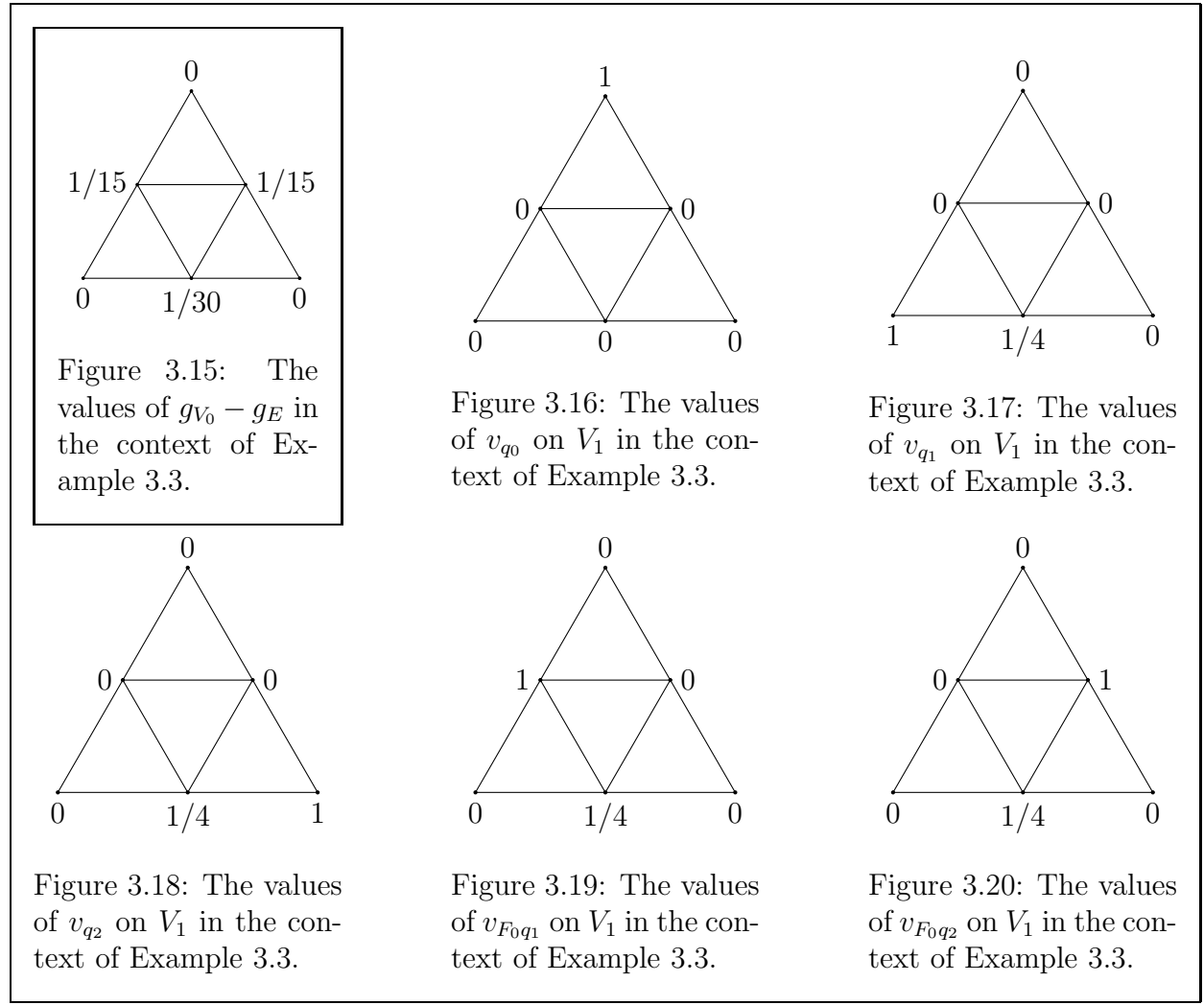
\begin{figure}[h]
\begin{framed}
\centering
\begin{minipage}{.27\textwidth}
    \begin{framed}
    \centering
    \begin{tikzpicture}[scale=.5]
    
    \draw[black] (0,0) -- (3,0) -- (1.5, 2.598) -- cycle;
    \draw[black] (0,0) -- (-3,0) -- (-1.5, 2.598) -- cycle;
    \draw[black] (-1.5, 2.598) -- (1.5, 2.598) -- (0, 5.196) -- cycle;
    \filldraw[black] (0, 5.196) circle (1pt) node[anchor=south] {0};
    \filldraw[black] (-3, 0) circle (1pt) node[anchor=north] {0};
    \filldraw[black] (3, 0) circle (1pt) node[anchor=north] {0};
    \filldraw[black] (0, 0) circle (1pt) node[anchor=north] {1/30};
    \filldraw[black] (1.5, 2.598) circle (1pt) node[anchor=west] {1/15};
    \filldraw[black] (-1.5, 2.598) circle (1pt) node[anchor=east] {1/15};
    
    \end{tikzpicture}
    \captionof{figure}{The values of $g_{V_0}-g_E$ in the context of Example \ref{ex:3.3}.}
    \label{fifteen}
    \end{framed}
\end{minipage}\hfill\begin{minipage}{.27\textwidth}
    \centering
    \begin{tikzpicture}[scale=.6]
    
    \draw[black] (0,0) -- (3,0) -- (1.5, 2.598) -- cycle;
    \draw[black] (0,0) -- (-3,0) -- (-1.5, 2.598) -- cycle;
    \draw[black] (-1.5, 2.598) -- (1.5, 2.598) -- (0, 5.196) -- cycle;
    \filldraw[black] (0, 5.196) circle (1pt) node[anchor=south] {1};
    \filldraw[black] (-3, 0) circle (1pt) node[anchor=north] {0};
    \filldraw[black] (3, 0) circle (1pt) node[anchor=north] {0};
    \filldraw[black] (0, 0) circle (1pt) node[anchor=north] {0};
    \filldraw[black] (1.5, 2.598) circle (1pt) node[anchor=west] {0};
    \filldraw[black] (-1.5, 2.598) circle (1pt) node[anchor=east] {0};
    
    \end{tikzpicture}
    \captionof{figure}{The values of $v_{q_0}$ on $V_1$ in the context of Example \ref{ex:3.3}.}
    \label{sixteen}
\end{minipage}\hfill\begin{minipage}{.27\textwidth}
    \centering
    \begin{tikzpicture}[scale=.6]
    
    \draw[black] (0,0) -- (3,0) -- (1.5, 2.598) -- cycle;
    \draw[black] (0,0) -- (-3,0) -- (-1.5, 2.598) -- cycle;
    \draw[black] (-1.5, 2.598) -- (1.5, 2.598) -- (0, 5.196) -- cycle;
    \filldraw[black] (0, 5.196) circle (1pt) node[anchor=south] {0};
    \filldraw[black] (-3, 0) circle (1pt) node[anchor=north] {1};
    \filldraw[black] (3, 0) circle (1pt) node[anchor=north] {0};
    \filldraw[black] (0, 0) circle (1pt) node[anchor=north] {1/4};
    \filldraw[black] (1.5, 2.598) circle (1pt) node[anchor=west] {0};
    \filldraw[black] (-1.5, 2.598) circle (1pt) node[anchor=east] {0};
    
    \end{tikzpicture}
    \captionof{figure}{The values of $v_{q_1}$ on $V_1$ in the context of Example \ref{ex:3.3}.}
    \label{seventeen}
\end{minipage}
\begin{minipage}{.27\textwidth}
    \centering
    \begin{tikzpicture}[scale=.6]
    
    \draw[black] (0,0) -- (3,0) -- (1.5, 2.598) -- cycle;
    \draw[black] (0,0) -- (-3,0) -- (-1.5, 2.598) -- cycle;
    \draw[black] (-1.5, 2.598) -- (1.5, 2.598) -- (0, 5.196) -- cycle;
    \filldraw[black] (0, 5.196) circle (1pt) node[anchor=south] {0};
    \filldraw[black] (-3, 0) circle (1pt) node[anchor=north] {0};
    \filldraw[black] (3, 0) circle (1pt) node[anchor=north] {1};
    \filldraw[black] (0, 0) circle (1pt) node[anchor=north] {1/4};
    \filldraw[black] (1.5, 2.598) circle (1pt) node[anchor=west] {0};
    \filldraw[black] (-1.5, 2.598) circle (1pt) node[anchor=east] {0};
    
    \end{tikzpicture}
    \captionof{figure}{The values of $v_{q_2}$ on $V_1$ in the context of Example \ref{ex:3.3}.}
    \label{eighteen}
\end{minipage}\hfill\begin{minipage}{.27\textwidth}
    \centering
    \begin{tikzpicture}[scale=.6]
    
    \draw[black] (0,0) -- (3,0) -- (1.5, 2.598) -- cycle;
    \draw[black] (0,0) -- (-3,0) -- (-1.5, 2.598) -- cycle;
    \draw[black] (-1.5, 2.598) -- (1.5, 2.598) -- (0, 5.196) -- cycle;
    \filldraw[black] (0, 5.196) circle (1pt) node[anchor=south] {0};
    \filldraw[black] (-3, 0) circle (1pt) node[anchor=north] {0};
    \filldraw[black] (3, 0) circle (1pt) node[anchor=north] {0};
    \filldraw[black] (0, 0) circle (1pt) node[anchor=north] {1/4};
    \filldraw[black] (1.5, 2.598) circle (1pt) node[anchor=west] {0};
    \filldraw[black] (-1.5, 2.598) circle (1pt) node[anchor=east] {1};
    
    \end{tikzpicture}
    \captionof{figure}{The values of $v_{F_0 q_1}$ on $V_1$ in the context of Example \ref{ex:3.3}.}
    \label{nineteen}
\end{minipage}\hfill\begin{minipage}{.27\textwidth}
    \centering
    \begin{tikzpicture}[scale=.6]
    
    \draw[black] (0,0) -- (3,0) -- (1.5, 2.598) -- cycle;
    \draw[black] (0,0) -- (-3,0) -- (-1.5, 2.598) -- cycle;
    \draw[black] (-1.5, 2.598) -- (1.5, 2.598) -- (0, 5.196) -- cycle;
    \filldraw[black] (0, 5.196) circle (1pt) node[anchor=south] {0};
    \filldraw[black] (-3, 0) circle (1pt) node[anchor=north] {0};
    \filldraw[black] (3, 0) circle (1pt) node[anchor=north] {0};
    \filldraw[black] (0, 0) circle (1pt) node[anchor=north] {1/4};
    \filldraw[black] (1.5, 2.598) circle (1pt) node[anchor=west] {1};
    \filldraw[black] (-1.5, 2.598) circle (1pt) node[anchor=east] {0};
    
    \end{tikzpicture}
    \captionof{figure}{The values of $v_{F_0 q_2}$ on $V_1$ in the context of Example \ref{ex:3.3}.}
    \label{twenty}
\end{minipage}
\end{framed}
\end{figure}

As in Example \ref{ex:3.2}, we assume specific values of $x_0$ and $x_1$ (in this case: $x_0=F_0 q_1$ and $x_1=F_0 q_2$), and know that one could determine the weights for any other choice of $x_0, x_1$ from these calculations. The methods for our calculations in this example are exactly the same as those in Example \ref{ex:3.2}.

$g_{V_0}-g_E$ is the harmonic $1$-spline with values shown in Figure \ref{fifteen}.
Thus
\begin{equation*}
    \int g_{V_0}-g_E d\mu = \frac{1}{27}
\end{equation*} and \begin{equation*}
    \int g_E d\mu = \frac{1}{54}.
\end{equation*}

From computing the values $g_E$ takes on all points in the finite graph approximations up to $V_{10}$, it appears that the maximum is $\frac{1}{30}$, which first occurs in $V_1$. We conjecture that this is the true maximum value of $h$, that
\begin{equation*}
    \delta_1 (E) = \frac{1}{30}.
\end{equation*}

The values of the functions $v_x$ are shown in Figures \ref{sixteen}-\ref{twenty}. The weights are
\begin{align*}
    p(q_0)&=\sfrac{1}{9},\\
    p(q_1)&=\sfrac{1}{6},\\
    p(q_2)&=\sfrac{1}{6},\\
    p(F_0 q_1)&=\sfrac{5}{18},\\
    p(F_0 q_2)&=\sfrac{5}{18}.
\end{align*}

For the uniform weights $\{w(x)\}$,
\begin{equation*}
    \delta(E,w)=R^{\sfrac{1}{2}} \cdot \frac{14}{45}.
\end{equation*}

The method used in Examples \ref{ex:3.2} and \ref{ex:3.3} of finding $g_E$ from the harmonic spline difference between $g_{V_0}$ and $g_E$ will work for any finite $E \supset V_0$. However, for larger $E$, there are, in some cases, improvements to the method. The post-criticially finite nature of the Sierpi\'{n}ski gasket allows us to easily analyze examples $E$ that divide the gasket into $m$-cells, where the inverse image of $E$ under each $F_w$ ($|w|=m$) is a more wieldy set (such as one of the Examples \ref{ex:3.1}, \ref{ex:3.2}, or \ref{ex:3.3}). The most important result that allows this analysis via decompositions into $m$-cells is the scaling of Green's functions.

\begin{theorem} \label{thm:3.8}
If $V_m \subseteq E$, and for all $|w|=m$, we denote $\left\{ F_{w}^{-1} x : x \in E \cap F_w SG \right\}$ (the inverse image of $E$ under $F_w$) by $E_w$, then
\begin{equation} \label{eq:thm3.8}
    G_E (x,y) = \left\{
    \begin{matrix}
    \left(\frac{3}{5}\right)^m G_{E_w}(F_w^{-1} x, F_w^{-1} y) && \mbox{if $x,y\in F_w SG$ and $|w|=m$}\\
    0 && \mbox{if $x$ and $y$ belong to separate $m$-cells}
    \end{matrix} \right..
\end{equation}
\end{theorem}

\begin{proof}

Let $a(x,y)$ be the right-hand side of \eqref{eq:thm3.8}, the function that we claim is $G_E$. It suffices to show that for a function $u \in \dom\lap$ such that $u|_E=0$,
\begin{equation}\label{2:thm3.8}
    -\int a(x,y) \left( \lap u(y) \right) d\mu(y) = u(x).
\end{equation}
If $x \in V_m$, both sides of \eqref{2:thm3.8} are $0$. If $x \notin V_m$, let $w$ be the unique word of length $m$ such that $x \in F_w K$. Let $x' = F_w^{-1} x$. The left-hand side of \eqref{2:thm3.8} is

\begin{align*}
    &-\int a \left( F_w x', y \right) (\lap u (y) d\mu(y))\\
    =& - \left( \frac{3}{5} \right)^m \int G_{E_w} (x', F_w y) (\lap u (y) d\mu(y))\\
    =& - \left( \frac{3}{5} \right)^m \left( \frac{1}{3} \right)^m \int_{F_w K} G_{E_w} (x',y') [\lap u \circ F_w] (y') d\mu(y')\\
    =& - \left( \frac{1}{5} \right)^m \int_{F_w K} G_{E_w} (x',y') [5^m \lap (u \circ F_w)] (y') d\mu(y')\\
    =& - \int G_{E_w} (x',y') \lap (u \circ F_w)(y')d\mu(y')\\
    =& u(F_w x')\\
    =& u(x)
\end{align*}
which verifies \eqref{2:thm3.8}. Thus, $a(x,y)$ is the Green's function for $E$.
\end{proof}

\begin{corollary} \label{cor:3.9}
Suppose $V_m \subseteq E$. Then

(a) \begin{equation} \label{3.9a}
    \delta_0 (E) = \left(\frac{\sum_{|w|=m}(\delta_0 (E_w))^2}{15^m}\right)^{\sfrac{1}{2}}.
\end{equation}

(b) \begin{equation} \label{3.9b}
    \delta_1 (E) = \frac{1}{5^m} \sup_{|w|=m} \delta_1 (E_w).
\end{equation}

(c) To simplify the notation, for all $\tilde{E}$, we let $p_{\tilde{E}}(x)$ refer to the weight of $x$ on $\tilde{E}$. Then for all $x\in E$:
\begin{equation*}
    p_E(x) = \left\{ \begin{matrix}
    \frac{1}{3^m} p_{E_w} (F_w^{-1} x) && \mbox{if $x$ belongs to a unique $m$-cell, $F_w SG$}\\
    \frac{1}{3^m} p_{E_w} (F_w^{-1} x)+\frac{1}{3^m} p_{E_{w'}} (F_{w'}^{-1} x) && \mbox{if $x$ belongs to two distinct $m$-cells, $F_w SG$ and $F_{w'} SG$}
    \end{matrix} \right..
\end{equation*}
\end{corollary}

\begin{proof} If $x$ belongs to the $m$-cell $F_w SG$ and $x'=F_w^{-1} x$:
\begin{align*}
    g_E (x) &= \int G_E (x,y) d\mu(y)\\
    &= \int_{F_w K} \left( \frac{3}{5} \right)^m G_{E_w} (F_w^{-1} x, F_w^{-1} y) d\mu(y) \\
    &= \left( \frac{3}{5} \right)^m \left( \frac{1}{3} \right)^m \int G_{E_w} (x',y') d\mu(y)\\
    &= \left( \frac{1}{5} \right)^m \int g_{E_w} (F_w^{-1} x).
\end{align*}

For (a), \begin{align*}
    (\delta_0 (E))^2 &= \int g_E d\mu = \sum_{|w|=m} \int_{F_w K} \left(\frac{1}{5}\right)^m  g_{E_w} \circ F_w^{-1} d\mu\\
    &= \left( \frac{1}{3} \right)^m \left( \frac{1}{5} \right)^m \sum_{|w|=m} \int g_{E_w} d\mu\\
    &= \left( \frac{1}{15}\right)^m \sum_{|w|=m} (\delta_0 (E_w))^2
\end{align*}
which implies \eqref{3.9a}.

For (b),
\begin{align*}
    \delta_1(E)&=\sup_z g_E (x)\\
    &= \sup_{|w|=m,x\in F_w K} \left( \frac{1}{5} \right)^m g_{E_w} (F_w^{-1} x)\\
    &= \left(\frac{1}{5}\right)^m \sup_{|w|=m}\sup_{x'\in K} g_{E_w}(x')\\
    &= \left(\frac{1}{5}\right)^m \sup_{|w|=m} \delta_1 (E_w)
\end{align*}
which is \eqref{3.9b}.

(c) is a trivial consequence of adding the harmonic indicators.

\end{proof}

The most obvious examples to apply Theorem \ref{thm:3.8} and Corollary \ref{cor:3.9} to are sets $E$ such that that $V_m \subseteq E \subseteq V_{m+1}$. In such examples, for all $|w|=m$, $E_w$ is either one of the sets described in Examples \ref{ex:3.1}, \ref{ex:3.2}, and \ref{ex:3.3}, or $V_1$, which is also simple. We first consider the notable case $E=V_m$, and then consider $V_m \subseteq E \subseteq V_{m+1}$ in general.

\begin{example} \label{ex:3.4} $E=V_m$. \end{example}

\begin{equation*}
    \delta_0 (V_m) = \frac{1}{3 \sqrt{2 \cdot 5^m}}.
\end{equation*}

\begin{equation*}
    \delta_1 (V_m) = \frac{1}{15\cdot 5^m}.
\end{equation*}

\begin{equation*}
    p(x) = \left\{ \begin{matrix}
    \frac{1}{3^{m+1}} && \mbox{if $x\in V_0$}\\
    \frac{2}{3^{m+1}} && \mbox{if $x \in V_m \setminus V_0$}
    \end{matrix}\right..
\end{equation*}

\begin{equation*}
    \delta(E,w) = \frac{2 \left( 3^m-1 \right)}{3^m (3^m+1)} R^{\sfrac{1}{2}}
\end{equation*}
for the uniform weights $\{w(x)\}$.

\begin{example} \label{ex:3.5} $V_m \subseteq E \subseteq V_{m+1}$. \end{example}

For all $|w|=m$, $E_w$ contains either $3$, $4$, $5$, or $6$ points. (In other words, $E_w$ is one of the sets described in Examples \ref{ex:3.1}, \ref{ex:3.2}, and \ref{ex:3.3}, or $E_w=V_1$, which is a special case of Example \ref{ex:3.4}.)

Let

\begin{align*}
    A &= \# \{ |w|=m : E_w = V_0 \},\\
    B &= \# \{ |w|=m : \#E_w=4 \},\\
    C &= \# \{ |w|=m : \#E_w=5 \},
\end{align*} and
\begin{equation*}
    D = \# \{ |w|=m : E_w = V_1 \}.
\end{equation*}

We can express $\delta_0 (E)$ and $\delta_1(E)$ in terms of $A$, $B$, $C$, and $D$:
\begin{align*}
    \delta_0 &= \left(\frac{A\cdot \frac{1}{18} + B \cdot \frac{5}{162} + C \cdot \frac{1}{54} + D \cdot \frac{1}{90}}{15^m}\right)^{\sfrac{1}{2}}.\\
    \delta_1 (E) &= \left\{ \begin{matrix}
    \frac{1}{15\cdot 5^m} && \mbox{if $A\neq 0$}\\
    \frac{11}{225\cdot5^m} && \mbox{if $A=0$ and $B\neq0$}\\
    \frac{1}{30\cdot5^m} && \mbox{if $A=0$, $B=0$, and $C\neq0$}\\
    \frac{1}{75\cdot5^m} && \mbox{A=B=C=0}
    \end{matrix} \right..
\end{align*}

The weights $\{p(x)\}$ can be calculated using part (c) of Corollary \ref{cor:3.9}. From this, $\delta(E,w)$ for the uniform weights $\{w(x)\}$ can be calculated if $R$ is known.

\begin{example} \label{ex:3.6}
$E=V_0 \cup \{ F_0 F_1 q_2, F_1 F_2 q_0, F_2 F_0 q_1 \}$ ($E$ consists of the three elements of $V_0$ and the three most interior points of $V_2$, as shown in Figure \ref{fig:s2}).
\end{example}

\begin{figure}[h]
\begin{framed}
\begin{center}
\begin{tikzpicture}

\draw[black] (0,0) -- (3,0) -- (1.5, 2.598) -- cycle;
\draw[black] (0,0) -- (-3,0) -- (-1.5, 2.598) -- cycle;
\draw[black] (-1.5, 2.598) -- (1.5, 2.598) -- (0, 5.196) -- cycle;

\draw[black] (0, 2.598) -- (.75, 3.897) -- (-.75, 3.897) -- cycle;
\draw[black] (-1.5, 0) -- (-.75, 1.299) -- (-2.25, 1.299) -- cycle;
\draw[black] (1.5, 0) -- (2.25, 1.299) -- (.75, 1.299) -- cycle;

\filldraw[black] (0, 5.196) circle (2pt);
\filldraw[black] (3, 0) circle (2pt);
\filldraw[black] (-3, 0) circle (2pt);

\filldraw[black] (0, 2.598) circle (2pt);
\filldraw[black] (-.75, 1.299) circle (2pt);
\filldraw[black] (.75, 1.299) circle (2pt);

\end{tikzpicture}
\end{center}

\caption{The set $E$ in Example \ref{ex:3.6}.}
\label{fig:s2}

\end{framed}
\end{figure}

The sample set $E$ in Example \ref{ex:3.5} can be thought of as very ``wide" (in that at a given level $k$, many $k$-cells are represented) but not very ``deep" (as the points of $E$ all come from $V_k$ for particularly small values of $k$). Given a finite number of points that we are allowed to pick for our sample set, some trade-off must necessarily be made between width and depth. In Example \ref{ex:3.6}, we choose a basic set that can be described as ``deeper" than the other sets of similar size we have considered so far, since it includes elements of $V_2 \setminus V_1$.

As usual, in order to calculate $\delta_0 (E)$, we consider the harmonic spline $g_{V_0}-g_E$. For all $x \in V_0$, $g_{V_0}(x)-g_E(x)=0-0=0$. For the interior values of $x \in E$, $g_E(x)=0$ since $x \in E$, and applying Lemma \ref{lem:3.1} (or Theorem \ref{thm:3.3}) gives $g_{V_0}(x)=\frac{1}{15}$. Therefore, the values of $g_{V_0}-g_E$ are as shown in Figure \ref{fig:v2spline}, where $a$ and $b$ are some constants. The function $g_{V_0}-g_E$ is harmonic away from $E$, so $a$ and $b$ must satisfy the average rule:
\begin{equation*}
    a=\frac{0+a+\frac{1}{15}+b}{4}, \quad \quad b=\frac{a+\frac{1}{15}+\frac{1}{15}+a}{4}.
\end{equation*}
Therefore, $a=\frac{1}{25}$ and $b=\frac{4}{75}$. Now we can calculate the integral $g_{V_0}-g_E$:
\begin{equation*}
    \int g_{V_0}-g_E d\mu = \frac{2}{45}.
\end{equation*}

So
\begin{equation*}
    (\delta_0 (E))^2 = \int g_E d\mu = \int g_{V_0} d\mu - \int g_{V_0} - g_E d\mu = \frac{1}{18} - \frac{2}{45} = \frac{1}{90}.
\end{equation*}
\begin{equation*}
    \delta_0(E)=\frac{1}{3\sqrt{10}}.
\end{equation*}

\begin{figure}[h]
\begin{framed}
\centering
\begin{minipage}{.45\textwidth}
    \centering
    \begin{tikzpicture}
    
    \draw[black] (0,0) -- (3,0) -- (1.5, 2.598) -- cycle;
\draw[black] (0,0) -- (-3,0) -- (-1.5, 2.598) -- cycle;
\draw[black] (-1.5, 2.598) -- (1.5, 2.598) -- (0, 5.196) -- cycle;

\draw[black] (0, 2.598) -- (.75, 3.897) -- (-.75, 3.897) -- cycle;
\draw[black] (-1.5, 0) -- (-.75, 1.299) -- (-2.25, 1.299) -- cycle;
\draw[black] (1.5, 0) -- (2.25, 1.299) -- (.75, 1.299) -- cycle;

\filldraw[black] (0, 0) circle (1 pt) node[anchor=north]{b};
\filldraw[black] (3, 0) circle (1 pt) node[anchor=north]{0};
\filldraw[black] (-3, 0) circle (1 pt) node[anchor=north]{0};
\filldraw[black] (1.5, 2.598) circle (1 pt) node[anchor=west]{ b};
\filldraw[black] (-1.5, 2.598) circle (1 pt) node[anchor=east]{b };
\filldraw[black] (0, 5.196) circle (1 pt) node[anchor=south]{0};

\filldraw[black] (0, 2.598) circle (1 pt) node[anchor=north]{$\frac{1}{15}$};
\filldraw[black] (.75, 3.897) circle (1 pt) node[anchor=west]{ a};
\filldraw[black] (-.75, 3.897) circle (1 pt) node[anchor=east]{a };

\filldraw[black] (-1.5, 0) circle (1 pt) node[anchor=north]{a};
\filldraw[black] (1.5, 0) circle (1 pt) node[anchor=north]{a};

\filldraw[black] (-.75, 1.299) circle (1 pt) node[anchor=west]{ $\frac{1}{15}$};
\filldraw[black] (.75, 1.299) circle (1 pt) node[anchor=east]{$\frac{1}{15}$ };
\filldraw[black] (-2.25, 1.299) circle (1 pt) node[anchor=east]{a };
\filldraw[black] (2.25, 1.299) circle (1 pt) node[anchor=west]{ a};

    \end{tikzpicture}
    \captionof{figure}{The values of $g_{V_0}-g_E$ on $V_2$, in the context of Example \ref{ex:3.6}. Constants $a$ and $b$ are as of yet unknown.}
    \label{fig:v2spline}
\end{minipage}\hfill\begin{minipage}{.45\textwidth}
    \centering
    \begin{tikzpicture}
    
    \draw[black] (0,0) -- (3,0) -- (1.5, 2.598) -- cycle;
\draw[black] (0,0) -- (-3,0) -- (-1.5, 2.598) -- cycle;
\draw[black] (-1.5, 2.598) -- (1.5, 2.598) -- (0, 5.196) -- cycle;

\draw[black] (0, 2.598) -- (.75, 3.897) -- (-.75, 3.897) -- cycle;
\draw[black] (-1.5, 0) -- (-.75, 1.299) -- (-2.25, 1.299) -- cycle;
\draw[black] (1.5, 0) -- (2.25, 1.299) -- (.75, 1.299) -- cycle;

\draw[line width=2pt] (-.75, 3.897) -- (.75, 3.897) -- (2.25, 1.299) -- (1.5, 0) -- (-1.5, 0) -- (-2.25, 1.299) -- cycle;

    \end{tikzpicture}
    \captionof{figure}{In example \ref{ex:3.6}, the biharmonic function $g_E$ obtains its maximum value, $\frac{1}{75}$, along these shaded lines.}
    \label{fig:shadedS2}
\end{minipage}
\end{framed}
\end{figure}

By the same inductive arguments that were used to prove Theorems \ref{thm:3.3} and \ref{thm:3.4} and Corollary \ref{cor:3.5}, $g_E$ obtains its maximum value, $\frac{1}{75}$ along the shaded lines in Figure \ref{fig:shadedS2}, so
\begin{equation*}
    \delta_1 (E) = \frac{1}{75}.
\end{equation*}
Now we calculate the weights $p(x)=\int v_x d\mu$ for $x \in E$. By symmetry, there are only two weights to calculate: $p(q_0)$ and $p(F_0 F_1 q_2)$. The indicators are shown in Figures \ref{fig:s2v0} and \ref{fig:s2v012}.

\begin{figure}[!htb]
\begin{framed}
\centering
\begin{minipage}{.45\textwidth}
    \centering
    \begin{tikzpicture}
    
    \draw[black] (0,0) -- (3,0) -- (1.5, 2.598) -- cycle;
\draw[black] (0,0) -- (-3,0) -- (-1.5, 2.598) -- cycle;
\draw[black] (-1.5, 2.598) -- (1.5, 2.598) -- (0, 5.196) -- cycle;

\draw[black] (0, 2.598) -- (.75, 3.897) -- (-.75, 3.897) -- cycle;
\draw[black] (-1.5, 0) -- (-.75, 1.299) -- (-2.25, 1.299) -- cycle;
\draw[black] (1.5, 0) -- (2.25, 1.299) -- (.75, 1.299) -- cycle;

\filldraw[black] (0, 0) circle (1 pt) node[anchor=north]{$\frac{1}{265}$};
\filldraw[black] (3, 0) circle (1 pt) node[anchor=north]{0};
\filldraw[black] (-3, 0) circle (1 pt) node[anchor=north]{0};
\filldraw[black] (1.5, 2.598) circle (1 pt) node[anchor=west]{ $\frac{26}{265}$};
\filldraw[black] (-1.5, 2.598) circle (1 pt) node[anchor=east]{$\frac{26}{265}$ };
\filldraw[black] (0, 5.196) circle (1 pt) node[anchor=south]{1};

\filldraw[black] (0, 2.598) circle (1 pt) node[anchor=north]{0};
\filldraw[black] (.75, 3.897) circle (1 pt) node[anchor=west]{ $\frac{97}{265}$};
\filldraw[black] (-.75, 3.897) circle (1 pt) node[anchor=east]{$\frac{97}{265}$ };

\filldraw[black] (-1.5, 0) circle (1 pt) node[anchor=north]{$\frac{2}{265}$};
\filldraw[black] (1.5, 0) circle (1 pt) node[anchor=north]{$\frac{2}{265}$};

\filldraw[black] (-.75, 1.299) circle (1 pt) node[anchor=west]{ 0};
\filldraw[black] (.75, 1.299) circle (1 pt) node[anchor=east]{0 };
\filldraw[black] (-2.25, 1.299) circle (1 pt) node[anchor=east]{$\frac{7}{265}$ };
\filldraw[black] (2.25, 1.299) circle (1 pt) node[anchor=west]{ $\frac{7}{265}$};

    \end{tikzpicture}
    \captionof{figure}{$v_{q_0}$ in Example \ref{ex:3.6}.}
    \label{fig:s2v0}
\end{minipage}\hfill\begin{minipage}{.45\textwidth}
    \centering
    \begin{tikzpicture}
    
    \draw[black] (0,0) -- (3,0) -- (1.5, 2.598) -- cycle;
\draw[black] (0,0) -- (-3,0) -- (-1.5, 2.598) -- cycle;
\draw[black] (-1.5, 2.598) -- (1.5, 2.598) -- (0, 5.196) -- cycle;

\draw[black] (0, 2.598) -- (.75, 3.897) -- (-.75, 3.897) -- cycle;
\draw[black] (-1.5, 0) -- (-.75, 1.299) -- (-2.25, 1.299) -- cycle;
\draw[black] (1.5, 0) -- (2.25, 1.299) -- (.75, 1.299) -- cycle;

\filldraw[black] (0, 0) circle (1 pt) node[anchor=north]{$\frac{4}{265}$};
\filldraw[black] (3, 0) circle (1 pt) node[anchor=north]{0};
\filldraw[black] (-3, 0) circle (1 pt) node[anchor=north]{0};
\filldraw[black] (1.5, 2.598) circle (1 pt) node[anchor=west]{ $\frac{104}{265}$};
\filldraw[black] (-1.5, 2.598) circle (1 pt) node[anchor=east]{$\frac{104}{265}$ };
\filldraw[black] (0, 5.196) circle (1 pt) node[anchor=south]{0};

\filldraw[black] (0, 2.598) circle (1 pt) node[anchor=north]{1};
\filldraw[black] (.75, 3.897) circle (1 pt) node[anchor=west]{ $\frac{123}{265}$};
\filldraw[black] (-.75, 3.897) circle (1 pt) node[anchor=east]{$\frac{123}{265}$ };

\filldraw[black] (-1.5, 0) circle (1 pt) node[anchor=north]{$\frac{28}{265}$};
\filldraw[black] (1.5, 0) circle (1 pt) node[anchor=north]{$\frac{28}{265}$};

\filldraw[black] (-.75, 1.299) circle (1 pt) node[anchor=west]{ 0};
\filldraw[black] (.75, 1.299) circle (1 pt) node[anchor=east]{0 };
\filldraw[black] (-2.25, 1.299) circle (1 pt) node[anchor=east]{$\frac{8}{265}$ };
\filldraw[black] (2.25, 1.299) circle (1 pt) node[anchor=west]{ $\frac{8}{265}$};

    \end{tikzpicture}
    \captionof{figure}{$v_{F_0 F_1 q_2}$ in Example \ref{ex:3.6}.}
    \label{fig:s2v012}
\end{minipage}
\end{framed}
\end{figure}

The weights are
\begin{align*}
    & p(q_0)=p(q_1)=p(q_2)=\sfrac{1}{9},\\
    & p(F_0 F_1 q_2) = p(F_1 F_2 q_0) = p(F_2 F_0 q_1) = \sfrac{2}{9}.
\end{align*}
If $\{ w(x) \}_{x \in E}$ are the uniform weights,
\begin{equation*}
    \delta(E, w)=\left( 3 \left| \frac{1}{9}-\frac{1}{6} \right| + 3 \left| \frac{2}{9} - \frac{1}{6} \right| \right)R^{\sfrac{1}{2}} = \frac{1}{3} R^{\sfrac{1}{2}}.
\end{equation*}

Interestingly, $V_1$ is another highly symmetric $6$-element sample set and $\delta_0 (E)=\delta_0 (V_1)$, $\delta_1 (E)=\delta_1 (V_1)$, and when $\{w(x)\}$ are the uniform weights for each set $E$, $\delta(E, w)=\delta(V_1, w)$. Therefore, by taking $E$ rather $V_1$ as our sample set (choosing depth over width), it is not clear whether we would be making a better or a worse choice.

We briefly mention one more family of sample sets $\tilde{E}_m$. For a fixed $m$, $\tilde{E}_m$ consists of $F_w x$ for all $|w|=m$, $x \in E$ (where $E$ is still the sample set in Example \ref{ex:3.6}). Because the values of $\delta_0$, $\delta_1$, and $\delta(E, w)$ (for the uniform weights $\{w(x)\}$) are the same for $E$ and $V_1$, Corollary \ref{cor:3.9} tells us that they will continue to be the same for $V_{m+1}$ and $\tilde{E}_m$, for all $m$.

\begin{align*}
    & \delta_0 (\tilde{E}_m) = \delta_1 (V_{m+1}) = \frac{1}{3\sqrt{2 \cdot 5^{m+1}}}. \\
    & \delta_1 (\tilde{E}_m) = \delta_1 (V_{m+1}) = \frac{1}{15\cdot 5^{m+1}}. \\
    & \delta(\tilde{E}_m, w) = \delta(V_{m+1}, w) = \frac{2(3^m-1)}{3^m+1} R^{\sfrac{1}{2}}.
\end{align*}
To calculate the weights $\{p(x)\}$ for $\tilde{E}_m$, notice that the harmonic spline indicators for $\tilde{E}_m$ are the indicators for $E$ but with $F_w ^{-1}$ for some $|w|=m$. (For those elements $x \in \tilde{E}_m$ that are shared between two distinct $m$-cells $F_w SG$ and $F_{w'} SG$, the indicator for $x$ in $\tilde{E}_m$ is the sum of two indicators of $E$, one composed with $F_w ^{-1}$ and the other composed with $F_{w'} ^{-1}$.) Thus

\begin{equation*}
    p_{\tilde{E}_m}(x) = \left\{\begin{matrix}
    \frac{1}{9 \cdot 3^m} & : x \in V_0 \\
    \frac{2}{9 \cdot 3^m} & : x \in (V_m \setminus V_0) \\
    \frac{2}{9 \cdot 3^m} & : \mbox{$x$ is one of the interior points of an $m$-cell}
    \end{matrix}\right..
\end{equation*}

\section{Other self-similar measures}

In this section, we apply the results of Section 2 to more fractals: The Sierpi\'{n}ski tetrahedron ($ST$) and the $3$-level gasket  ($SG_3$), both of which will be covered in less depth than the Sierpi\'{n}ski gasket was in Section 3. Like in Section 3, our starting point is using \cite{splines} to determine the values of $g_{V_0}$ on $V_1$ for these fractals. However, whereas \cite{splines} gives us these values directly for $SG$, it does not for $ST$ or $SG_3$. Therefore, we will have to apply the general algorithm of Section 2 of \cite{splines} in its entirety to $ST$ and $SG_3$. We begin this section with a summary of that algorithm. We slightly modify the notation and indexing of \cite{splines} to be consistent with our own and to be the most useful for our purposes.

Let $K$ be a p.c.f. self-similar fractal with boundary $V_0 = \{ q_k \}_{0 \leq k < N_0}$ generated by a set of contractions $\{ F_i \}_{0 \leq i < N}$, for some $N_0$ and $N$. For the fractals we consider in this paper, it will help to add the simplifying assumption that $N_0 \leq N$ and each $q_k$ is the fixed point of $F_k$. For $m$, let $V_m=\{F_w x : |w|=m, x \in V_0\}$, and let $V_* = \bigcup_m V_m$. Let $K$ have a regular harmonic structure with Dirichlet form $\energy$ on $V_1$ satisfying

\begin{equation*}
    \energy (u, v) = \sum_{i=0}^{N-1} r_i ^{-1} \energy(u \circ F_i, v \circ F_i)
\end{equation*}
and a self-similar probability measure $\mu$ satisfying

\begin{equation*}
    \mu = \sum_{i=0}^{N-1} \mu_i (\mu \circ F_i).
\end{equation*}
Let $\laplace$ be the associated Laplacian. For all $j$, let $\mathcal{H}_j = \{ f : \laplace^{j+1} f = 0\}$. An easy basis for $\mathcal{H}_j$ is $\{ f_{lk} \}_{0 \leq l \leq j, 0 \leq k<N_0}$, where $f_{lk}$ is the solution to

\begin{equation*}
    \laplace ^m f_{lk} (q_n) = \delta_{ml} \delta_{kn} \quad \quad \mbox{for all $m, n$ such that $0 \leq m \leq l$ and $0 \leq n<N_0$}.
\end{equation*}
Define the harmonic functions $h_i$ as usual such that $h_i (q_k) = \delta_{ik}$. (This means that $h_i = f_{0i}$ for all $i \in \{0, 1, 2, \dots, N_0-1\}$.) For all $k, k', n, n' \in \{0, 1, 2, N_0-1\}$, let

\begin{equation} \label{eq:4.1}
    A(kk', nn') = \sum_{i=0}^{N-1} \mu_i h_k (F_i q_n) h_{k'} (F_i q_{n'}).
\end{equation}
It is a result \cite{splines} that if

\begin{equation*}
    I(kk') = \sum_{i=0}^{N-1} \mu_i \int (h_k \circ F_i) (h_{k'} \circ F_i) d\mu,
\end{equation*}
then the vector $I(kk')$ is an eigenvector of the matrix $A(kk', nn')$ corresponding to eigenvalue $1$, and

\begin{equation*}
    \sum_{k=0}^{N_0-1} \sum_{k'=0}^{N_0-1} I(kk') = 1.
\end{equation*}
It is easy to compute $A(kk', nn')$ for any example $K$ (such as $ST$ and $SG_3$), so $I(kk')$ can be determined.

Let $X$ be the matrix whose rows and columns are indexed by the elements of $V_1 \setminus V_0$, such that

\begin{equation} \label{eq:4.2}
    X_{pq} = \energy(v_p, v_q)
\end{equation}
where $v_p$ and $v_q$ are harmonic $1$-splines such that $v_p(r)=\delta_{pr}$ and $v_q(r)=\delta_{qr}$. Let $G=X^{-1}$. For all $i, i' \in \{0, 1, 2, \dots, N-1\}$ and $n, n' \in \{0, 1, 2, \dots, N_0-1\}$, let

\begin{equation*}
    \gamma(i, i', n, n') = \left\{ \begin{matrix}
        G_{F_i q_n, F_{i'} q_{n'}} & : \mbox{ if $F_i q_n, F_{i'} q_{n'} \in (V_1 \setminus V_0)$} \\
        0 & : \mbox{otherwise}
    \end{matrix} \right..
\end{equation*}
Finally, it is another result in \cite{splines} that

\begin{equation} \label{eq:4.4}
    f_{1k}(F_i q_n) = \sum_{i'=0}^{N-1} \sum_{n'=0}^{N_0-1} \sum_{k'=0}^{N_0-1} - \mu_{i'} \gamma(i, i', n, n') I(k'n') h_k (F_{i'} q_{k'}).
\end{equation}

After using this recipe to calculate the values of $f_{1k}$ on $V_1$ for our fractal $K$, the values of $g_{V_0}$ on $V_1$ can be determined. $\laplace g_{V_0} = -1$, so

\begin{equation*}
    g_{V_0} = - \sum_{k=0}^{N_0-1} f_{1k}.
\end{equation*}
We will then require three more results, Lemma \ref{lem:4.1}, Theorem \ref{thm:4.3}, and Corollary \ref{cor:generalK}. These are the generalizations of Lemma \ref{lem:3.1}, Theorem \ref{thm:3.8}, and Corollary \ref{cor:3.9} respectively.

\begin{lemma} \label{lem:4.1}
If $u$ is a function on $K$ with $\laplace u = -1$, $|w|=m$, $u \circ F_w = v$, $\tilde{v}$ is the harmonic function with the same values on $V_0$ as $v$, and $x \in (V_1 \setminus V_0)$, then

\begin{equation*}
    u(F_w x) = v(x) = \tilde{v}(x) + \mu_w r_w g_{V_0}(x)
\end{equation*}
(Recall that if $w=w_1 w_2 \dots w_m$, $\mu_w = \mu_1 \mu_2 \cdots \mu_m$ and $r_w = r_1 r_2 \cdots r_m$.)

\end{lemma}

\begin{proof}
The proof is essentially the same as that of Lemma \ref{lem:3.1}.

\begin{equation*}
    \laplace v = \laplace(u \circ F_w) = \mu_w r_w (\laplace u) \circ F_w = \mu_w r_w \cdot (-1)
\end{equation*}
so

\begin{equation*}
    \laplace (\mu_w^{-1} r_w^{-1} v) = -1.
\end{equation*}
$\mu_w^{-1} r_w^{-1} \tilde{v}$ is harmonic and has the same values on the boundary as $\mu_w^{-1} r_w^{-1} v$, while $g_{V_0}$ has the same Laplacian as $\mu_w^{-1} r_w^{-1} v$ but vanishes on the boundary. Thus,

\begin{align*}
    \mu_w^{-1} r_w^{-1} v &= \mu_w^{-1} r_w^{-1} \tilde{v} + g_{V_0}\\
    v &= \tilde{v} + \mu_w r_w g_{V_0}\\
    v(x) &= \tilde{v}(x) + \mu_w r_w g_{V_0}(x)\\
    u(F_w x) &= \tilde{v}(x) + \mu_w r_w g_{V_0}(x).
\end{align*}

\end{proof}

\begin{theorem} \label{thm:4.3}
If $V_m \subseteq E$, then

\begin{equation} \label{eq:thm4.3}
    G_E (x, y) = \left\{ \begin{matrix}
        r_m G_{E_w} (F_w^{-1}x, F_w^{-1}y) & : \mbox{ if $x, y \in F_w K$}\\
        0 & : \mbox{ if $x$ and $y$ belong to separate $m$-cells}
    \end{matrix}\right..
\end{equation}

\end{theorem}

\begin{proof}
Let

\begin{equation*}
    a(x, y) = \left\{ \begin{matrix}
        G_{E_w}( F_w^{-1} x, F_w^{-1} y) & : \mbox{ if $x, y \in F_w K, |w|=m$}\\
        0 & : \mbox{ if $x$ and $y$ belong to separate $m$-cells}
    \end{matrix} \right..
\end{equation*}
Fix $u \in \dom\laplace$ such that $u|_{V_m}=0$. For all $x \in K$, if $F_w K$ is the $m$-cell that $x$ belongs to,

\begin{align*}
    -\int a(x, y) \laplace u(y) d\mu(y) =& -\sum_{|w'|=m} \int_{F_w K} a(x, y) \laplace u(y) d\mu(y) \\
    =& -\sum_{|w'|=m, w' \neq w} \int_{F_w K} 0 \laplace u (y) d\mu(y)\\
    &- \int_{F_w K} G_{E_w} (F_w^{-1} x, F_w^{-1} y) \laplace u (y) d\mu(y)\\
    =& -\mu_w \int G_{E_w} (F_w^{-1} x, y') \laplace u (F_w y') d\mu(y').
\end{align*}
Note that $\laplace (u \circ F_w) = r_w \mu_w \laplace u \circ F_w$, so $r_w^{-1} \mu_w^{-1} \laplace(u \circ F_w) = \laplace u \circ F_w$. Therefore, this becomes

\begin{align*}
    -\int a(x, y) \laplace u(y) d\mu(y) &= -\mu_w \int_K G_{E_w} (F_w^{-1} x, y') [\laplace u \circ F_w] (y') d\mu(y') \\
    &= -\mu_w \int G_{E_w} (F_w^{-1} x, y') (r_w^{-1} \mu_w^{-1}) (\laplace(u \circ F_w))(y') d\mu(y')\\
    &= - r_w^{-1} \int G_{E_w} (F_w^{-1}x, y') (\laplace (u \circ F_w))(y') d\mu(y').
\end{align*}
$u \circ F_w$ is in $\dom \laplace$ and vanishes on $E_w$ (because $u$ vanishes on the boundary of $E$) so this becomes

\begin{align*}
    -\int a(x, y) d\mu(y) &= r_w^{-1} (u \circ F_w)(F_w^{-1} x) \\
    -\int r_w a(x, y) d\mu(y) &= u(x).
\end{align*}
This holds for all $u \in \dom\laplace$, so

\begin{equation} \label{eq:usedinthm4.3}
    G_E(x, y) = r_w a(x, y).
\end{equation}
 \eqref{eq:usedinthm4.3} is equivalent to \eqref{eq:thm4.3}.

\end{proof}

For Corollary \ref{cor:generalK}, we bring back the notations $E_w$ and $p_{\tilde{E}}(x)$ from Section 3.

\begin{corollary} \label{cor:generalK}
If $V_m \subseteq E$,

(a)

\begin{equation} \label{eq:4.7}
    \delta_0(E)=\left( \sum_{|w|=m} \mu_w^2 r_w (\delta_0(E_w))^2 \right)^{\sfrac{1}{2}}.
\end{equation}

(b)

\begin{equation*}
    \delta_1(E)=\sup_{|w|=m} \mu_w r_w \delta_1(E_w).
\end{equation*}

(c)

\begin{equation} \label{eq:4.9}
    p_E(x) = \sum_{|w|=m, x \in F_w K} \mu_w p_{E_w}(F_w^{-1}x).
\end{equation}

\end{corollary}

\begin{proof}

By Theorem \ref{thm:4.3},

\begin{equation*}
    G_E (x, y) = \left\{ \begin{matrix}
        r_w G_{E_w} (F_w^{-1}x, F_w^{-1}y) & : \mbox{ if $x, y \in F_w K$}\\
        0 & : \mbox{ if $x$ and $y$ belong to separate $m$-cells}
    \end{matrix}\right..
\end{equation*}

To find $\delta_0$, we take the square-root of $\int g_E d\mu$. For all $x \in K$, if $F_w K$ is the $m$-cell to which $x$ belongs,

\begin{equation*}
    g_E(x) = \int G_E(x, y) d\mu(y)
\end{equation*}
so by Theorem \ref{thm:4.3},

\begin{align*}
    g_E(x) &= \int_{F_w K} r_w G_{E_w} (F_w^{-1} x, F_w^{-1} y) d\mu(y)\\
    &= \mu_w \int_{K} r_w G_{E_w} (F_w^{-1}x, F_w^{-1}y) d\mu(y')
\end{align*}
or

\begin{equation} \label{eq:4.8}
    g_E(x) = \mu_w r_w g_{E_w} (F_w^{-1}x).
\end{equation}
Therefore,

\begin{align*}
    \int g_E(x) d\mu(x) &= \sum_{|w|=m} \int_{F_w K} g_E(x) d\mu(x)\\
    &= \sum_{|w|=m} \mu_w r_w \int_{F_w K} g_{E_w} (F_w^{-1}x)\\
    &= \sum_{|w|=m} \mu_w r_w \cdot \mu_w \int_{K} g_{E_w} (x') d\mu(x')\\
    &= \sum_{|w|=m} \mu_w^2 r_w \int g_E d\mu.
\end{align*}
Taking the square-root of both sides yields \eqref{eq:4.7}.

By \eqref{eq:4.8},

\begin{align*}
    \delta_1(E) = \sup_{x \in K} g_E(x) &= \sup_{|w|=m} \sup_{x \in F_w K} \mu_w r_w g_{E_w} (F_w^{-1}x)\\
    &= \sup_{|w|=m} \sup_{x\in K} \mu_w r_w g_{E_w} (x)\\
    &= \sup_{|w|=m} \mu_w r_w \sup_{x\in K} g_{E_w} (x)\\
    &= \sup_{|w|=m} \mu_w r_w \delta_1(E_w).
\end{align*}

The weights are as in \eqref{eq:4.9} because for each cell $F_w ST$ containing $x$, if $v_x$ is the indicator for $F_w^{-1}x$ with respect to $E_w$, the contribution to this cell to the weight of $x$ with respect to $E$ is $\mu_w \int v_x d\mu = \mu_w p_{E_w}(F_w^{-1}x)$.

\end{proof}

We now apply these results to the Sierpi\'{n}ski tetrahedron ($ST$). Recall that $ST$ is generated by the four similarities in $\mathbb{R}^3$ with contraction ratio $\frac{1}{2}$ and fixed points the vertices of a regular tetrahedron. For $ST$, $N=4$, $N_0=4$, $\mu_i=\frac{1}{4}$, and $r_i = \frac{2}{3}$. The values of the harmonic functions on $V_1$ are

\begin{equation*}
    h_j (F_i q_k) = \left\{ \begin{matrix}
        1 & : \mbox{ if $i=j=k$} \\
        0 & : \mbox{if $j \neq i = k$}\\
        \sfrac{1}{3} & : \mbox{if $j\neq k$ and ($i=j$ or $i=k$)}\\
        \sfrac{1}{6} & : \mbox{$i, j, k$ all distinct}
    \end{matrix} \right..
\end{equation*}
Let us index $A(kk', nn')$ and $I(kk')$ by the ordering

\begin{equation*}
    kk'<nn' \Longleftrightarrow \big( \mbox{$k<n$ or ($k=n$ and $k'<n'$)}\big).
\end{equation*}
By \eqref{eq:4.1}, $A(kk', nn')$ is

\begin{equation*}
    \frac{1}{144} \begin{pmatrix}
        48 & 16 & 16 & 16 & 16 & 6 & 5 & 5 & 16 & 5 & 6 & 5 & 16 & 5 & 5 & 6\\
        8 & 32 & 12 & 12 & 2 & 8 & 3 & 3 & 3 & 12 & 5 & 4 & 3 & 12 & 4 & 5\\
        8 & 12 & 32 & 12 & 3 & 5 & 12 & 4 & 2 & 3 & 8 & 3 & 3 & 4 & 12 & 5\\
        8 & 12 & 12 & 32 & 3 & 5 & 4 & 12 & 3 & 4 & 5 & 12 & 2 & 3 & 3 & 8\\
        8 & 2 & 3 & 3 & 32 & 8 & 12 & 12 & 12 & 3 & 5 & 4 & 12 & 3 & 4 & 5\\
        6 & 16 & 5 & 5 & 16 & 48 & 16 & 16 & 5 & 16 & 6 & 5 & 5 & 16 & 5 & 6\\
        5 & 3 & 12 & 4 & 12 & 8 & 32 & 12 & 3 & 2 & 8 & 3 & 4 & 3 & 12 & 5\\
        5 & 3 & 4 & 12 & 12 & 8 & 12 & 32 & 4 & 3 & 5 & 12 & 3 & 2 & 3 & 8\\
        8 & 3 & 2 & 3 & 12 & 5 & 3 & 4 & 32 & 12 & 8 & 12 & 12 & 4 & 3 & 5\\
        5 & 12 & 3 & 4 & 3 & 8 & 2 & 3 & 12 & 32 & 8 & 12 & 4 & 12 & 3 & 5\\
        6 & 5 & 16 & 5 & 5 & 6 & 16 & 5 & 16 & 16 & 48 & 16 & 5 & 5 & 16 & 6\\
        5 & 4 & 3 & 12 & 4 & 5 & 3 & 12 & 12 & 12 & 8 & 32 & 3 & 3 & 2 & 8\\
        8 & 3 & 3 & 2 & 12 & 5 & 4 & 3 & 12 & 4 & 5 & 3 & 32 & 12 & 12 & 8\\
        5 & 12 & 4 & 3 & 3 & 8 & 3 & 2 & 4 & 12 & 5 & 3 & 12 & 32 & 12 & 8\\
        5 & 4 & 12 & 3 & 4 & 5 & 12 & 3 & 3 & 3 & 8 & 2 & 12 & 12 & 32 & 8\\
        6 & 5 & 5 & 16 & 5 & 6 & 5 & 16 & 5 & 5 & 6 & 16 & 16 & 16 & 16 & 48
    \end{pmatrix}.
\end{equation*}
By taking $I(kk')$ the eigenvector of magnitude $1$ corresponding to eigenvalue $1$,

\begin{equation*}
    I(kk')=\left\{ \begin{matrix}
        \sfrac{7}{80} & : \mbox{ if $k=k'$} \\
        \sfrac{13}{240} & : \mbox{ if $k \neq k'$}
    \end{matrix} \right..
\end{equation*}
By computing the energies $\energy_1(v_p, v_q)$, and indexing the rows and columns by the ordering $F_0 q_1 < F_0 q_2 < F_0 q_3 < F_1 q_2 < F_1 q_3 < F_2 q_3$,

\begin{equation*}
    X = \frac{1}{2} \begin{pmatrix}
        18 & -3 & -3 & -3 & -3 & 0\\
        -3 & 18 & -3 & -3 & 0 & -3\\
        -3 & -3 & 18 & 0 & -3 & -3\\
        -3 & -3 & 0 & 18 & -3 & -3\\
        -3 & 0 & -3 & -3 & 18 & -3\\
        0 & -3 & -3 & -3 & -3 & 18
    \end{pmatrix}
\end{equation*}
so

\begin{equation*}
    G=X^{-1}=\frac{1}{72} \begin{pmatrix}
        10 & 3 & 3 & 3 & 3 & 2\\
        3 & 10 & 3 & 3 & 2 & 3\\
        3 & 3 & 10 & 2 & 3 & 3\\
        3 & 3 & 2 & 10 & 3 & 3\\
        3 & 2 & 3 & 3 & 10 & 3\\
        2 & 3 & 3 & 3 & 3 & 10
    \end{pmatrix}
\end{equation*}
or
\begin{equation*}
    \gamma(i, i', n, n')=\left\{ \begin{matrix}
        0 & : \mbox{ if $F_i q_n \in V_0$ or $F_{i'} q_{n'} \in V_0$}\\
        \sfrac{10}{72} & : \mbox{ if $F_i q_n = F_{i'} q_{n'} \in (V_1 \setminus V_0)$}\\
        \sfrac{2}{72} & : \mbox{ if $\{i, i', n, n'\} = \{0, 1, 2, 3\}$}\\
        \sfrac{3}{72} & : \mbox{ otherwise}
    \end{matrix}\right..
\end{equation*}
All that remains is to plug into \eqref{eq:4.4}. This yields:

\begin{equation*}
    f_{1j} (F_i q_k) = \left\{ \begin{matrix}
        0 & : \mbox{ if $i=k$}\\
        \sfrac{-5}{432} & : \mbox{ if $i \neq k$, and ($j=i$ or $j=k$)}\\
        \sfrac{-4}{432} & : \mbox{ if $i, j, k$ all distinct}
    \end{matrix} \right..
\end{equation*}
We now proceed to calculate the weights and discrepancies for some sample sets $E$ (where $K=ST$).

\begin{example} \label{ex:4.1}
    $K=ST$, $E=V_0$.
\end{example}

$g_{V_0}=-f_{10}-f_{11}-f_{12}-f_{13}$, so for $x \in V_1$:

\begin{equation*}
    g_{V_0} (x) = \left\{ \begin{matrix}
        \sfrac{-1}{16} & : \mbox{ if $x \in (V_1 \setminus V_0)$}\\
        0 & : \mbox{ if $x \in V_0$}
    \end{matrix} \right..
\end{equation*}
By applying Lemma \ref{lem:4.1} to $ST$, if $u$ is a function on $ST$ with $\laplace u = -1$, $w$ is a word of length $m$, $\{i, j, k, l\} = \{0, 1, 2, 3\}$, $u(F_w q_i)=a$, $u(F_w q_j)=b$, $u(F_w q_k)=c$, and $u(F_w q_l)=d$, then

\begin{equation}
    u(F_w F_i q_j) = \frac{2a+2b+c+d}{6} + \frac{1}{16\cdot6^m}.
\end{equation}
It follows from this that

\begin{equation} \label{h_functions}
    g_{V_0} = \sum_{m=0}^\infty h_m,
\end{equation}
where $h_m$ is the $(m+1)$-spline such that for all $x \in V_{m+1}$,

\begin{equation*}
    h_m(x)=\left\{ \begin{matrix}
        \frac{1}{16\cdot 6^m} & : \mbox{ if $x \in (V_{m+1} \setminus V_m)$}\\
        0 & : \mbox{ if $x \in V_m$}
    \end{matrix}\right..
\end{equation*}

\begin{theorem} \label{thm:4.2}
If $E=V_0$, then the weights $\{p(x)\}$ are the uniform weights, $\delta(E, w)=0$,

\begin{equation*}
    \delta_0(E)=\frac{3}{4\sqrt{10}}
\end{equation*}
and
\begin{equation*}
    \delta_1(E)=\frac{1}{16}.
\end{equation*}
\end{theorem}

\begin{proof}
By symmetry, the weights are equal, so $p(x)=\sfrac{1}{3}$ and $\delta(E, w)=0$. For each $m$,

\begin{equation*}
    \int h_m d\mu = \sum_{|w|=m} \left( \frac{1}{4} \right)^m \left( \frac{3 \cdot \frac{1}{16\cdot6^m}+0}{4} \right) = \frac{3}{64\cdot 6^m}.
\end{equation*}
By \eqref{h_functions},

\begin{equation*}
    \int g_{V_0} d\mu = \sum_{m=0}^\infty \frac{3}{64\cdot6^m}=\frac{9}{160}.
\end{equation*}
By Definition \ref{def:2.1},

\begin{equation*}
    \delta_0 (V_0) = \left( \int g_{V_0} d\mu \right)^{\sfrac{1}{2}} = \frac{3}{4\sqrt{10}}.
\end{equation*}

For $\delta_1$, we will first show by induction that for all $m\geq 1$, for all $|w|=m$ such that the character $0$ does not occur in $w$, $F_{0w}(q_1)=F_{0w}(q_2)=F_{0w}(q_3)=\frac{1}{16}$, $F_{0w}(q_0)=\frac{1}{16}-\frac{1}{16\cdot 6^m}$, and $g_{V_0}$ attains its supremum in $F_{0w}ST$.

\textit{Base case}: Let $m=1$. Let $j$ be the one character of $w$. $F_{0j} (q_j) = F_0 (q_j) = \sfrac{1}{16}$. Let $\{k, l\} = \{0, 1, 2, 3\} \setminus \{0, j\}$. 

\begin{align*}
    F_{0j} (q_j) &= F_0 (q_j) = \frac{1}{16}.\\
    &\\
    F_{0j} (q_k) &= \left( \frac{2F_0(q_j)+2F_0(q_k)+F_0(q_0)+F_0(q_l)}{6}+\frac{1}{16\cdot6} \right)\\
    &= \left(\frac{5}{6}\cdot\frac{1}{16}+\frac{1}{6}\cdot0+\frac{1}{16\cdot6}\right)=\frac{1}{16} \\
    &\\
    F_{0j} (q_l) &= F_{0j} (q_k) = \frac{1}{16}.\\
    &\\
    F_{0j} (q_0) &= \left( \frac{2F_0 (q_0) + 2F_0(q_j) + F_0 (q_k) + F_0 (q_l)}{6} + \frac{1}{16\cdot6} \right)\\
    &= \left( \frac{4}{6} \cdot \frac{1}{16} + \frac{2}{6} \cdot 0 +\frac{1}{16\cdot 6} \right) = \frac{5}{96} = \left( \frac{1}{16} - \frac{1}{16\cdot 6} \right).
\end{align*}
All $2$-cells of the form $F_{ii}ST$ are symmetric, as are all $2$-cells of the form $F_{ii'}$ such that $i \neq i'$. Therefore, $g_{V_0}$ must attain its supremum either on all cells $F_{ii}ST$, on all cells $F_{ii'}ST$ ($i \neq i'$) or both. The values on the boundary of $F_{ii'} ST$ ($\sfrac{5}{96}$, $\sfrac{1}{16}$, $\sfrac{1}{16}$, and $\sfrac{1}{16}$) are greater than those on the boundary of $F_{ii}ST$ ($0$, $\sfrac{5}{96}$, $\sfrac{5}{96}$, and $\sfrac{5}{96}$) so the supremum must be attained in $F_{ii'}ST$. If we let $i=0, i'=j$, $g_{V_0}$ attains its supremum on $F_{0j} ST$. Thus, the result holds for the base case $m=1$.

\textit{Inductive step}: Suppose the result holds for $m$. Consider $w$, a word of length $m$ with no $0$s, and $j$ an element of $\{1, 2, 3\}$. Then if $\{k, l\} = \{1, 2, 3\} \setminus \{j\}$:

\begin{align*}
    F_{0wj} (q_j) &= F_{0w} (q_j) = \frac{1}{16}.\\
    &\\
    F_{0wj} (q_k) &= \left( \frac{2F_{0w}(q_j)+2F_{0w}(q_k)+F_{0w}(q_0)+F_{0w}(q_l)}{6} + \frac{1}{16\cdot6^{m+1}} \right)\\
    &= \left( \frac{\frac{2}{16}+\frac{2}{16}+\left(\frac{1}{16}-\frac{1}{16\cdot6^m}\right)+\frac{1}{16}}{6} + \frac{1}{16\cdot6^{m+1}} \right) \\
    &= \left( \frac{(6/16)}{6} - \frac{1}{16 \cdot 6^{m+1}} + \frac{1}{16 \cdot 6^{m+1}} \right) = \frac{1}{16}.\\
    &\\
    F_{0wj} (q_l) &= F_{0wj} (q_k) = \frac{1}{16}.\\
    &\\
    F_{0wj} (q_0) &= \left( \frac{2F_{0w}(q_0) + 2F_{0w}(q_j) + F_{0w} (q_k) + F_{0w} (q_l)}{6} + \frac{1}{16\cdot 6^{m+1}} \right)\\
    &= \left( \frac{2 \left( \frac{1}{16} - \frac{1}{16\cdot6^m} \right) + \frac{2}{16}+\frac{1}{16}+\frac{1}{16}}{6} + \frac{1}{16\cdot6^{m+1}} \right)\\
    &= \left( \frac{1}{16} - \frac{1}{16 \cdot 6^{m+1}} \right).
\end{align*}
By the inductive hypothesis, $g_{V_0}$ attains its supremum on $F_{0w} ST$. The cells $F_{0wj} ST$, $F_{0wk} ST$, and $F_{0wl} ST$ are symmetrical, and have boundary values greater than those of $F_{0w0} ST$, so $g_{V_0}$ attains its supremum on each of them, including $F_{0wj} ST$. This completes the inductive step.

Thus, if we let $\{w_{(n)}\}_{n \in \mathbb{N}}$ be any sequence of words such that each $w_{(n)}$ has length $n$, and the leading $k$-character substring of $w_{(n)}$ is $w_{(k)}$ for all $k\leq n$, then

\begin{align*}
    \lim_{n \rightarrow \infty} g_{V_0} (F_{w_{(n)}}) &= \sup_{x \in ST} g_{V_0} (x).\\
    &\\
    \lim_{n \rightarrow \infty} \frac{1}{16} &= \delta_1(V_0).\\
    &\\
    \frac{1}{16} &= \delta_1(V_0).
\end{align*}

\end{proof}

\begin{example}
    $K=ST$, $E=V_m$.
\end{example}

Theorem \ref{thm:4.2} and Corollary \ref{cor:generalK} allow us to compute the weights and discrepancies for $E=V_m$.

\begin{theorem}
If $E=V_m$,
\begin{equation*}
    \delta_0(E)=\frac{3}{4\sqrt{10 \cdot 6^m}},
\end{equation*}

\begin{equation*}
    \delta_1(E)=\frac{1}{16\cdot6^m},
\end{equation*}
the weights $\{p(x)\}$ are

\begin{equation*}
    p(x)=\left\{ \begin{matrix}
        \left( \frac{1}{4}\right)^{m+1} & : \mbox{ if $x \in V_0$}\\
        2\left(\frac{1}{4}\right)^{m+1} & : \mbox{ if $x \in (V_m \setminus V_0)$}
    \end{matrix}\right.,
\end{equation*}
and for the uniform weights $\{w(x)\}$,

\begin{equation*}
    \delta(E, w)=\frac{3\left(4^m-1\right)}{4^m \left(4^m+1\right)} R^{\sfrac{1}{2}}.
\end{equation*}

\end{theorem}

\begin{proof}
The values of $\delta_0$, $\delta_1$, and the weights follow from Theorem \ref{cor:generalK} (for $K=ST$). $\delta(E, w)$ is computed from the weights.
\end{proof}

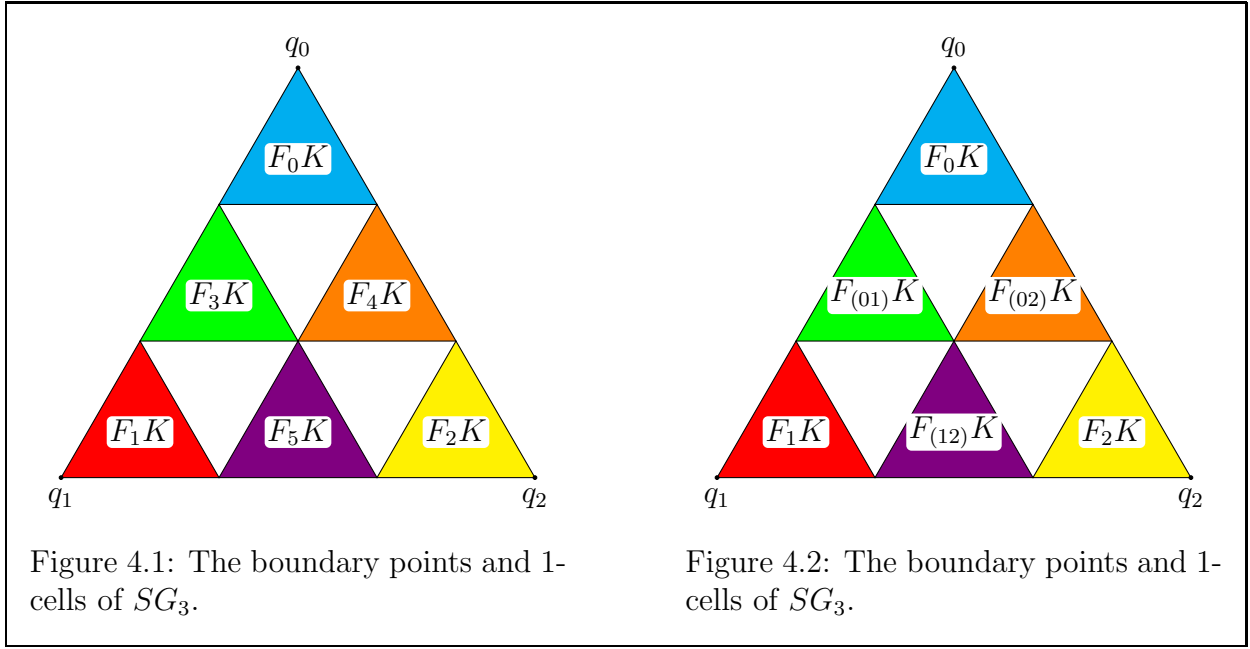
\begin{figure}[h]
\begin{framed}
\centering
\begin{minipage}{.45\textwidth}
    \centering
    \begin{tikzpicture}[scale=.7]
    
    \draw[black, fill=cyan] (0, 5.192) -- (-1.5, 2.596) -- (1.5, 2.596) -- cycle;
    \draw[black, fill=green] (-1.5, 2.596) -- (-3, 0) -- (0, 0) -- cycle;
    \draw[black, fill=orange] (1.5, 2.596) -- (0, 0) -- (3, 0) -- cycle;
    \draw[black, fill=red] (-3, 0) -- (-4.5, -2.596) -- (-1.5, -2.596) -- cycle;
    \draw[black, fill=violet] (0, 0) -- (-1.5, -2.596) -- (1.5, -2.596) -- cycle;
    \draw[black, fill=yellow] (3, 0) -- (1.5, -2.596) -- (4.5, -2.596) -- cycle;
    
    \filldraw[black] (0, 5.192) circle (1 pt) node[anchor=south]{$q_0$};
    \filldraw[black] (-4.5, -2.596) circle (1 pt) node[anchor=north]{$q_1$};
    \filldraw[black] (4.5, -2.596) circle (1 pt) node[anchor=north]{$q_2$};
    
    \node [fill=white,rounded corners=2pt,inner sep=1pt] at (0, 3.461){$F_0 K$};
    \node [fill=white,rounded corners=2pt,inner sep=1pt] at (-3, -1.731){$F_1 K$};
    \node [fill=white,rounded corners=2pt,inner sep=1pt] at (3, -1.731){$F_2 K$};
    \node [fill=white,rounded corners=2pt,inner sep=1pt] at (-1.5, .865){$F_3 K$};
    \node [fill=white,rounded corners=2pt,inner sep=1pt] at (1.5, .865){$F_4 K$};
    \node [fill=white,rounded corners=2pt,inner sep=1pt] at (0, -1.731){$F_5 K$};

    \end{tikzpicture}
    \captionof{figure}{The boundary points and $1$-cells of $SG_3$.}
    \label{fig:4.1}
\end{minipage}\hfill\begin{minipage}{.45\textwidth}
    \centering
    \begin{tikzpicture}[scale=.7]
    
    \draw[black, fill=cyan] (0, 5.192) -- (-1.5, 2.596) -- (1.5, 2.596) -- cycle;
    \draw[black, fill=green] (-1.5, 2.596) -- (-3, 0) -- (0, 0) -- cycle;
    \draw[black, fill=orange] (1.5, 2.596) -- (0, 0) -- (3, 0) -- cycle;
    \draw[black, fill=red] (-3, 0) -- (-4.5, -2.596) -- (-1.5, -2.596) -- cycle;
    \draw[black, fill=violet] (0, 0) -- (-1.5, -2.596) -- (1.5, -2.596) -- cycle;
    \draw[black, fill=yellow] (3, 0) -- (1.5, -2.596) -- (4.5, -2.596) -- cycle;
    
    \filldraw[black] (0, 5.192) circle (1 pt) node[anchor=south]{$q_0$};
    \filldraw[black] (-4.5, -2.596) circle (1 pt) node[anchor=north]{$q_1$};
    \filldraw[black] (4.5, -2.596) circle (1 pt) node[anchor=north]{$q_2$};
    
     \node [fill=white,rounded corners=2pt,inner sep=1pt] at (0, 3.461){$F_0 K$};
    \node [fill=white,rounded corners=2pt,inner sep=1pt] at (-3, -1.731){$F_1 K$};
    \node [fill=white,rounded corners=2pt,inner sep=1pt] at (3, -1.731){$F_2 K$};
    \node [fill=white,rounded corners=2pt,inner sep=1pt] at (-1.5, .865){$F_{(01)} K$};
    \node [fill=white,rounded corners=2pt,inner sep=1pt] at (1.5, .865){$F_{(02)} K$};
    \node [fill=white,rounded corners=2pt,inner sep=1pt] at (0, -1.731){$F_{(12)} K$};

    \end{tikzpicture}
    \captionof{figure}{The boundary points and $1$-cells of $SG_3$.}
    \label{fig:4.2}
\end{minipage}
\end{framed}
\end{figure}

Now, we provide some examples for another p.c.f. self-similar fractal, $SG_3$. $N=6$ and $N_0=3$. The points $q_k$ and the cells $F_i SG_3$ are shown in Figure \ref{fig:4.1}. To make dealing with these contractions more intuitive, from this point on we will refer to $F_3$ as $F_{(01)}$, $F_4$ as $F_{(02)}$, and $F_5$ as $F_{(12)}$. The renamed cells are displayed in Figure \ref{fig:4.2}. For all $i$, $\mu_i =\frac{1}{6}$ and $r_i = \frac{7}{15}$. The values of the harmonic function $h_0$ on $V_1$ are shown in Figure \ref{fig:4.3}. (The values for the other harmonic functions can be determined from this by symmetry.) If $A(kk', nn')$ is indexed by the ordering $(kk'<nn' \Longleftrightarrow \mbox{$k<k'$ or ($k=k'$ and $n<n'$)})$, then 
\begin{equation*}
    A(kk', nn') = \frac{1}{1350} \begin{pmatrix}
        410 & 219 & 219 & 219 & 123 & 113 & 219 & 113 & 123\\
        125 & 280 & 161 & 55 & 125 & 71 & 71 & 161 & 97\\
        125 & 161 & 280 & 71 & 97 & 161 & 55 & 71 & 125\\
        125 & 55 & 71 & 280 & 125 & 161 & 161 & 71 & 97\\
        123 & 219 & 113 & 219 & 410 & 219 & 113 & 219 & 123\\
        97 & 71 & 161 & 161 & 125 & 280 & 71 & 55 & 125\\
        125 & 71 & 55 & 161 & 97 & 71 & 280 & 161 & 125\\
        97 & 161 & 71 & 71 & 125 & 55 & 161 & 280 & 125\\
        123 & 113 & 219 & 113 & 123 & 219 & 219 & 219 & 410\\
    \end{pmatrix},
\end{equation*}
so the eigenvector $I(kk')$ is

\begin{equation*}
    I(kk') = \left\{ \begin{matrix}
        \sfrac{551}{3735} & : \mbox{ if $k=k'$}\\
        \sfrac{347}{3735} & : \mbox{ if $k \neq k'$}
    \end{matrix} \right..
\end{equation*}\begin{figure}
\begin{framed}
\centering
\begin{minipage}{.45\textwidth}
    \centering
    \begin{tikzpicture}[scale=.7]
    \draw[black] (0, 5.192) -- (-1.5, 2.596) -- (1.5, 2.596) -- cycle;
    \draw[black] (-1.5, 2.596) -- (-3, 0) -- (0, 0) -- cycle;
    \draw[black] (1.5, 2.596) -- (0, 0) -- (3, 0) -- cycle;
    \draw[black] (-3, 0) -- (-4.5, -2.596) -- (-1.5, -2.596) -- cycle;
    \draw[black] (0, 0) -- (-1.5, -2.596) -- (1.5, -2.596) -- cycle;
    \draw[black] (3, 0) -- (1.5, -2.596) -- (4.5, -2.596) -- cycle;
    
    \filldraw[black] (0, 5.192) circle (1 pt) node[anchor=south]{1};
    \filldraw[black] (-1.5, 2.596) circle (1 pt) node[anchor=east]{8/15 };
    \filldraw[black] (1.4, 2.596) circle (1 pt) node[anchor=west]{ 8/15};
    \filldraw[black] (-3, 0) circle (1 pt) node[anchor=east]{4/15 };
    \filldraw[black] (0, 0) circle (1 pt) node[anchor=south]{\contour{white}{5/15}};
    \filldraw[black] (3, 0) circle (1 pt) node[anchor=west]{ 4/15};
    \filldraw[black] (-4.5, -2.596) circle (1 pt) node[anchor=north]{0};
    \filldraw[black] (-1.5, -2.596) circle (1 pt) node[anchor=north]{3/15};
    \filldraw[black] (1.5, -2.596) circle (1 pt) node[anchor=north]{3/15};
    \filldraw[black] (4.5, -2.596) circle (1 pt) node[anchor=north]{0};
    \end{tikzpicture}
    \captionof{figure}{The values of $h_0$ on $V_1$.}
    \label{fig:4.3}
\end{minipage}\hfill\begin{minipage}{.45\textwidth}
    \centering
    \begin{tikzpicture}[scale=.7]
    \draw[black] (0, 5.192) -- (-1.5, 2.596) -- (1.5, 2.596) -- cycle;
    \draw[black] (-1.5, 2.596) -- (-3, 0) -- (0, 0) -- cycle;
    \draw[black] (1.5, 2.596) -- (0, 0) -- (3, 0) -- cycle;
    \draw[black] (-3, 0) -- (-4.5, -2.596) -- (-1.5, -2.596) -- cycle;
    \draw[black] (0, 0) -- (-1.5, -2.596) -- (1.5, -2.596) -- cycle;
    \draw[black] (3, 0) -- (1.5, -2.596) -- (4.5, -2.596) -- cycle;
    
    \filldraw[black] (-1.5, 2.596) circle (1 pt) node[anchor=east]{$F_0 q_1$ };
    \filldraw[black] (1.5, 2.596) circle (1 pt) node[anchor=west]{ $F_0 q_2$};
    \filldraw[black] (-3, 0) circle (1 pt) node[anchor=east]{$F_1 q_0$ };
    \filldraw[black] (0, 0) circle (1 pt) node[anchor=south]{\contour{white}{p}};
    \filldraw[black] (3, 0) circle (1 pt) node[anchor=west]{ $F_2 q_0$};
    \filldraw[black] (-1.5, -2.596) circle (1 pt) node[anchor=north]{$F_1 q_2$};
    \filldraw[black] (1.5, -2.596) circle (1 pt) node[anchor=north]{$F_2 q_1$};
    \end{tikzpicture}
    \captionof{figure}{The elements of $V_1 \setminus V_0$.}
    \label{fig:4.4}
\end{minipage}
\end{framed}
\end{figure}We move on to the matrix $X$. Let $p$ refer to the point in the middle of $SG_3$: $p=F_{(01)} q_2 = F_{(02)} q_1 = F_{(12)} q_0$. The $7$ elements of $V_1 \setminus V_0$ (seen in Figure \ref{fig:4.4}) are $F_0 q_1$, $F_0 q_2$, $F_1 q_0$, $F_1 q_2$, $F_2 q_0$, $F_2 q_1$, and $p$. We will consider them in this order for the indexing of $X$.

\begin{equation*}
    X=\frac{1}{7} \begin{pmatrix}
        60 & -15 & -15 & 0 & 0 & 0 & -15\\
        -15 & 60 & 0 & 0 & -15 & 0 & -15\\
        -15 & 0 & 60 & -15 & 0 & 0 & -15\\
        0 & 0 & -15 & 60 & 0 & -15 & -15\\
        0 & -15 & 0 & 0 & 60 & -15 & -15\\
        0 & 0 & 0 & -15 & -15 & 60 & -15\\
        -15 & -15 & -15 & -15 & -15 & -15 & 90
    \end{pmatrix}.
\end{equation*}

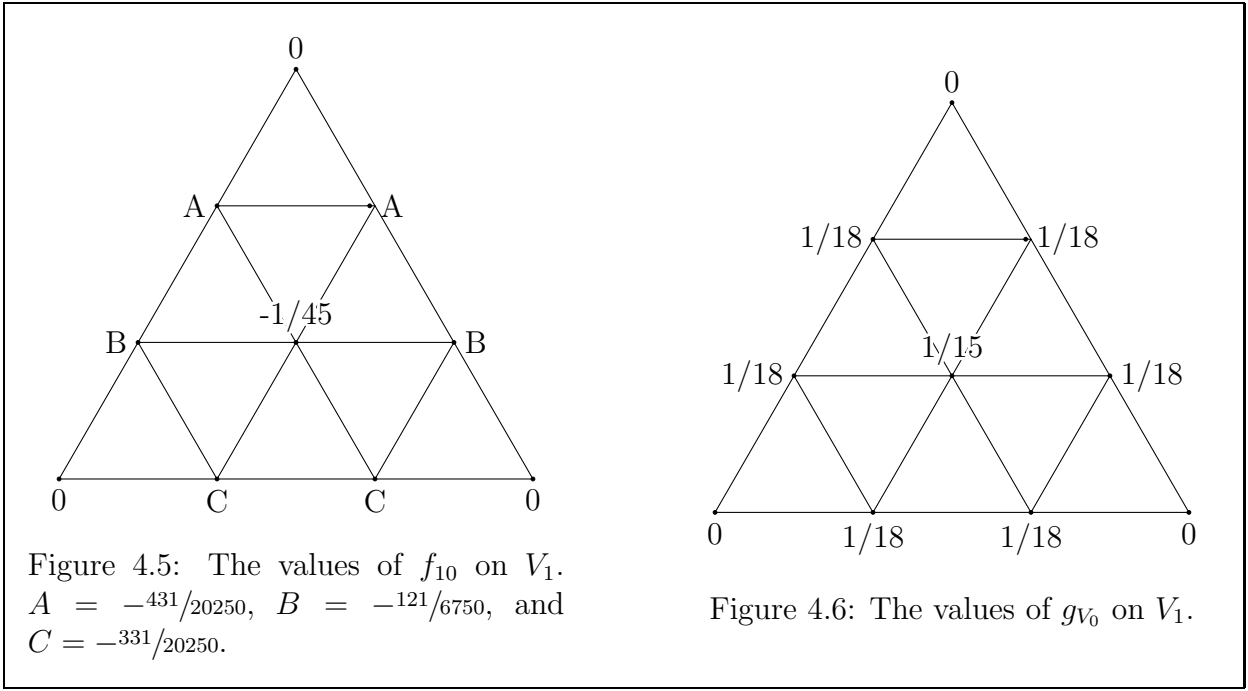
\begin{equation*}
    G=X^{-1}=\frac{1}{2700} \begin{pmatrix}
        469 & 203 & 203 & 133 & 133 & 119 & 210\\
        203 & 469 & 133 & 119 & 203 & 133 & 210\\
        203 & 133 & 469 & 203 & 119 & 133 & 210\\
        133 & 119 & 203 & 469 & 133 & 203 & 210\\
        133 & 203 & 119 & 133 & 469 & 203 & 210\\
        119 & 133 & 133 & 203 & 203 & 469 & 210\\
        210 & 210 & 210 & 210 & 210 & 210 & 420 % blaze it
    \end{pmatrix}.
\end{equation*}
This allows us to evaluate $f_{1k}(F_i q_n)$. The values for $f_{10}$ are shown in Figure \ref{fig:4.5}. The values for other harmonic functions can be determined by symmetry. The values for $g_{V_0} = -\sum_{k=0}^2 f_{1k}$ are shown in Figure \ref{fig:4.6}. For $x \in V_1$,

\begin{equation*}
    g_{V_0} = \left\{ \begin{matrix}
        0 & : \mbox{ if $x \in V_0$}\\
        \sfrac{1}{15} & : \mbox{ if $x=p$}\\
        \sfrac{1}{18} & : \mbox{ if $x \in V_1 \setminus (V_0 \cup \{p\})$}
    \end{matrix}\right..
\end{equation*}

\begin{figure}
\begin{framed}
\centering
\begin{minipage}{.45\textwidth}
    \centering
    \begin{tikzpicture}[scale=.7]
    \draw[black] (0, 5.192) -- (-1.5, 2.596) -- (1.5, 2.596) -- cycle;
    \draw[black] (-1.5, 2.596) -- (-3, 0) -- (0, 0) -- cycle;
    \draw[black] (1.5, 2.596) -- (0, 0) -- (3, 0) -- cycle;
    \draw[black] (-3, 0) -- (-4.5, -2.596) -- (-1.5, -2.596) -- cycle;
    \draw[black] (0, 0) -- (-1.5, -2.596) -- (1.5, -2.596) -- cycle;
    \draw[black] (3, 0) -- (1.5, -2.596) -- (4.5, -2.596) -- cycle;
    
    \filldraw[black] (0, 5.192) circle (1 pt) node[anchor=south]{0};
    \filldraw[black] (-1.5, 2.596) circle (1 pt) node[anchor=east]{A };
    \filldraw[black] (1.4, 2.596) circle (1 pt) node[anchor=west]{ A};
    \filldraw[black] (-3, 0) circle (1 pt) node[anchor=east]{B };
    \filldraw[black] (0, 0) circle (1 pt) node[anchor=south]{\contour{white}{-1/45}};
    \filldraw[black] (3, 0) circle (1 pt) node[anchor=west]{ B};
    \filldraw[black] (-4.5, -2.596) circle (1 pt) node[anchor=north]{0};
    \filldraw[black] (-1.5, -2.596) circle (1 pt) node[anchor=north]{C};
    \filldraw[black] (1.5, -2.596) circle (1 pt) node[anchor=north]{C};
    \filldraw[black] (4.5, -2.596) circle (1 pt) node[anchor=north]{0};
    \end{tikzpicture}
    \captionof{figure}{The values of $f_{10}$ on $V_1$. $A = -\sfrac{431}{20250}$, $B = -\sfrac{121}{6750}$, and $C = -\sfrac{331}{20250}$.}
    \label{fig:4.5}
\end{minipage}\hfill\begin{minipage}{.45\textwidth}
    \centering
    \begin{tikzpicture}[scale=.7]
    \draw[black] (0, 5.192) -- (-1.5, 2.596) -- (1.5, 2.596) -- cycle;
    \draw[black] (-1.5, 2.596) -- (-3, 0) -- (0, 0) -- cycle;
    \draw[black] (1.5, 2.596) -- (0, 0) -- (3, 0) -- cycle;
    \draw[black] (-3, 0) -- (-4.5, -2.596) -- (-1.5, -2.596) -- cycle;
    \draw[black] (0, 0) -- (-1.5, -2.596) -- (1.5, -2.596) -- cycle;
    \draw[black] (3, 0) -- (1.5, -2.596) -- (4.5, -2.596) -- cycle;
    
    \filldraw[black] (0, 5.192) circle (1 pt) node[anchor=south]{0};
    \filldraw[black] (-1.5, 2.596) circle (1 pt) node[anchor=east]{1/18 };
    \filldraw[black] (1.4, 2.596) circle (1 pt) node[anchor=west]{ 1/18};
    \filldraw[black] (-3, 0) circle (1 pt) node[anchor=east]{1/18 };
    \filldraw[black] (0, 0) circle (1 pt) node[anchor=south]{\contour{white}{1/15}};
    \filldraw[black] (3, 0) circle (1 pt) node[anchor=west]{ 1/18};
    \filldraw[black] (-4.5, -2.596) circle (1 pt) node[anchor=north]{0};
    \filldraw[black] (-1.5, -2.596) circle (1 pt) node[anchor=north]{1/18};
    \filldraw[black] (1.5, -2.596) circle (1 pt) node[anchor=north]{1/18};
    \filldraw[black] (4.5, -2.596) circle (1 pt) node[anchor=north]{0};
    
    \end{tikzpicture}
    \captionof{figure}{The values of $g_{V_0}$ on $V_1$.}
    \label{fig:4.6}
\end{minipage}
\end{framed}
\end{figure}

\begin{example} \label{ex:4.3}
    $K=SG_3$, $E=V_0$.
\end{example}

Our first sample set for $SG_3$ is $V_0$. The discrepancies can be calculated using the results we have shown for a general $K$ and the values of $g_{V_0}$ on $V_1$.

\begin{theorem} \label{thm:4.6}
    For $SG_3$,
    \begin{equation*}
        \delta_0(V_0)=\left(\frac{13}{249}\right)^{\sfrac{1}{2}}, \quad \quad \delta_1(V_0)=\frac{540}{8051},
    \end{equation*}
    $p(x)=\frac{1}{3}$ for $x \in V_0$, and $\delta(V_0, w)=0$.
\end{theorem}

\begin{proof}
Because $\mu_i = \frac{1}{6}$ and $r_i = \frac{7}{15}$ for $SG_3$, Lemma \ref{lem:4.1} applied to $SG_3$ says that if $w$ is a word of length $m$, $g_{V_0}(F_w q_0) = a$, $g_{V_0}(F_w q_1) = b$, and $g_{V_0}(F_w q_2) = c$,

\begin{align} \label{eq:4.11}
\begin{split}
    u(F_w F_0 q_1) &= \frac{8a+4b+3c}{15} + \left(\frac{7}{90}\right)^m \frac{1}{18},\\
    &\\
    u(F_w F_0 q_2) &= \frac{8a+3b+4c}{15} + \left(\frac{7}{90}\right)^m \frac{1}{18},\\
    &\\
    u(F_w F_1 q_0) &= \frac{4a+8b+3c}{15} + \left(\frac{7}{90}\right)^m \frac{1}{18},\\
    &\\
    u(F_w F_1 q_2) &= \frac{3a+8b+4c}{15} + \left(\frac{7}{90}\right)^m \frac{1}{18},\\
    &\\
    u(F_w F_2 q_0) &= \frac{4a+3b+8c}{15} + \left(\frac{7}{90}\right)^m \frac{1}{18},\\
    &\\
    u(F_w F_2 q_1) &= \frac{3a+4b+8c}{15} + \left(\frac{7}{90}\right)^m \frac{1}{18},\\
    &\\
    \mbox{and }u(F_w p) &= \frac{a+b+c}{3}+\left(\frac{7}{90}\right)^m \frac{1}{15}.
\end{split}
\end{align}
Thus,

\begin{equation*}
    g_{V_0} = \sum_{m=0}^{\infty} h_m,
\end{equation*}
where $h_m$ is the $(m+1)$-spline such that for all $x \in V_{m+1}$,

\begin{equation*}
    h_m(x) = \left\{ \begin{matrix}
        0 & : \mbox{ if $x \in V_m$}\\
        &\\
        \left(\frac{7}{90}\right)^m \frac{1}{18} & : \mbox{ if $x=F_w y$ for some $|w|=m$, $y \in V_1 \setminus
        (V_0 \cup \{p\})$}\\
        &\\
        \left(\frac{7}{90}\right)^m \frac{1}{15} & : \mbox{ if $x=F_w p$ for some $|w|=m$}
    \end{matrix} \right..
\end{equation*}
For each $m$,

\begin{equation*}
    \int h_m d\mu = \frac{3\cdot0 + 6\cdot2\cdot\left(\frac{7}{90}\right)^m \left(\frac{1}{18}\right) + 1\cdot3\cdot \left(\frac{7}{90}\right)^m \left(\frac{1}{15}\right)}{3+6\cdot2+1\cdot3} = \frac{13}{270} \left(\frac{7}{90}\right)^m.
\end{equation*}
Thus,

\begin{equation*}
    \int g_{V_0} d\mu = \sum_{m=0}^{\infty} \frac{13}{270} \left(\frac{7}{90}\right)^m = \frac{13}{249}
\end{equation*}
so

\begin{equation*}
    \delta_0(V_0) = \left(\frac{13}{249}\right)^{\sfrac{1}{2}}.
\end{equation*}

% start delta_1 proof

Let $w=(01)2$. Note that $w$ is a word of length $2$, because $(01)$ is a character that is really equal to $3$. Let $w^2 = (01)2(01)2$, $w^3 = (01)2(01)2(01)2$, and so on. Let $F_{w^\infty}$ be the fixed point of $F_w$. (This definition is natural because for all $x \in SG_3$, $\lim_{k \rightarrow \infty} F_{w^k} x = F_{w^\infty}$.) If \eqref{eq:4.11} is used to compute the values of $g_{V_0}$ on $V_2$, then

\begin{equation} \label{4.onefifteenth}
    F_{(01)2} (q_0) = F_{(01)2} (q_1) = F_{(01)2} (q_2) = \frac{1}{15}
\end{equation}
and for all $|w'|=2$, $i \in \{0, 1, 2\}$,

\begin{equation} \label{4.greatestboundary}
    F_{(01)2}(q_i) \geq F_{w'}(q_i).
\end{equation}
By \eqref{4.greatestboundary}, $g_{V_0}$ attains its supremum on $F_{(01)2} SG_3 = F_w SG$. Let

\begin{equation*}
    u = \left(\frac{7}{90}\right)^{-2} \left( g_{V_0} \circ F_w - \frac{1}{15} \right).
\end{equation*}
By \eqref{4.onefifteenth}, $u(x)=0$ for all $x \in V_0$. The Laplacian of $u$ is

\begin{equation*}
    \left( \frac{7}{90}\right)^{-2} \left( \laplace \left( g_{V_0} \circ F_w \right) - \laplace \left( \frac{1}{15} \right) \right) = \left( \frac{7}{90} \right)^{-2} \left( \left( \frac{7}{90}\right)^2 \laplace g_{V_0} - 0 \right) = -1.
\end{equation*}
Thus,

\begin{equation*}
    u=g_{V_0}.
\end{equation*}
By repeating this argument indefinitely and using induction, $g_{V_0}$ attains its supremum in $F_{w^k} SG_3$ for all $k$. Thus,

\begin{equation*}
    \delta_1(V_0) = g_{V_0} (F_{w^\infty}).
\end{equation*}
This means that $\delta_1(V_0)$ satisfies

\begin{equation*}
    \delta_1(V_0) = \frac{1}{15} + \left(\frac{7}{90}\right)^2 \delta_1(V_0)
\end{equation*}
so

\begin{equation*}
    \delta_1(V_0)=\frac{540}{8051}.
\end{equation*}

% end delta_1 proof

The weights are uniform by symmetry.

\end{proof}

\begin{example}
    $K=SG_3$, $E=V_m$.
\end{example}

As with $ST$, Corollary \ref{cor:generalK} allows us to extend the results for $V_0$ to $V_m$. For the weights $\{p(x)\}$ of $V_m$, we introduce a function $\eta : V_m \rightarrow \mathbb{N}$. For all $x \in V_m$, $\eta(x)$ is the number of $m$-cells to which $x$ belongs. If we consider the graph $\Gamma_m$ with vertex set $V_m$ and an edge between every $F_w x$ and $F_w y$ such that $|w|=m$, $x, y \in V_0$, and $x \neq y$, then $\eta(x)$ is equal to $\frac{1}{N_0-1}$ times the number of neighbors of $x$ in $\Gamma_m$. In the case of $SG_3$, $\eta(x)$ is $1$ if $x \in V_0$, $3$ if $x = F_w p$ for some $w$ of \emph{any} length (not only $m$), and $2$ otherwise.

\begin{theorem} \label{eq:4.13'}
For the $3$-level gasket,
\begin{equation}
    \delta_0(V_m) = \left(\frac{7}{90}\right)^{\sfrac{m}{2}} \frac{13}{249},
\end{equation}

\begin{equation} \label{eq:4.14}
    \delta_1(V_m) = \left(\frac{7}{90}\right)^m \frac{540}{8051},
\end{equation}

\begin{equation} \label{eq:4.15}
    p_{V_m}(x) = \frac{\eta(x)}{3 \cdot 6^m},
\end{equation}
and

\begin{equation*}
    \delta(V_m, w) = \left\{ \begin{matrix}
        0 & : \mbox{ if $m=0$}\\
        &\\
        \frac{4}{15} R^{\sfrac{1}{2}} & : \mbox{ if $m=1$}\\
        &\\
        \frac{4}{5} \cdot \frac{(6^m-1)(6^m+4)}{(6^m)(7\cdot6^m+8)} R^{\sfrac{1}{2}} & : \mbox{ if $m \geq 2$}
    \end{matrix}\right..
\end{equation*}

\end{theorem}

\begin{proof}
Apply Corollary \ref{cor:generalK} to $K=SG_3$, $E=V_m$. (a) gives (4.14), (b) gives \eqref{eq:4.14}, and (c) gives \eqref{eq:4.15}. For $m=0$ and $m=1$, $\delta(V_m, w)$ can be directly computed. To calculate $\delta(V_m, w)$ for $m\geq 2$, first let $K_k = \{ x \in V_m | \eta(x)=k\}$ for all $k$. $K_1=V_0$, so $\# K_1=3$. For each $j$-cell, there are $6$ elements $ x \in (V_{j+1} \setminus V_{j})$ with $\eta(x)=2$ and one with $\eta(x)=3$. Thus,

\begin{equation*}
    \# K_2 = \sum_{j=0}^{m-1} 6\cdot 6^j = \frac{6}{5}(6^m-1)
\end{equation*}
and

\begin{equation*}
    \# K_3 = \sum_{j=0}^{m-1} 6^j = \frac{1}{5}(6^m-1).
\end{equation*}
The number of points of $V_m$ is

\begin{equation*}
    \# K_1 + \# K_2 + \# K_3 = 3 + \frac{7}{5}(6^m-1).
\end{equation*}
so the uniform weights are

\begin{equation*}
    w(x) = \frac{1}{3 + \frac{7}{5}(6^m-1)} = \frac{5}{7\cdot 6^m+8}.
\end{equation*}
For $x \in K_1$,

\begin{equation*}
    w(x)-p(x) = \frac{5}{7\cdot 6^m+8} - \frac{1}{3 \cdot 6^m} = \frac{8(6^m-1)}{3(6^m)(7\cdot6^m+8)}.
\end{equation*}
For $x \in K_2$,

\begin{equation*}
    w(x)-p(x) = \frac{5}{7\cdot 6^m+8} - \frac{2}{3 \cdot 6^m} = \frac{6^m-16}{3(6^m)(7\cdot6^m+8)}.
\end{equation*}
For $x \in K_3$,

\begin{equation*}
    p(x)-w(x) = \frac{3}{3\cdot 6^m} - \frac{5}{7\cdot 6^m+8} = \frac{6\cdot6^m+24}{3(6^m)(7\cdot6^m+8)}.
\end{equation*}
So

\begin{align*}
    \delta(V_m, w) &= \sum_{x \in V_m} |p(x)-w(x)| \\
    &\\
    &= \frac{6^m-1}{3(6^m)(7\cdot6^m+8)} \left( 3\cdot8 + \frac{6}{5}(6^m-16) + \frac{1}{5}(6\cdot6^m+24) \right) \\
    &\\
    &= \frac{4}{5}\cdot\frac{(6^m-1)(6^m+4)}{(6^m)(7\cdot6^m+8)}.
\end{align*}

\end{proof}

Interestingly, $SG_3$ is the first fractal we have encountered in which $\delta(V_m, w)$ does not decay exponentially to $0$ as $m$ increases. Rather,

\begin{equation*}
    \delta(V_m, w) \overset{m\rightarrow \infty}{\longrightarrow} \frac{4}{35}.
\end{equation*}
This is because there is a set $S_m$ of points (the elements $x$ with $\eta(x)=3$) whose weights differ substantially from the uniform weights and

\begin{equation*}
    \frac{\# S_m}{\# V_m} > 0.
\end{equation*}
In our previous examples $SG$ and $ST$, the only points in $V_m$ whose weights differed substantially from the uniform weights (for large $m$) were those in $V_0$, a set which does not grow at all with $m$. This means that the uniform weights $\{w(x)\}$ are a poor choice to use to numerically integrate functions on $SG_3$.
 
\section{Energy measures}

 We now turn our attention to harmonic energy measures on the Sierpi\'{n}ski gasket: measures $\nu_{h, H}$ where $h$ and $H$ are harmonic functions and for any cell $C$, $\nu_{h, H}=\energy_C(h, H)$. We develop a technique to calculate integrals of the form $\int u d\nu_{h, H}$, where $u$ is a harmonic spline. Later in the section, we will generalize the results beyong $SG$ to a more extensive class of fractals. Given a finite set $E \subset V_*$, we will derive a method to produce a set of weights $\{p(x)\}$ that can be used to numerically integrate any function that satisfies the conditions of Theorem \ref{thm:2.6}.
   
 First, we show that $\left\{ \nu_0, \nu_1, \nu_2 \right\}$ is a basis of the set of harmonic energy measures on $SG$, and provide a formula to express any harmonic energy measure as a linear combination of $\nu_0$, $\nu_1$, and $\nu_2$.
 
\begin{theorem} \label{thm:5.1} If $h=a_0 h_0+a_1 h_1 +a_2 h_2$ and $H=b_0 h_0 +b_1 h_1 +b_2 h_2$, then

\begin{align} \label{eq:5.2}
\begin{split}
\nu_{h, H} = &  \left( a_0 b_0 + \frac{a_1 b_2 + a_2 b_1 - a_0 b_1 - a_1 b_0 - a_0 b_2 - a_2 b_0}{2} \right) \nu_0 \\
& + \left( a_1 b_1 + \frac{a_0 b_2 + a_2 b_0 - a_1 b_0 - a_0 b_1 - a_1 b_2 - a_2 b_1}{2} \right) \nu_1 \\
& + \left( a_2 b_2 + \frac{a_0 b_1 + a_1 b_0 - a_2 b_0 - a_0 b_2 - a_2 b_1 - a_1 b_2}{2} \right) \nu_2
\end{split}
\end{align}

\end{theorem}
\begin{proof}

We use the fact that energies are additive in the sense that $ \energy_C (u+v, w) =  \energy_C(u, w) + \energy_C(v, w) $.
This additivity clearly follows from the definition of energy, and holds for both the first and second variable.

By expanding for each variable, \begin{equation*} 
\nu_{h, H} = \sum_i \sum_j \nu_{a_i h_i, b_j h_j} .
\end{equation*}
Clearly, for any cell $C$, $\energy_C(au, bv) = ab \energy_C(u, v)$, so

\begin{equation} \label{eq:5.4}
\nu_{h, H} = \sum_i \sum_j a_i b_j \nu_{h_i, h_j}.
\end{equation}
It is a result in \cite{energy} that

\begin{align*}
\begin{split}
\nu_{h_0, h_1} & = \frac{1}{2} \left( -\nu_0 - \nu_1 + \nu_2 \right) \\
\nu_{h_0, h_2} & = \frac{1}{2} \left( -\nu_0 + \nu_1 - \nu_2 \right) \\
\nu_{h_1, h_2} & = \frac{1}{2} \left( \nu_0 - \nu_1 - \nu_2 \right).
\end{split}
\end{align*}
Thus, all $9$ terms in \eqref{eq:5.4} can be expressed as linear combinations of $\left\{ \nu_0, \nu_1, \nu_2 \right\}$, and when they are and their sum is taken, the result is \eqref{eq:5.2}.

\end{proof}

To calculate the weights $\left\{p(x)\right\}$, we must calculate integrals of the form $\int u d\nu$, where $\nu$ is a harmonic energy measure and $u$ is a harmonic spline (more specifically, $u$ is an indicator of some $x \in E$). This problem can be split into two problems: determining $\int u d\nu$ from the values of $\int u \circ F_w d\nu$ for $|w|=m$, and taking the integral with respect to $\nu$ of a harmonic function.

To solve the first problem, we will construct matrices $M_w$ for each word $w$ such that for every continuous function $f$, we can use $M_w$ to evaluate integrals $\int_{F_w SG} f d\nu_i$.

\begin{theorem} \label{thm:5.2}

If

\begin{align}
\begin{split}
M_0 = & \frac{1}{15} \begin{pmatrix}
9 & 0 & 0 \\
2 & 2 & -1 \\
2 & -1 & 2
\end{pmatrix}, \\
M_1 = & \frac{1}{15} \begin{pmatrix}
2 & 2 & -1 \\
0 & 9 & 0 \\
-1 & 2 & 2
\end{pmatrix}, \\
 M_2 = & \frac{1}{15} \begin{pmatrix}
2 & -1 & 2 \\
-1 & 2 & 2 \\
0 & 0 & 9
\end{pmatrix}
\end{split}
\end{align}
then for all continuous functions $f$, and for all $i \in \left\{0, 1, 2\right\}$,

\begin{equation} \label{eq:5.6}
\begin{pmatrix}
\int_{F_i SG} f d\nu_0 \\
\int_{F_i SG} f d\nu_1 \\
\int_{F_i SG} f d\nu_2
\end{pmatrix} = M_i \begin{pmatrix}
\int f \circ F_i d\nu_0 \\
\int f \circ F_i d\nu_1 \\
\int f \circ F_i d\nu_2
\end{pmatrix}.
\end{equation}

\end{theorem}

\begin{proof}
By the definition of an energy measure,

\begin{equation} \label{eq:5.7}
\int_{F_0 SG} f d\nu_0 = \frac{5}{3} \int f \circ F_0 d\nu_{h_0+\frac{2}{5}h_1+\frac{2}{5}h_2}
\end{equation}
and

\begin{equation} \label{eq:5.8}
\int_{F_0 SG} f d\nu_1 = \frac{5}{3} \int f \circ F_0 d\nu_{h_0+\frac{2}{5}h_1+\frac{2}{5}h_2}.
\end{equation}
By Theorem \ref{thm:5.1},

\begin{equation} \label{eq:5.9}
\nu_{h_0+\frac{2}{5}h_1+\frac{2}{5}h_2} = \frac{9}{25} \nu_0
\end{equation}
and

\begin{equation} \label{eq:5.10}
\nu_{\frac{2}{5}h_1+\frac{1}{5}h_2} = \frac{2}{25} \nu_0 + \frac{2}{25} \nu_1 - \frac{1}{25} \nu_2.
\end{equation}
By \eqref{eq:5.7} and \eqref{eq:5.9},

\begin{equation*} 
\int_{F_0 SG} f d\nu_0 = \frac{3}{5} \int f \circ F_0 d\nu_0
\end{equation*}
and by \eqref{eq:5.8} and \eqref{eq:5.10}

\begin{equation*} 
\int_{F_0 SG} f d\nu_1 = \frac{2}{15} \int f \circ F_0 d\nu_0 + \frac{2}{15} \int f \circ F_0 d\nu_1 - \frac{1}{15} \int f \circ F_0 d\nu_2.
\end{equation*}
By symmetry, 

\begin{equation*}
\int_{F_0 SG} f d\nu_2 = \frac{2}{15} \int f \circ F_0 d\nu_0 - \frac{1}{15} \int f \circ F_0 d\nu_1 + \frac{2}{15} \int f \circ F_0 d\nu_2.
\end{equation*}
Thus

\begin{equation*}
M_0 = \frac{1}{15} \begin{pmatrix}
9 & 0 & 0 \\
2 & 2 & -1 \\
2 & -1 & 2
\end{pmatrix}
\end{equation*}
and by symmetry

\begin{equation*}
M_1 = \frac{1}{15} \begin{pmatrix}
2 & 2 & -1 \\
0 & 9 & 0 \\
-1 & 2 & 2
\end{pmatrix}, \quad M_2 = \frac{1}{15} \begin{pmatrix}
2 & -1 & 2 \\
-1 & 2 & 2 \\
0 & 0 & 9
\end{pmatrix}.
\end{equation*}.

\end{proof}

Interestingly, \cite{energy} showed that for what turn out to be the same matrices $M_0$, $M_1$, and $M_2$:

\begin{equation} \label{olduse}
\begin{pmatrix}
\int_{F_0 SG} f d\nu_i \\
\int_{F_1 SG} f d\nu_i \\
\int_{F_2 SG} f d\nu_i
\end{pmatrix} = M_i \begin{pmatrix}
\int f \circ F_0 d\nu_i \\
\int f \circ F_1 d\nu_i \\
\int f \circ F_2 d\nu_i
\end{pmatrix}
\end{equation}
That these matrices satisfy both \eqref{eq:5.6} and \eqref{olduse} may be a simple coincidence, arising from the fact that $(M_i)_{i, j} = (M_i)_{j, i}.$ whenever $i \neq j$.

\begin{theorem} \label{thm:5.3}
\begin{equation*} 
\begin{pmatrix}
\int_{F_w SG} f d\nu_0 \\
\int_{F_w SG} f d\nu_1 \\
\int_{F_w SG} f d\nu_2
\end{pmatrix} = M_w \begin{pmatrix}
\int f \circ F_w d\nu_0 \\
\int f \circ F_w d\nu_1 \\
\int f \circ F_w d\nu_2
\end{pmatrix}.
\end{equation*}

\end{theorem}

\begin{proof}
This theorem is proven by induction.

If $m=1$, the result follows from Theorem \ref{thm:5.2}.

If the result holds for all words of length $m$ and $w$ is a word of length $m+1$, let $w = w_{1} w'$ for some word $w'$ of length $m$ and some $w_1 \in \left\{0, 1, 2\right\}$, $w=w_1 w'$. By the same argument as the one used in the proof of Theorem \ref{thm:5.2} (that of appealing to the definition of an energy measure and then using Theorem \ref{thm:5.1}), 

\begin{equation} \label{eq:5.20}
\begin{pmatrix}
\int_{F_w SG} f d\nu_0 \\
\int_{F_w SG} f d\nu_1 \\
\int_{F_w SG} f d\nu_2
\end{pmatrix} = M_{w_1} \begin{pmatrix}
\int_{F_{w'}SG}  f \circ F_{w_1} d\nu_0 \\
\int_{F_{w'}SG}  f \circ F_{w_1} d\nu_1 \\
\int_{F_{w'}SG}  f \circ F_{w_1} d\nu_2
\end{pmatrix}.
\end{equation}
By the inductive hypothesis,

\begin{equation} \label{eq:5.21}
\begin{pmatrix}
\int_{F_{w'} SG} f \circ F_{w_1} d\nu_0 \\
\int_{F_{w'} SG} f \circ F_{w_1} d\nu_1 \\
\int_{F_{w'} SG} f \circ F_{w_1} d\nu_2
\end{pmatrix} = M_{w'} \begin{pmatrix}
\int f \circ F_{w_1} \circ F_{w'} d\nu_0 \\
\int f \circ F_{w_1} \circ F_{w'} d\nu_1 \\
\int f \circ F_{w_1} \circ F_{w'} d\nu_2
\end{pmatrix}.
\end{equation}
$M_{w_1}M_{w'} = M_w$ and $f \circ F_{w_1} \circ F_{w'} = f \circ F_w$, so by \eqref{eq:5.20} and \eqref{eq:5.21}:

\begin{equation*}
\begin{pmatrix}
\int_{F_w SG} f d\nu_0 \\
\int_{F_w SG} f d\nu_1 \\
\int_{F_w SG} f d\nu_2
\end{pmatrix} = M_w \begin{pmatrix}
\int f \circ F_w d\nu_0 \\
\int f \circ F_w d\nu_1 \\
\int f \circ F_w d\nu_2
\end{pmatrix}.
\end{equation*}

\end{proof}

Because every harmonic function is a linear combination of $\left\{h_0, h_1, h_2\right\}$, and Theorem \ref{thm:5.1} constitutes a formula to express every harmonic energy measure as a linear combination of $\left\{\nu_0, \nu_1, \nu_2\right\}$, the problem of taking the integral with respect to a harmonic energy measure $\nu$ is solved by calculating the integrals $\int h_i d\nu_j$ and taking linear combinations with the appropriate coefficients.

\begin{theorem} \label{thm:5.4}
For $i, j \in \left\{0, 1, 2\right\}$:

\begin{equation}
\int h_i d\nu_j = \left\{ \begin{matrix}
1 & \mbox{if $i=j$} \\
\frac{1}{2} & \mbox{if $i \neq j$}
\end{matrix} \right.
\end{equation}

\end{theorem} 

\begin{proof}
Let $\alpha = \int h_0 d\nu_0$ and $\beta = \int h_1 d\nu_0$. By symmetry, $\int h_i d\nu_i = \alpha$ for all $i$ and $\int h_i d\nu_j = \beta$ whenever $i \neq j$.
From the integral of a constant function, $\alpha+2\beta=2$.
By \eqref{olduse},

\begin{align*}
\begin{pmatrix}
\int_{F_0 SG} h_0 d\nu_0 \\
\int_{F_1 SG} h_0 d\nu_0 \\
\int_{F_2 SG} h_0 d\nu_0
\end{pmatrix} & = \frac{1}{15} \begin{pmatrix}
9 & 0 & 0 \\
2 & 2 & -1 \\
2 & -1 & 2
\end{pmatrix} \begin{pmatrix}
\int f \circ F_0 d\nu_0 \\
\int f \circ F_1 d\nu_0 \\
\int f \circ F_2 d\nu_0
\end{pmatrix}. \\
\end{align*}
Thus

\begin{align*}
\int h_{0} d\nu_{0} = \sum_{i} \int_{F_{i}SG} h_{0} d\nu_{0} = & \frac{3}{5} \int h_{0} + \frac{2}{5} h_{1} + \frac{2}{5} h_{2} d\nu_{0} \\
+ & \frac{2}{15} \int \frac{2}{5} h_{0} + \frac{1}{5} h_{2} d\nu_{0} \\
+ & \frac{2}{15} \int \frac{2}{5} h_{0} + \frac{1}{5} h_{2} d\nu_{1} \\
- & \frac{1}{15} \int \frac{2}{5} h_{0} + \frac{1}{5} h_{2} d\nu_{2} \\
+ & \frac{2}{15} \int \frac{2}{5} h_{0} + \frac{1}{5} h_{1} d\nu_{0} \\
- & \frac{1}{15} \int \frac{2}{5} h_{0} + \frac{1}{5} h_{1} d\nu_{1} \\
+ & \frac{2}{15} \int \frac{2}{5} h_{0} + \frac{1}{5} h_{1} d\nu_{2},
\end{align*}
so

\[ 75 \alpha = 9(5 \alpha + 4 \beta) + 2 (2 \alpha + \beta) + 2(3\beta) - (\alpha + 2\beta) + 2 (2\alpha+\beta)-(\alpha+2\beta)+2(3\beta), \]
\[ 75 \alpha = 51\alpha + 48\beta, \]
and

\[ \alpha = 2\beta \]
Because $\alpha+2\beta=2$ and $\alpha=2\beta$, $\alpha=1$ and $\beta=\frac{1}{2}$. In other words,
\begin{equation} \label{alphabeta}
    \int h_i d\nu_{jk} = \left\{ \begin{matrix}
        1 & : \mbox{$i=j$}\\
        \sfrac{1}{2} & : \mbox{$i \neq j$}
    \end{matrix} \right..
\end{equation}

\end{proof}

We can now, in principle, compute $\int u d\nu$ for any harmonic spline $u$ and harmonic energy measure $\nu$. In particular, if $E$ is a finite subset of $V_*$, and for all $x\in E$ we let $p(x)=\int v_x d\nu$ (where $v_x$ is the function that is harmonic away from $E$ with $V_x|_E = \delta_x$), then any function that satisfies the conditions of Theorem \ref{thm:2.6} can be numerically integrated using the weights $\{p(x)\}$.

The results proven in this section so far relating to the Sierpi\'{n}ski gasket generalize to any self-similar p.c.f. fractal generated by a finite iterated system $\left\{F_j\right\}$ satisfying the conditions of section 4, if the results are expressed using $\left\{\nu_{h_i, h_j}\right\}_{0\leq i<j<N_0}$ (instead of $ \{\nu_i\}_{0\leq i<N_0} $) as the spanning set for the set of harmonic energy measures. This choice of spanning set may seem odd or unnatural, because the measures $\nu$ that we are interested in are non-negative (for example, the Kusuoka measure $\nu = \sum_i \nu_i$), as are the measures $\nu_i$, while the measures $\nu_{h_i, h_j}$ ($i\neq j$) are signed; the Theorems of section \ref{generic} require that $\nu$ be non-negative. However, it is possible to express measures $\nu_i$ as linear combinations of $\{\nu_{h_i, h_j}\}_{0\leq i<j<N_0}$, and the reverse is not true. Therefore, though our choice of spanning set may be less natural, it is necessary to generalize this section's results.

Let K be defined as in section 4, with $V_0 = \left\{ q_i \right\}_{0 \leq i < N_0}$ and the harmonic functions $\left\{ h_i \right\}$ such that $\left( \sum_{i=0}^{N_0-1} a_i h_i \right) (q_j) = a_j$.

If $0 \leq i, j < N_0$, denote $\nu_{h_i, h_j}$ by $\nu_{ij}$.

\begin{theorem} \label{thm:5.5}

If $\nu$ is a harmonic energy measure on $K$ (that is, $\nu = \nu_{h, H}$ for some $h = \sum_i a_i h_i$ and $H = \sum_j b_j h_j$), then $\nu$ is a linear combination of $\left\{ \nu_{ij} \right\}_{i \neq j}$ given by

\begin{equation*} 
\nu = \sum_{0 \leq i<j<N_0} \left( a_i b_j + a_j b_i - a_i b_i - a_j b_j \right) \nu_{ij}.
\end{equation*}

\end{theorem}

\begin{proof}

By the additivity and scalar multiplication of energy measures, 

\begin{equation*}
\nu = \sum_{i=0}^{N_0-1} \sum_{j=0}^{N_0-1} a_i b_j \nu_{ij}.
\end{equation*}
For each $i$, $\sum_{j=0}^{N_0-1} \nu_{ij} = \nu_{h_i, \left(h_0+h_1+...+h_{n-1}\right)} = \nu_{h_i, 1} = 0$, so $\nu_{ii} = - \sum_{j \neq i} \nu_{ij}$.

Therefore, 

\begin{equation*}
\nu = \sum_{0\leq i<j <N_0} \left( a_i b_j + a_j b_i - a_i b_i - a_j b_j \right) \nu_{ij}.
\end{equation*}

\end{proof}

For the next theorems, we will speak of matrices whose rows and columns are indexed by pairs $(j, k)$ such that $0 \leq j < k < N_0$, ordered lexicographically, so $(j, k)$ comes \textquotedblleft before" $(l, m)$ if $i<l$ or $i=l$ and $k<m$.

We will refer to the \textquotedblleft $(j, k)$-th row" or \textquotedblleft $(l, m)$-th column or \textquotedblleft $((j, k), (l, m))$-th entry" of such a matrix.

We will also define constants $a_{i, j, k, l, m}$ and matrices $M_i$ and $M_w$ as follows:

\begin{definition}

If $i \in \{0, 1, 2, \dots, N-1\}$, $0 \leq j<k<N_0$, and $0 \leq l<m<N_0$, let

\begin{align*}
a_{i, j, k, l, m} =  & (h_j \circ F_i)(q_l) \cdot (h_k \circ F_i)(q_m) \\
+& (h_j \circ F_i)(q_m) \cdot (h_k \circ F_i)(q_l) \\
-& (h_j \circ F_i)(q_l) \cdot (h_k \circ F_i)(q_l) \\
-& (h_j \circ F_i)(q_m) \cdot (h_k \circ F_i)(q_m).
\end{align*}

\bigskip

For all $i$, let $M_i$ be the matrix with rows and columns indexed by $ \left\{ (j, k) \right\}_{0\leq j<k<N_0} $ whose $((j, k), (l, m))$-th entry is $r_i^{-1} a_{i, j, k, l, m}$.

\bigskip

For all words $w = w_1 w_2 ... w_m$, let $M_w = M_{w_1} M_{w_2} ... M_{w_m}$.

\end{definition}

\begin{theorem} \label{thm:5.7}

For all words $w$ and continuous functions $f$,

\begin{equation} \label{eq:5.25}
\begin{pmatrix}
\int_{F_w K} f d\nu_{01} \\
\vdots \\
\int_{F_w K} f d\nu_{(N_0-2)(N_0-1)}
\end{pmatrix} = M_w \begin{pmatrix}
\int f \circ F_w d\nu_{01} \\
\vdots \\
\int f \circ F_w d\nu_{(N_0-2)(N_0-1)}
\end{pmatrix}.
\end{equation}

\end{theorem}

\begin{proof}

First, suppose $w$ is a word of length $1$, whose one character is $i$. For all $0 \leq j<k<N_0$, by the definition of an energy measure,

\begin{equation} \label{eq:5.23}
\int_{F_i K} f d\nu_{jk} = r_i^{-1} \int f \circ F_i d\nu_{h_j \circ F_i, h_k \circ F_i}.
\end{equation}
By applying Theorem \ref{thm:5.5} to $\nu_{h_j \circ F_i, h_k \circ F_i}$,

\begin{equation} \label{eq:5.24}
\nu_{h_j \circ F_i, h_k \circ F_i} = \sum_{0 \leq l<m<N_0} a_{i, j, k, l, m} \nu_{lm}.
\end{equation}
By \eqref{eq:5.23} and \eqref{eq:5.24}, 

\begin{equation*}
\int_{F_i K} f d\nu_{jk} = r_i^{-1} \sum_{0 \leq l<m<N_0} a_{i, j, k, l, m} \int f \circ F_i d\nu_{lm}.
\end{equation*}
This statement for all $(j, k)$ is equivalent to \eqref{eq:5.25} for $w=i$. This theorem extends to longer words by the same argument as used in the proof of Theorem \ref{thm:5.3}.

\end{proof}

To compute weights for Theorem \ref{thm:2.6}, all that is left to do is evaluate the integrals of the form $\int h_i d\nu_{jk}$ for $0\leq j < k < N_0$. For all $i, j, k$:

\begin{equation*}
\int h_i d\nu_{jk} = \sum_{l=0}^{N-1} \int_{F_l K} h_i d\nu_{jk}.
\end{equation*}
By applying Theorem \ref{thm:5.7}, each integral $\int_{F_l K} h_i d\nu_{jk}$ can be expressed as a linear combination of $\left\{ \int h_i d\nu_{jk} \right\}_{0\leq j<k<N_0}$. Doing this for all $i, j, k$ yields a system of $\frac{1}{2}(N_0^3-N_0^2)$ equations and $\frac{1}{2}(N_0^3-N_0^2)$ unknowns for the integrals $\int h_i d\nu_{jk}$. This system is linearly dependent because setting every integral equal to $0$ would be one solution. However, combining this system with the equations

\begin{equation*}
    \sum_{i=0}^{N_0-1} \int h_i d\nu_{jk} = \nu_{jk}(K)=\energy(h_j, h_k)
\end{equation*}
for all $(j, k)$ will in most cases make it independent (it does in all of our examples). It is possible that this system will have an un unwieldy amount of terms. Likely, symmetry can be used to reduce it to a more manageable system.

We now choose some specific self-similar p.c.f. fractals and list the results obtained when the above calculations are performed. For each fractal, these calculations determine the matrices $M_i$. The fractals chosen are the Unit Interval, the Sierpi\'{n}ski gasket, the Sierpi\'{n}ski tetrahedron, the Sierpi\'{n}ski $n$-hedron for a general $n$, and the $3$-level gasket. (The Sierpi\'{n}ski $n$-hedron is generated by the similarities with contraction ratio $\frac{1}{2}$ whose fixed points are $n$ pairwise equidistant vertices in $\mathbb{R}^{n-1}$.)

Note that the matrices $M_0, M_1, M_2$ for the Sierpi\'{n}ski gasket are not the same as the ones given by Theorem \ref{thm:5.2}, because of our change in choice of spanning set: the matrices in Theorem \ref{thm:5.2} satisfy
\begin{equation*}
    \begin{pmatrix}
        \int_{F_i SG} f d\nu_0 \\
        \int_{F_i SG} f d\nu_1 \\
        \int_{F_i SG} d\nu_2
    \end{pmatrix}
    = M_i \begin{pmatrix}
        \int f \circ F_i d\nu_0 \\
        \int f \circ F_i d\nu_1 \\
        \int f \circ F_i d\nu_2
    \end{pmatrix},
\end{equation*}
while the matrices in this table satisfy
\begin{equation*}
    \begin{pmatrix}
        \int_{F_i SG} f d\nu_{01} \\
        \int_{F_i SG} f d\nu_{02} \\
        \int_{F_i SG} d\nu_{12}
    \end{pmatrix}
    = M_i \begin{pmatrix}
        \int f \circ F_i d\nu_{01} \\
        \int f \circ F_i d\nu_{02} \\
        \int f \circ F_i d\nu_{12}
    \end{pmatrix}.
\end{equation*}
%For the Pentagasket, the entries of the matrices are not rational (see the construction of energy on the Pentagasket in \cite{pentagasket}, so it is more visually meaningful to express them in decimal. Also, the Pentagasket matrices are $10 \times 10$, so we only list $M_0$. The other matrices $M_1, M_2, M_3, M_4$ can be easily obtained from $M_0$ by a symmetry argument.

\begin{center}
\begin{longtable}{| l | c | c | c |}
\hline
    \textbf{Fractal} & \textbf{Picture} & \textbf{Matrices}\\
    \hline
    Unit Interval ($I$) &
    
    \begin{centering}
    \begin{tikzpicture}
        \usetikzlibrary{decorations.pathreplacing}
        \centering
        \draw[cyan] (0,0) -- (1,0); 
        \draw[red] (1,0) -- (2,0);
        \filldraw[black] (0,0) circle (1 pt) node[anchor=east]{$q_0$};
        \filldraw[black] (2,0) circle (1pt) node[anchor=west]{$q_1$};
        \filldraw[black] (1,0) circle (1pt);
        
        \draw[decorate, decoration={brace, amplitude=3pt, mirror}] (0,-.1)--(.95,-.1) node[black,midway,yshift=-0.3cm] {\tiny $F_0 I$};
        \draw[decorate,decoration={brace,amplitude=3pt, mirror}] (1.05,-.1)--(2,-.1) node[black,midway,yshift=-0.3cm] {\tiny $F_1 I$};
    \end{tikzpicture}
    \end{centering}

    &
    
    \begin{tabular}{l}
        $M_0=M_1=\frac{1}{2}$.\\
        
        \\
        For all $w$, for all continuous $f$,
        \\
        $\int_{F_w I} f d\nu_{01} = M_w \int f \circ F_w d\nu_{01}$.
        \\
         
    \end{tabular}
    
    \\
    \hline
    Sierpi\'{n}ski gasket ($SG$) &
    
    \begin{centering}
    \begin{tikzpicture}[scale=.35]
        \centering
        \draw[black,fill=cyan] (0, 5.196) -- (-1.5, 2.598) -- (1.5, 2.598) -- cycle;
        \draw[black,fill=red] (-3, 0) -- (0, 0) -- (-1.5, 2.598) -- cycle;
        \draw[black,fill=yellow] (3, 0) -- (0, 0) -- (1.5, 2.598) -- cycle;
        
        \filldraw[black] (0, 5.196) circle (1 pt) node[anchor=south]{$q_0$};
        \filldraw[black] (-3, 0) circle (1 pt) node[anchor=north]{$q_1$};
        \filldraw[black] (3,0) circle (1 pt) node[anchor=north]{$q_2$};
    \end{tikzpicture}
    \end{centering}
    
    &
    
    \begin{tabular}{l}
        $M_0 = \frac{1}{15} \begin{pmatrix}
        6 & 3 & 0\\
        3 & 6 & 0\\
        -2 & -2 & 1
        \end{pmatrix}$.
        \\
        $M_1=\frac{1}{15} \begin{pmatrix}
        6 & 0 & 3\\
        -2 & 1 & -2\\
        3 & 0 & 6
        \end{pmatrix}$.
        \\
        $M_2 = \frac{1}{15} \begin{pmatrix}
        -2 & -2 & 1\\
        0 & 6 & 3\\
        0 & 3 & 6
        \end{pmatrix}$.\\
        
        \\
        For all $w$, for all continuous $f$,\\
        $\begin{pmatrix}
        \int_{F_w SG} f d\nu_{01}\\
        \int_{F_w SG} f d\nu_{02}\\
        \int_{F_w SG} f d\nu_{12} \end{pmatrix} = M_w \begin{pmatrix}
        \int f\circ F_w d\nu_{01}\\
        \int f \circ F_w d\nu_{02}\\
        \int f \circ F_w d\nu_{12} \end{pmatrix}$.
        \\
         
    \end{tabular}
    \\
    \hline
    Sierpi\'{n}ski tetrahedron ($ST$) &
    
    \begin{centering}
    \begin{tikzpicture}
        \centering
        \draw[black, fill=cyan] (0,2) -- (.5, 1.15) -- (0,1) -- (-.5, 1.15) -- cycle;
        \draw[black] (0,2) -- (0,1);
        
        \draw[black, fill=red] (-.5, 1.15) -- (-1, .3) -- (-.5, .15) -- (-.25, .575) -- cycle;
        \draw[black] (-.5, 1.15) -- (-.5, .15);
        
        \draw[black, fill=yellow] (0,1) -- (-.5, .15) -- (0,0) -- (.5, .15) -- cycle;
        \draw[black] (0,0) -- (0,1);
        
        \draw[black, fill=green] (.5, 1.15) -- (1, .3) -- (.5, .15) -- (.25, .575) -- cycle;
        \draw[black] (.5, 1.15) -- (.5, .15);
        
        \filldraw[black] (0,2) circle (1 pt) node[anchor=south]{$q_0$};
        \filldraw[black] (-1,.3) circle (1 pt) node[anchor=east]{$q_1$};
        \filldraw[black] (0,0) circle (1 pt) node[anchor=north]{$q_2$};
        \filldraw[black] (1,.3) circle (1 pt) node[anchor=west]{$q_3$};
    \end{tikzpicture}
    \end{centering} 
    
    &
    
    \begin{tabular}{l}
        $M_0 = \frac{1}{24} \begin{pmatrix}
        8 & 4 & 4 & 0 & 0 & 0\\
        4 & 8 & 4 & 0 & 0 & 0\\
        4 & 4 & 8 & 0 & 0 & 0\\
        -2 & -2 & -1 & 1 & 0 & 0\\
        -2 & -1 & -2 & 0 & 1 & 0\\
        -1 & -2 & -2 & 0 & 0 & 1 \end{pmatrix}$.
        \\
        $M_1= \frac{1}{24} \begin{pmatrix}
        8 & 0 & 0 & 4 & 4 & 0 \\
        -2 & 1 & 0 & -2 & -1 & 0 \\
        -2 & 0 & 1 & -1 & -2 & 0 \\
        4 & 0 & 0 & 8 & 4 & 0 \\
        4 & 0 & 0 & 4 & 8 & 0 \\
        -1 & 0 & 0 & -2 & -2 & 1 \end{pmatrix}$.
        \\
        $M_2 = \frac{1}{24} \begin{pmatrix}
        1 & -2 & 0 & -2 & 0 & 1\\
        0 & 8 & 0 & 4 & 0 & 4\\
        0 & -2 & 1 & -1 & 0 & -2\\
        0 & -2 & 1 & -1 & 0 & -2\\
        0 & 4 & 0 & 8 & 0 & 4\\
        0 & 4 & 0 & 4 & 0 & 8 \end{pmatrix}$.
        \\
        $M_3 = \frac{1}{24} \begin{pmatrix}
        1 & 0 & -2 & 0 & -2 & -1\\
        0 & 1 & -2 & 0 & -1 & -2\\
        0 & 0 & 8 & 0 & 4 & 4\\
        0 & 0 & -1 & 1 & -2 & -2\\
        0 & 0 & 4 & 0 & 8 & 4\\
        0 & 0 & 4 & 0 & 4 & 8 \end{pmatrix}$.
        \\
        
        \\
        For all $w$, for all continuous $f$,
        \\
        $\begin{pmatrix}
        \int_{F_w ST} f d\nu_{01} \\
        \int_{F_w ST} f d\nu_{02} \\
        \int_{F_w ST} f d\nu_{03} \\
        \int_{F_w ST} f d\nu_{12} \\
        \int_{F_w ST} f d\nu_{13} \\
        \int_{F_w ST} f d\nu_{23} \end{pmatrix} = M_w \begin{pmatrix}
        \int f \circ F_w d\nu_{01} \\
        \int f \circ F_w d\nu_{02} \\
        \int f \circ F_w d\nu_{03} \\
        \int f \circ F_w d\nu_{12} \\
        \int f \circ F_w d\nu_{13} \\
        \int f \circ F_w d\nu_{23} \end{pmatrix}$.
        \\
         
    \end{tabular}
    
    \\
    \hline

    Sierpi\'{n}ski $n$-hedron
    &
    
    &
    \begin{tabular}{l}
        $(M_i)_{(j,k),(l,m)} = \frac{1}{n(n+2)} \cdot$\\
        $\big[\xi(j,i,l)\xi(k,i,m) + \xi(j,i,m)\xi(k,i,l)$\\
        $- \xi(j,i,l)\xi(k,i,l) - \xi(j,i,m)\xi(k,i,m)\big]$,\\
        where\\
        $\xi(a, b, c)=\left\{ \begin{matrix}
        n+2 & : \mbox{$a=b=c$}\\
        2 & : \mbox{$a=b \neq c$}\\
        2 & : \mbox{$a=c \neq b$}\\
        0 & : \mbox{$a \neq b=c$}\\
        1 & : \mbox{$a\neq b$, $b\neq c$, $a \neq c$}
        \end{matrix}\right.$.
    \end{tabular}
    
    \\
    \hline
    
    $3$-level gasket ($SG_3$)
    &
    
    \begin{centering}
    \begin{tikzpicture}[scale=.3]
    
    \draw[black, fill=cyan] (0, 5.192) -- (-1.5, 2.596) -- (1.5, 2.596) -- cycle;
    \draw[black, fill=green] (-1.5, 2.596) -- (-3, 0) -- (0, 0) -- cycle;
    \draw[black, fill=orange] (1.5, 2.596) -- (0, 0) -- (3, 0) -- cycle;
    \draw[black, fill=red] (-3, 0) -- (-4.5, -2.596) -- (-1.5, -2.596) -- cycle;
    \draw[black, fill=violet] (0, 0) -- (-1.5, -2.596) -- (1.5, -2.596) -- cycle;
    \draw[black, fill=yellow] (3, 0) -- (1.5, -2.596) -- (4.5, -2.596) -- cycle;
    
    \filldraw[black] (0, 5.192) circle (1 pt) node[anchor=south]{$q_0$};
    \filldraw[black] (-4.5, -2.596) circle (1 pt) node[anchor=north]{$q_1$};
    \filldraw[black] (4.5, -2.596) circle (1 pt) node[anchor=north]{$q_2$};
    
     \node [fill=white,rounded corners=2pt,inner sep=1pt] at (0, 3.461){\tiny $F_0 K$};
    \node [fill=white,rounded corners=2pt,inner sep=1pt] at (-3, -1.731){\tiny $F_1 K$};
    \node [fill=white,rounded corners=2pt,inner sep=1pt] at (3, -1.731){\tiny $F_2 K$};
    \node [fill=white,rounded corners=2pt,inner sep=1pt] at (-1.5, .865){\tiny $F_{(01)} K$};
    \node [fill=white,rounded corners=2pt,inner sep=1pt] at (1.5, .865){\tiny $F_{(02)} K$};
    \node [fill=white,rounded corners=2pt,inner sep=1pt] at (0, -1.731){\tiny $F_{(12)} K$};

    \end{tikzpicture}
    \end{centering}
    
    &
    \begin{tabular}{l}
        $M_0 = \frac{1}{105} \begin{pmatrix}
            28 & 7 & 0\\
            7 & 28 & 0\\
            -12 & -12 & 1
        \end{pmatrix}$.\\
        $M_1 = \frac{1}{105} \begin{pmatrix}
            28 & 0 & 7\\
            -12 & 1 & -12\\
            7 & 0 & 28
        \end{pmatrix}$.\\
        $M_2 = \frac{1}{105} \begin{pmatrix}
            1 & -12 & -12\\
            0 & 28 & 7\\
            0 & 7 & 28
        \end{pmatrix}$.\\
        $M_{(01)} = \frac{1}{105} \begin{pmatrix}
            16 & 3 & 3\\
            0 & 6 & -2\\
            0 & -2 & 6
        \end{pmatrix}$.\\
        $M_{(02)} = \frac{1}{105} \begin{pmatrix}
            6 & 0 & -2\\
            3 & 16 & 3\\
            -2 & 0 & 6
        \end{pmatrix}$.\\
        $M_{(12)} = \frac{1}{105} \begin{pmatrix}
            6 & -2 & 0\\
            -2 & 6 & 0\\
            3 & 3 & 16
        \end{pmatrix}$.\\
        For all $w$, for all continuous $f$,\\
        $ \begin{pmatrix}
            \int_{F_w SG_3} f d\nu_{01}\\
            \int_{F_w SG_3} f d\nu_{02}\\
            \int_{F_w SG_3} f d\nu_{12}
        \end{pmatrix} = M_w \begin{pmatrix}
            \int f \circ F_w d\nu_{01}\\
            \int f \circ F_w d\nu_{02}\\
            \int f \circ F_w d\nu_{12}
        \end{pmatrix}$.
    \end{tabular}
    \\
    \hline
\end{longtable}
\end{center}

By the method that we used to reach \eqref{alphabeta}, we can compute the integrals of the form $\int h_i d\nu_{jk}$. In the table below, we list these integrals for the same fractals as in the above table. From these, since each measure of the form $\nu_{i}$ is equal to $-\sum_{j\neq i} \nu_{ij}$, we can calculate integrals of the form $\int h_i d\nu_{j}$:
\begin{equation*}
    \int h_i d\nu_{j}=-\sum_{k\neq j} \int h_i d\nu_{jk}.
\end{equation*}
These basic integrals for our example fractals ($I$, $SG$, $ST$, the $n$-hedron for $3\leq n\leq 100$, and $SG_3$) are listed in the table below. For $I$, $SG$, and $ST$, the integrals were calculated by hand, while for the $n$-hedrons and the $3$-level gasket, they were calculated with the assistance of a computer program. We hypothesize that the formulae for the $n$-hedron continue to hold for all positive integers $n \geq 3$.

\begin{longtable}{| l | c | c |}
\hline
    \textbf{Fractal} & \textbf{$\int h_i d\nu_{jk}$} & \textbf{$\int h_i d\nu_{j}$}\\ \hline
    Unit Interval & $-\frac{1}{2}$ & $\frac{1}{2}$ \\ \hline
    Sierpi\'{n}ski gasket & $\left\{ \begin{matrix} -\sfrac{1}{2} & : \mbox{$i=j$ or $i=k$} \\ 0 & : \mbox{$i$, $j$, $k$ distinct}\end{matrix}\right.$ & $\left\{ \begin{matrix} 1 & : \mbox{$i=j$} \\ \sfrac{1}{2} & : \mbox{$i\neq j$}\end{matrix}\right.$ \\ \hline
    Sierpi\'{n}ski tetrahedron & 
    $\left\{ \begin{matrix} -\sfrac{1}{2} & : \mbox{$i=j$ or $i=k$} \\ 0 & : \mbox{$i$, $j$, $k$ distinct}\end{matrix}\right.$ &
    $\left\{ \begin{matrix} \sfrac{3}{2} & : \mbox{$i=j$} \\ \sfrac{1}{2} & : \mbox{$i\neq j$} \end{matrix} \right.$ \\ \hline
    $n$-hedron ($3\leq n\leq 100 $) & $\left\{ \begin{matrix} -\sfrac{1}{2} & : \mbox{$i=j$ or $i=k$} \\ 0 & : \mbox{$i$, $j$, $k$ distinct}\end{matrix}\right.$ & $\left\{ \begin{matrix} \frac{1}{2}(n-1) & : \mbox{$i=j$} \\ \sfrac{1}{2} & : \mbox{$i\neq j$}\end{matrix}\right.$ \\ \hline
    \begin{comment}
    Pentagasket & 
    $\left\{ \begin{matrix}
        -0.17 & : \mbox{$i, j, k$ as in Figure 5.1a} \\
        -0.01 & : \mbox{$i, j, k$ as in Figure 5.1b} \\
        -0.004 & : \mbox{$i, j, k$ as in Figure 5.1c} \\
        -0.07 & : \mbox{$i, j, k$ as in Figure 5.1d} \\
        0.005 & : \mbox{$i, j, k$ as in Figure 5.1e} \\
        -0.003 & : \mbox{$i, j, k$ as in Figure 5.1f}
    \end{matrix}\right.$ & 
    $\left\{ \begin{matrix}
        0.47 & : \mbox{$i, j$ as in Figure 5.2a} \\
        0.18 & : \mbox{$i, j$ as in Figure 5.2b} \\
        0.09 & : \mbox{$i, j$ as in Figure 5.2c}
    \end{matrix}\right.$\\
    \hline
    \end{comment}
    $3$-level gasket &
    $\left\{ \begin{matrix}
        -\sfrac{1}{2} & : \mbox{$i=j$ or $i=k$} \\
        0 & : \mbox{$i, j, k$ distinct}
    \end{matrix}\right.$ &
    $\left\{ \begin{matrix}
        1 & : \mbox{$i=j$} \\
        \sfrac{1}{2} & : \mbox{$i \neq j$}
    \end{matrix}\right.$\\
    \hline
\end{longtable}

As with $SG$, we can numerically integrate any function that satisfies the conditions of Theorem \ref{thm:2.6} with respect to a non-negative harmonic energy measure $\nu$: For some finite $E \subseteq V_*$, let $\{v_x\}_{x \in E}$ be the usual indicator splines, and let the weights be $p(x)=\int v_x d\nu$.
%clearpage


\begin{thebibliography}{12}
 
%\bibitem{pentagasket} Bryant Adams, S. Alex Smith, Robert S. Strichartz, and Alexander Teplyaev, \textit{The spectrum of the Laplacian on the pentagasket}, Fractals in Graz 2001. Trends Math., Birkh{\"a}user, Basel (2003), 1-24.% MR 2091699 (2006g:28017)

\bibitem{energy} Renee Bell, Ching-Wei Ho, and Robert S. Strichartz, \textit{Energy measures of harmonic functions on the Sierpi\'{n}ski gasket.} Indiana Univ. Math. J., \textbf{63} (2014), 831-868.

\bibitem{ka1} Naotako Kajino, \textit{Heat kernel asymptotics for the measurable Riemann structure on the Sierpi\'{n}ski gasket}. Potential Anal., \textbf{36} (2012), 67-115.

\bibitem{kigami} Jun Kigami, \textit{Analysis on Fractals}. Camb. Univ. Press (2001).

\bibitem{ki2} Jun Kigami, \textit{Measurable Riemannian geometry on the Sierpi\'{n}ski gasket: the Kusuoka measure and the Gaussian heat kernel estimate}. Math. Ann., \textbf{340}(4) (2008), 781-804.

\bibitem{Ko} Alex Kontorovich, \textit{From Apollonius to Zaremba: Local-global phenomena in thin orbits}. Bull. Amer. Math. Soc., \textbf{50}(2) (2013), 187-228.

\bibitem{kusuoka} Shigeo Kusuoka, \textit{Dirichlet forms on fractals and products of random matrices}. Publ. Res. Inst. Math. Sci., \textbf{25} (1989), 659-680.

\bibitem{website} Jens Malmquist and Robert S. Strichartz, \textit{Numerical integration for fractal measures - programs}. \url{www.math.cornell.edu/~jensmalmquist}.

\bibitem{Nie} Harald Niederreiter, \textit{Quasi-Monte Carlo methods and pseudo-random numbers}. Bull. Amer. Math. Soc., \textbf{84}(6) (1978), 957-1041.

\bibitem{LRSU} Pak-Hin Li, Nicholas Ryder, Robert S. Strichartz, and Baris Evren Ugurcan, \textit{Extensions and their minimizations on the Sierpinski Gasket}. Potential Anal., \textbf{41} (2014), 1167-1201.

\bibitem{S} Robert S. Strichartz, \textit{Differential Equations on Fractals: A Tutorial.} Princeton Univ. Press (2006).

\bibitem{StTs} Robert S. Strichartz and Tsu Tong Tse, \textit{Local behavior of smooth functions for the energy Laplacian on the Sierpi\'{n}ski gasket}. Analysis, Munich, \textbf{30}(3) (2000), 285-299.
 
\bibitem{splines} Robert S. Strichartz and Michael Usher, \textit{Splines on fractals.} Math. Prob. Camb. Phil. Soc., \textbf{120} (2000).

\end{thebibliography}
\end{document}